\documentclass[a4paper,twoside,11pt]{article}
\usepackage[utf8]{inputenc}
\usepackage[english]{babel}
\usepackage{amssymb, amsfonts, amsthm, amsmath}
\usepackage[hidelinks]{hyperref}
\usepackage{pgf,tikz} % for drawing vector graphics inside latex
\usetikzlibrary{math,arrows,calc,through}
\usepackage[margin=2.5cm]{geometry} 
\usepackage[mathscr]{euscript}
\usepackage{graphicx,accents}
\usepackage{placeins}
\usepackage{caption} % so that hyperref in pdf navigation jumps to top of figures
\usepackage{mathtools} % so only referenced equations are numbered (with line below)
\usepackage{subcaption}
\usepackage{enumitem}
\usepackage{bm} % bold math symbols
\usepackage[]{changepage} %for the margins in the statement environment

\numberwithin{equation}{section}

\theoremstyle{plain}
\newtheorem{theorem}{Theorem}[section]
\newtheorem{corollary}[theorem]{Corollary}
\newtheorem{lemma}[theorem]{Lemma}

\theoremstyle{definition}
\newtheorem{definition}[theorem]{Definition}
\newtheorem{remark}[theorem]{Remark}

\theoremstyle{remark}

\newenvironment{statement}
{	\itshape
	\vspace{1mm}
	\begin{adjustwidth}{5mm}{5mm}
		\begin{center}
			\large
}
{ 
		\end{center}
	\end{adjustwidth}
	\vspace{1mm}
}

\newcommand{\N}{\mathbb{N}}
\newcommand{\Z}{\mathbb{Z}}
\newcommand{\R}{\mathbb{R}}

\newcommand{\RP}{{\mathbb{R}\mathrm{P}}}

\DeclareMathOperator{\cro}{cr} %cross-ratio
\newcommand{\pah}[1]{\raisebox{0pt}[0pt][0pt]{$#1^\vee_{\scalebox{0.6}[1]{$\bm{-}$}}$}}
\newcommand{\pav}[1]{#1^\vee_{\,\bm{\shortmid}}}
\newcommand{\pa}[1]{#1^\vee}

\newcommand{\inc}{\mathcal{C}}

\newcommand{\p}{\varphi}

\DeclareMathOperator{\North}{
	\begin{tikzpicture}[scale=1.4, line cap=round]
		\draw (0,0) -- (0,1ex) -- (1ex,1ex) -- (1ex,0) -- (0,0);
		\fill (0ex,0.5ex) rectangle ++(1ex,0.5ex);
	\end{tikzpicture}\kern-1pt
}
\DeclareMathOperator{\East}{
	\begin{tikzpicture}[scale=1.4, line cap=round]
		\draw (0,0) -- (0,1ex) -- (1ex,1ex) -- (1ex,0) -- (0,0);
		\fill (0.5ex,0ex) rectangle ++(0.5ex,1ex);
	\end{tikzpicture}\kern-1pt
}
\DeclareMathOperator{\South}{
	\begin{tikzpicture}[scale=1.4, line cap=round]
		\draw (0,0) -- (0,1ex) -- (1ex,1ex) -- (1ex,0) -- (0,0);
		\fill (0ex,0ex) rectangle ++(1ex,0.5ex);
	\end{tikzpicture}\kern-1pt
}
\DeclareMathOperator{\West}{
	\begin{tikzpicture}[scale=1.4, line cap=round]
		\draw (0,0) -- (0,1ex) -- (1ex,1ex) -- (1ex,0) -- (0,0);
		\fill (0ex,0ex) rectangle ++(0.5ex,1ex);
	\end{tikzpicture}\kern-1pt
}
\DeclareMathOperator{\NE}{
	\begin{tikzpicture}[scale=1.4, line cap=round]
		\draw (0,0) -- (0,1ex) -- (1ex,1ex) -- (1ex,0) -- (0,0);
		\fill (0.5ex,0.5ex) rectangle ++(0.5ex,0.5ex);
	\end{tikzpicture}\kern-1pt
}
\DeclareMathOperator{\SE}{
	\begin{tikzpicture}[scale=1.4, line cap=round]
		\draw (0,0) -- (0,1ex) -- (1ex,1ex) -- (1ex,0) -- (0,0);
		\fill (0.5ex,0) rectangle ++(0.5ex,0.5ex);
	\end{tikzpicture}\kern-1pt
}
\DeclareMathOperator{\SW}{
	\begin{tikzpicture}[scale=1.4, line cap=round]
		\draw (0,0) -- (0,1ex) -- (1ex,1ex) -- (1ex,0) -- (0,0);
		\fill (0ex,0ex) rectangle ++(0.5ex,0.5ex);
	\end{tikzpicture}\kern-1pt
}
\DeclareMathOperator{\NW}{
\begin{tikzpicture}[scale=1.4, line cap=round]
	\draw (0,0) -- (0,1ex) -- (1ex,1ex) -- (1ex,0) -- (0,0);
	\fill (0ex,0.5ex) rectangle ++(0.5ex,0.5ex);
\end{tikzpicture}\kern-1pt
}
\DeclareMathOperator{\WEmid}{
	\begin{tikzpicture}[scale=1.4, line cap=round]
		\draw (0,0) -- (0,1ex) -- (1ex,1ex) -- (1ex,0) -- (0,0);
		\fill (0.25ex,0ex) rectangle ++(0.5ex,1ex);
	\end{tikzpicture}\kern-1pt
}

\DeclareMathOperator{\NSmid}{
	\begin{tikzpicture}[scale=1.4, line cap=round]
		\draw (0,0) -- (0,1ex) -- (1ex,1ex) -- (1ex,0) -- (0,0);
		\fill (0ex,0.25ex) rectangle ++(1ex,0.5ex);
	\end{tikzpicture}\kern-1pt
}

\DeclareMathOperator{\WE}{
\begin{tikzpicture}[scale=1.4, line cap=round]
	\draw (0,0) -- (0,1ex);%
	\draw (0,1ex) -- (1ex,1ex);%
	\draw (1ex,1ex) -- (1ex,0);%
	\draw (1ex,0) -- (0,0);%
	\draw (0.5ex,0) -- (0.5ex,1ex);%
\end{tikzpicture}\kern-1pt
}

\DeclareMathOperator{\NS}{
\begin{tikzpicture}[scale=1.4, line cap=round]
	\draw (0,0) -- (0,1ex);%
	\draw (0,1ex) -- (1ex,1ex);%
	\draw (1ex,1ex) -- (1ex,0);%
	\draw (1ex,0) -- (0,0);%
	\draw (0ex,0.5ex) -- (1ex,0.5ex);%
\end{tikzpicture}\kern-1pt
}

\tikzstyle{bvert}=[draw,circle,fill=black,minimum size=5pt,inner sep=0pt]
\tikzstyle{blvert}=[draw,circle,fill=blue,draw=blue,minimum size=5pt,inner sep=0pt]
\tikzstyle{wvert}=[draw,circle,fill=white,minimum size=5pt,inner sep=0pt]
\tikzstyle{blwvert}=[draw,circle,draw=blue,fill=white,minimum size=5pt,inner sep=0pt]
\tikzset{circle through 3 points/.style n args={3}{%
insert path={let    \p1=($(#1)!0.5!(#2)$),
                    \p2=($(#1)!0.5!(#3)$),
                    \p3=($(#1)!0.5!(#2)!1!-90:(#2)$),
                    \p4=($(#1)!0.5!(#3)!1!90:(#3)$),
                    \p5=(intersection of \p1--\p3 and \p2--\p4)
                    in
                 node at (\p5) [draw,circle through= {(#1)}]{}}
}}

\title{Discrete Kœnigs nets, inscribed quadrics\\ and autoconjugate curves}
\author{Niklas Christoph Affolter
     \thanks{TU Wien, Institute of Discrete Mathematics and Geometry, Wiedener Hauptstr.~8-10/104, A-1040 Vienna, Austria.
       \textit{E-mail address}: \texttt{affolter@posteo.net}} \footnotemark[2]\ ,
      Alexander Yves Fairley
     \thanks{TU Berlin, Institute of Mathematics, Strasse des 17.~Juni 136, 10623 Berlin, Germany.
	\textit{E-mail address}: \texttt{fairley@tu-berlin.de}}
}

\date{October 30, 2025 \vspace{-5mm}}

\begin{document}%

\maketitle%

\begin{abstract}
	Discrete Kœnigs nets are a special class of discrete surfaces that play a fundamental role in discrete differential geometry, in particular in the study of discrete isothermic and minimal surfaces. Recently, it was shown by Bobenko and Fairley that Kœnigs nets can be characterized by the existence of touching inscribed conics. We generalize the touching inscribed conics by showing the existence of higher-dimensional inscribed quadrics for Kœnigs nets. Additionally, we study Kœnigs $d$-grids, which are Kœnigs nets with parameter lines that are contained in $d$-dimensional subspaces. We show that Kœnigs $d$-grids have a remarkable global property: there is a special inscribed quadric that all parameter spaces are tangent to. Finally, we establish a bijection between Kœnigs $d$-grids and pairs of discrete autoconjugate curves.
\end{abstract}

\setcounter{tocdepth}{1}
\tableofcontents

\newpage

\section{Introduction}

In differential geometry, conjugate nets are a special class of parametrized surfaces in $\RP^n$. Intuitively, as explained in \cite{Tzitzeica1924geometrie}, a conjugate net is a parametrized surface such that the coordinate curves dissect the surface into infinitesimal planar quads.
In \emph{discrete differential geometry} (see \cite{BS2008DDGbook} for an introduction), a \emph{Q-net} is a map $P\colon \Z^2 \rightarrow \RP^n$, such that the image of each unit square is contained in a plane \cite{sauer1970differenzen, dsmultidimconjugate, BS2008DDGbook}. Q-nets are also called \emph{discrete conjugate nets} since they constitute an integrable discretization of (smooth) conjugate nets.

In this paper, we study a fundamental reduction of Q-nets known as \emph{(discrete) Kœnigs nets} \cite{BS2009Koenigsnets}. Since we are interested in the discrete theory in this paper, in the following the term ``Kœnigs net'' always refers to the discrete notion. Similarly to the smooth theory, (discrete) Kœnigs nets have received much interest, due to their relation to discrete isothermic nets. Indeed, many discretizations of isothermic nets are Kœnigs nets, including the original \emph{(circular) isothermic nets} \cite{bpdisosurfaces,BS2009Koenigsnets}, \emph{S-isothermic nets} \cite{bpdisosurfaces,bhsminimal}, \emph{S-conical (isothermic) nets} \cite{bhsconical}, as well as the later generalization of S-isothermic nets in \cite{BS2008DDGbook}. Moreover, Kœnigs nets have also been studied in the context of \emph{infinitesimal deformations} \cite{sauer1933wackelige}, discretizations of curvature \cite{bpwcurvature}, \emph{confocal quadrics} \cite{bsstconfocal},   \emph{$\Omega$-nets} \cite{bchjpromega},  \emph{cone-nets} \cite{KMT2023conenets}, the \emph{resistor subvariety} \cite{agprvrc} and \emph{AGAG-webs} \cite{mpagag}. Another interesting aspect of Kœnigs nets is that they have multiple different characterizations by linear conditions \cite{BS2008DDGbook}. For instance, Kœnigs nets are Q-nets such that for each vertex, the intersection points of the diagonals of the four incident faces are coplanar.

Thus, it is somewhat unexpected that Kœnigs nets exhibit a number of non-linear phenomena. The first instance of such a phenomenon is due to Doliwa \cite{doliwa2003}. In his article, he investigated a special class of Q-nets that were later understood to be the diagonal intersection nets of Kœnigs nets \cite{BS2008DDGbook}. For these special Q-nets he showed that one may associate a certain \emph{Kœnigs conic} to each face, and that the existence of these Kœnigs conics is characteristic for these special Q-nets. The second instance came more recently, when Bobenko and Fairley \cite{bobenkofairley2021nets} showed that each Kœnigs net comes with a 1-parameter family of touching (inscribed) conics (see Figure~\ref{figure:touchingconics}) -- a new porism. They also showed that the existence of touching conics is characteristic for Kœnigs nets.

\begin{figure}[tb] 
	\centering
	\includegraphics[width=0.4\textwidth]{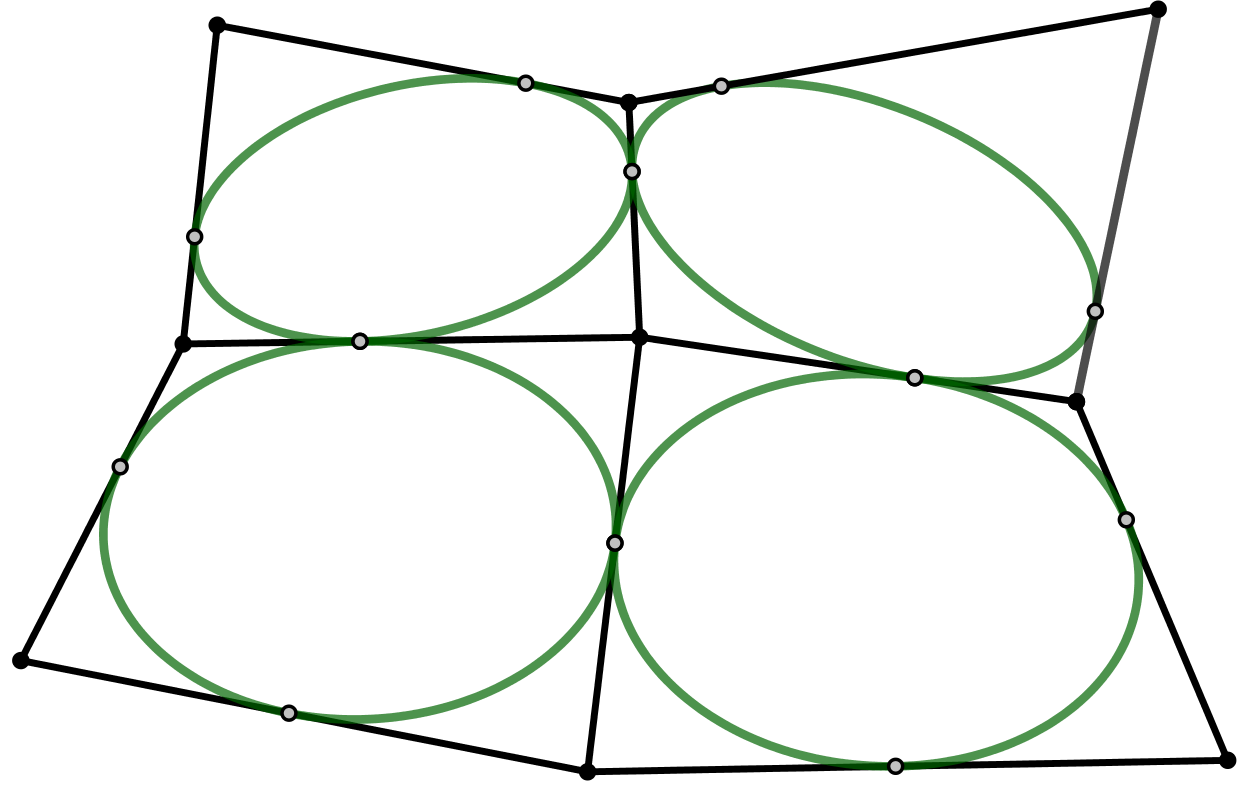}
    \hspace{10mm}
	\includegraphics[width=0.4\textwidth]{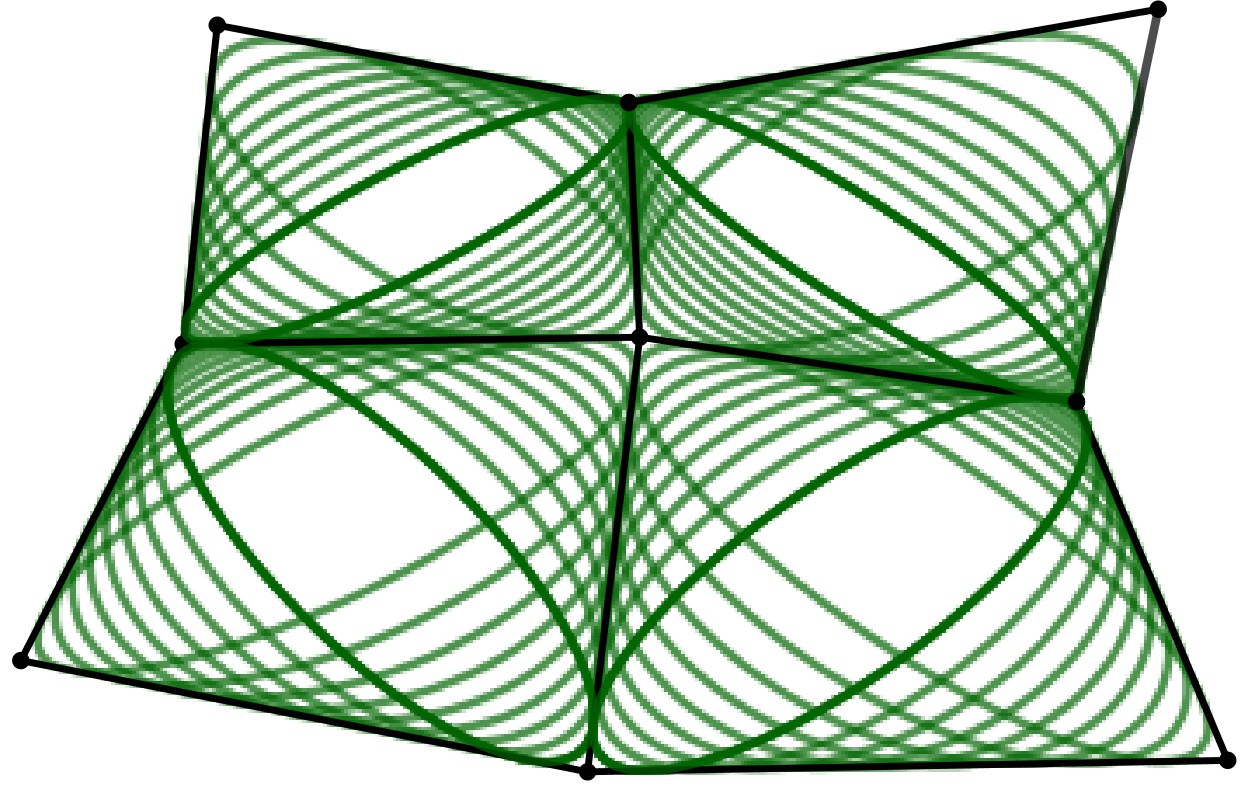}
	\caption{Any Kœnigs net has a $1$-parameter family of touching conics. The existence of one instance of touching conics (as shown on the left) is sufficient to ensure the existence of a $1$-parameter family of touching conics (as shown on the right).} 
	\label{figure:touchingconics}
\end{figure}

Our first contribution is a generalization of touching conics to what we call \emph{inscribed quadrics}. Consider a Kœnigs net $P$ defined on a rectangular subset of $\Z^2$. We call such a Kœnigs net \emph{extensive} if it satisfies a certain genericity condition. Briefly stated, $P$ is extensive if the dimension of the space joined by all the points of $P$ is as large as possible. For an instance of touching conics $\mathcal C$, we define an \emph{inscribed quadric} as a quadric that contains all the conics of $\mathcal C$ and is tangent to the parameter spaces of $P$.
With this definition, our first main result is as follows:
\begin{statement}
	Every extensive Kœnigs net has a 1-parameter family of inscribed quadrics. 
\end{statement}

See Figure~\ref{figure:twoinscribedquadrics} for two examples from the $1$-parameter family of inscribed quadrics for a K{\oe}nigs net in $\RP^3$. %In each case, the black lines are tangent to the inscribed quadric.
We also show that every inscribed quadric contains Doliwa's Kœnigs conics. Moreover, despite the extensivity condition, our first main result has implications for non-extensive cases, which we remark upon throughout the paper. 

\begin{figure}
	\centering
	\includegraphics[width=0.47\textwidth]{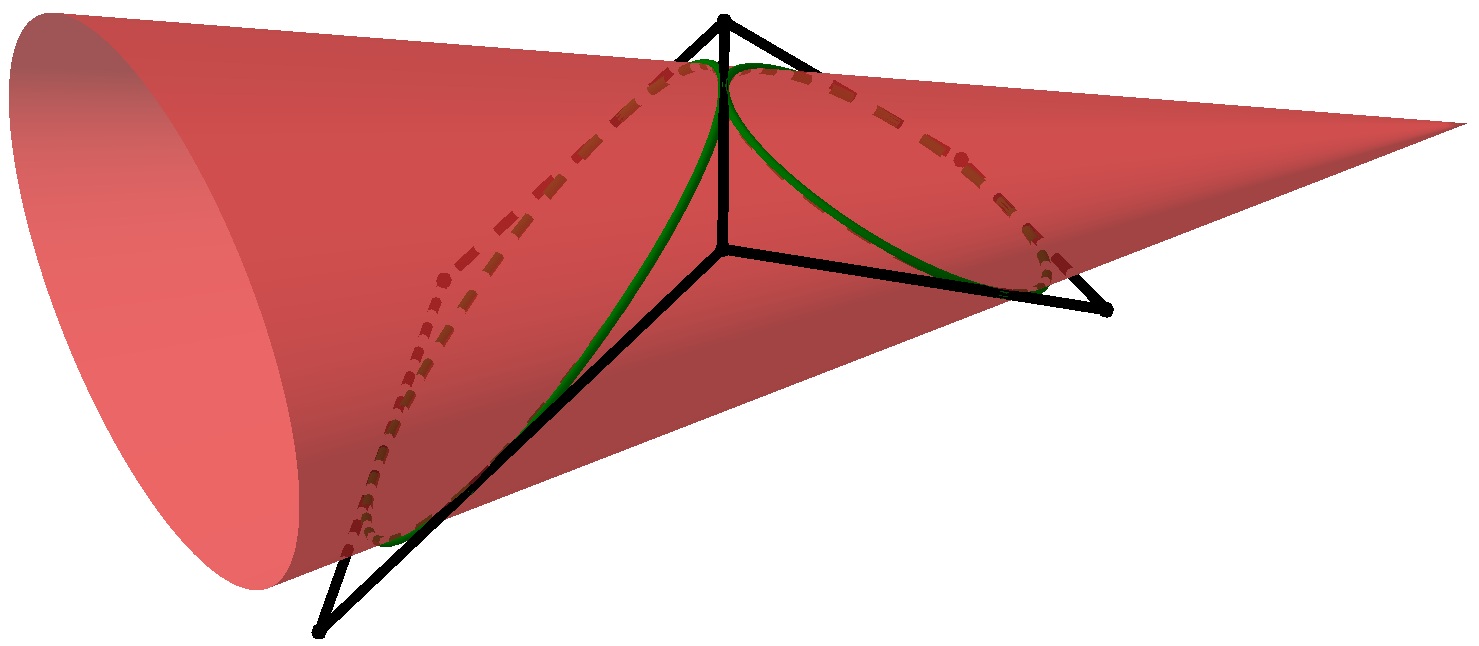}
	\hfill
	\includegraphics[width=0.47\textwidth]{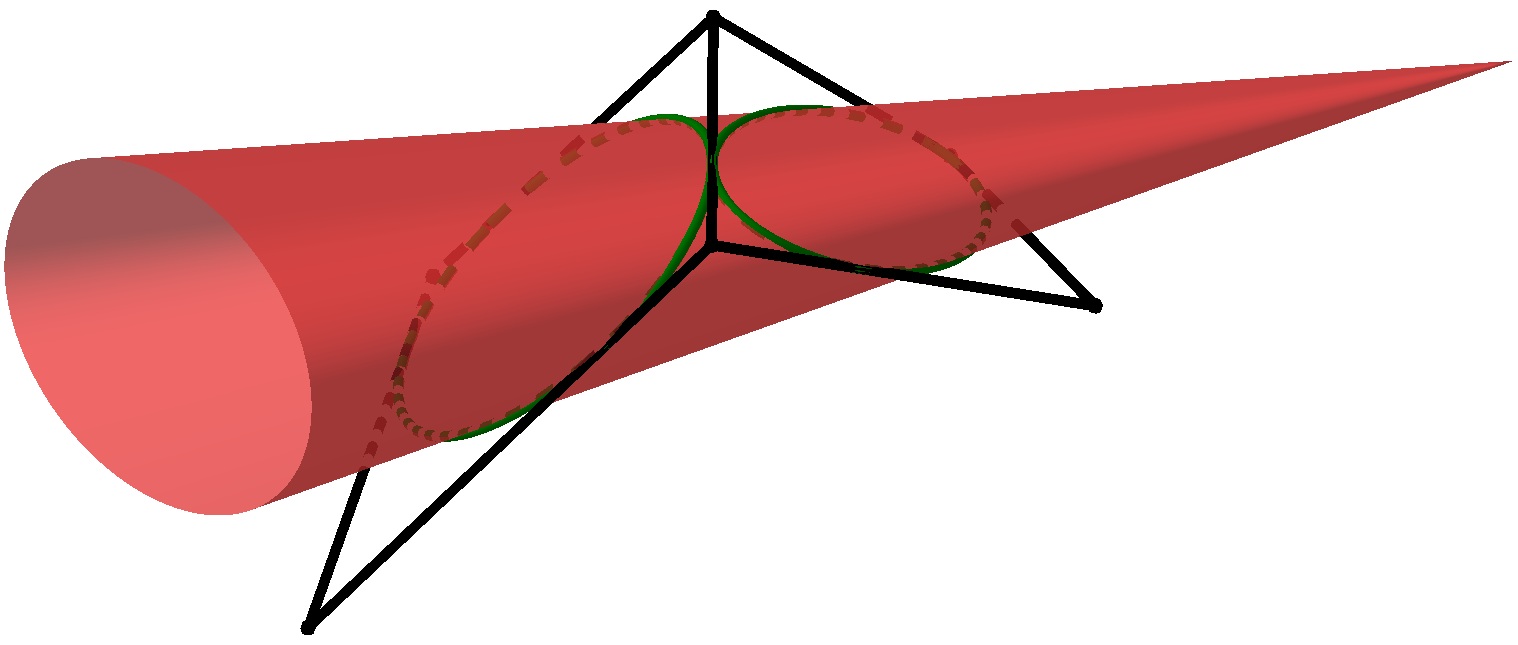}
	\caption{Two inscribed quadrics (red) for the same extensive K{\oe}nigs net but with different instances of touching conics (green). The seven edges (black) are tangent to the inscribed quadrics.}
    \label{figure:twoinscribedquadrics}
\end{figure}

Our second contribution concerns \emph{Kœnigs $d$-grids}, which are Kœnigs nets such that the parameter spaces are $d$-dimensional projective subspaces for some fixed $d$. Note that we consider Kœnigs $d$-grids defined both on rectangular subsets of $\Z^2$ or on \emph{all} of $\Z^2$.
We call a Kœnigs $d$-grid \emph{generic} if it satisfies a mild genericity condition that we make precise later on. Briefly stated, for certain square patches, we require that the points are in general position. With this notion, our second main result is as follows:
\begin{statement}
	Every generic Kœnigs $d$-grid has a unique inscribed quadric.
\end{statement}
Additionally, we show that the unique inscribed quadric is non-degenerate. Furthermore, we show that there are special (non-generic) Kœnigs $d$-grids such that there is a 1-parameter family of inscribed quadrics. The occurrence of these special grids is interesting in itself, but also provides evidence that our genericity assumptions are necessary.

Let us also point out that generic Kœnigs $d$-grids are not extensive, therefore our first main result does not immediately apply to Kœnigs $d$-grids. Despite that, we do make use of the first main result to prove the second main result. Note that the case $d=1$ of our second main result states that in this case there is a unique instance of touching conics such that all the touching conics coincide, and this result was previously obtained in \cite{bobenkofairley2021nets}.

Our third contribution offers an elegant method to describe all Kœnigs $d$-grids. To do so, we introduce the notion of a \emph{(discrete) autoconjugate curve}, which is a discrete curve that is contained in a (non-degenerate) quadric in $\RP^{2d}$, such that its osculating $(d-1)$-spaces are also contained in the quadric. We call a pair of autoconjugate curves \emph{generic} if the pair satisfies a certain genericity condition. Briefly stated, a pair of autoconjugate curves is generic if certain pairs of osculating spaces are in general position. With these definitions, our third main result is as follows:
\begin{statement}
	Generic Kœnigs $d$-grids are in bijection with generic pairs of autoconjugate curves.
\end{statement}
Given a generic Kœnigs $d$-grid, the quadric containing the pair of autoconjugate curves is the unique inscribed quadric provided by our second main result. Furthermore, the autoconjugate curves are the $d$-th forwards and backwards \emph{Laplace transforms} (see Section~\ref{sec:laplace}) of the diagonal intersection net. Conversely, the Kœnigs $d$-grid is obtained by intersecting osculating spaces of the autoconjugate curves.

We believe our results are interesting from the standalone viewpoint of discrete geometry. Additional interest comes from the fact that our results can be compared to classical results in (smooth) differential geometry. In fact, Tzitzéica obtained analogues to our second and third main results  \cite{Tzitzeica1924geometrie} -- almost exactly 100 years ago. However, we should point out that we used completely different (and novel) methods to prove our results. Moreover, in keeping with the standards of the time, the results in \cite{Tzitzeica1924geometrie} are not completely rigorous (by modern standards) in the sense that they are not accompanied by a discussion of sufficient or necessary genericity constraints. Also, as far as the authors are aware, there is no corresponding version of our first main result in (smooth) differential geometry.

We would like to mention that we think our results also raise interesting new questions, see Section~\ref{sec:isothermic} for more details. As mentioned at the beginning of the introduction, isothermic nets are special cases of Kœnigs nets. Consequently, all three of our main results also apply to isothermic nets, though it would be interesting to see what additional properties arise for the inscribed quadrics and autoconjugate curves for isothermic nets. Finally, there has been renewed interest in smooth and discrete isothermic nets with planar or spherical curvature lines \cite{Abresch1987, Walter1987, Bernstein2001, bobenko2023compact, bobenko2023isothermictorifamilyplanar, cpsconstrained, hoffmannSzewieczek2024isothermic}. These are special cases of Kœnigs grids (directly or via a Möbius lift), but they are not generic Kœnigs grids. Thus, it would be interesting to see how our results need to be modified to apply to these cases.

\subsection*{Plan of the paper}

In Section~\ref{sec:basics} we briefly recall basic notions of projective geometry. In Section~\ref{sec:results} we give the precise main definitions and results of our paper. In Sections~\ref{sec:laplace} and \ref{sec:inscribedconics} we recall the necessary theory of Laplace transformations and touching conics, respectively. In Section~\ref{sec:binets} we take a short detour presenting a side result: a new instance of Kœnigs binets that arise naturally in our work. Section~\ref{sec:gluing} gives a version of a classical lemma about pencils of quadrics in a form that is suitable for our proofs. Subsequently, Sections~\ref{sec:isncribedquadrics}, \ref{sec:constrained} and \ref{sec:autoconjugate} each are dedicated to proving one of our three main results. Finally, in Section~\ref{sec:isothermic} we give a few concluding remarks with respect to possible further directions of research.

\subsection*{Acknowledgments}

N.~C.~Affolter and A.~Y.~Fairley were supported by the Deutsche Forschungsgemeinschaft (DFG) Collaborative Research Center TRR 109 ``Discretization in Geometry and Dynamics''. We would like to thank Christian Müller and Jan Techter for invaluable discussions.

\section{Geometry of $\RP^n$}\label{sec:basics}

We briefly review some basic notions of projective geometry that are necessary for the understanding of this paper, see also \cite{casas2014} for more background. 

For $x,y \in \R^{n+1} \setminus \{0\}$ we use the equivalence relation $x \sim y$ if and only if $x= \lambda y$ for some non-zero $\lambda \in \R$. We define \emph{projective space} $\RP^n$ as
\begin{align}
	\RP^n = (\R^{n+1} \setminus \{0\})\ /\ {\sim} = \{ [x] \ | \ x \in \R^{n+1} \setminus \{0\} \}.
\end{align}
For a point $[x] \in \RP^n$, we call $x \in \R^{n+1}\setminus \{0\}$ a \emph{representative vector} of $[x]$. Every $d$-dimensional projective subspace $A \subset \RP^n$ is the projectivization of a $(d+1)$-dimensional linear subspace $a \subset \R^{n+1}$, that is
\begin{align}
	A = [a] = \{[x] \mid x \in a \setminus \{0\}\}.
\end{align}
For two projective subspaces $A = [a]$, $B= [b]$ the \emph{join} $A \vee B$ is
\begin{align}
	A \vee B = [ \mbox{span}\{a,b\}].
\end{align}

For four points $P_1,P_2,P_3,P_4 \in \RP^1$ given by representative vectors $p_1,p_2,p_3,p_4\in \R^2$, the \emph{cross-ratio} is given by
\begin{align}
	\cro(P_1,P_2,P_3,P_4) = \frac{\det(p_1, p_2) \det(p_3, p_4)}{\det(p_2, p_3) \det(p_4,p_1)}.
\end{align}
The cross-ratio is independent of the choice of representative vectors and it is invariant under projective transformations. For four points on a line $\RP^1 \subset \RP^n$, the cross-ratio is defined analogously by restricting to $\RP^1$.

Let $A,C\subset \RP^n$ be two supplementary projective subspaces, which means that $A \vee C = \RP^n$ and $A \cap C = \emptyset$. Recall that the \emph{central projection}  $\pi: \RP^n \setminus C \rightarrow A$ (with \emph{center} $C$) is defined such that $\pi(X)= (X \vee C) \cap A$. Note that cross-ratios are invariant under central projection.

Given a (non-trivial) symmetric bilinear form $\p \colon \R^{n+1} \times \R^{n+1} \to \R$, the associated \emph{quadric} is the set
\begin{align}
	\mathcal{Q} := \{[x] \in \RP^n \mid \p(x,x) = 0\}.
\end{align}
Quadrics in $\RP^2$ are called \emph{conics}.
Non-zero scalar multiples of $\p$ define the same quadric $\mathcal Q$. Two points $[x], [y] \in \RP^n$ are \emph{conjugate} relative to $\mathcal{Q}$ if $\p(x,y) = 0$. The \emph{polar} of a point $P = [p] \in \RP^n$ with respect to $\mathcal{Q}$ is
\begin{align}
	P^\perp := \{[x] \in \RP^n \mid \p (p,x) = 0\}.
\end{align}
A point $P$ in $\mathcal{Q}$ is \emph{singular} if $P^\perp = \RP^n$. A non-empty quadric is \emph{non-degenerate} if it has no singular points. A projective subspace $A \subset \RP^n$ is an \emph{isotropic} subspace of a quadric $\mathcal{Q}$ if $A \subset \mathcal{Q}$.

Let $e_1, e_2, \dots, e_n$ denote the unit vectors of $\R^n$. By a change of basis every symmetric bilinear form $\p$ can be brought into a standard form such that $\p(e_i, e_j) = 0$ whenever $i\neq j$ and such that $\p(e_i,e_i) \in \{1,-1,0\}$. We denote the \emph{signature} of $\p$ by a string composed of $\mathtt{+}$, $\mathtt{-}$, $\mathtt{0}$. The number of  $\mathtt{+}$, $\mathtt{-}$, $\mathtt{0}$ in the signature corresponds to the number of $e_i$ with $\p(e_i,e_i)$ equal to $1,-1,0$ respectively. The signature of a quadric $\mathcal{Q}$ is defined as the signature of any representative symmetric bilinear form. The signature of $\mathcal{Q}$ is well defined up to sign, that is up to exchanging  all $\mathtt+$ with all $\mathtt-$. As an example, the signature of a non-degenerate conic in $\RP^2$ is $\texttt{(++-)}$ or equivalently $\texttt{(+--)}$.

\section{Main results}\label{sec:results}

After introducing some basic notation, we will briefly introduce Q-nets in Section~\ref{sec:resultsqnets}. Subsequently, each of the Sections~\ref{sec:resultskoenigs}, \ref{sec:resultsgrids} and \ref{sec:resultsauto} is devoted to the presentation of one of our three main results.

%We begin by introducing some notation. 
Throughout, we let $m \in \N$ be a positive integer and for $a,b \in \N$ we write
\begin{align}
	\Sigma_m := \{0,1, \ldots, m\}, \quad  \Sigma_{a,b} := \Sigma_a \times \Sigma_b.
\end{align}
We write $\Sigma$ when a statement makes sense for both the finite and the infinite case, that is $\Sigma$ represents $\Sigma_{a,b}$ or $\Z^2$. 

%Thus, Q-nets are defined with the domain $\Sigma$.

%We already mentioned Q-nets \cite{sauer1970differenzen, BS2008DDGbook} in the introduction. We give here a formal definition as well.

\subsection{Q-nets} \label{sec:resultsqnets}

Most maps we consider in the paper are special cases of Q-nets \cite{sauer1970differenzen, dsmultidimconjugate, BS2008DDGbook}, which are defined as follows.
\begin{definition}
	A \emph{Q-net} is a map $\Sigma \rightarrow \RP^n$ such that the image of each face is contained in a plane.
\end{definition}

We are mostly interested in generic Q-nets in the sense of the following definition.
\begin{definition}
	A \emph{non-degenerate} Q-net $P: \Sigma \rightarrow \RP^n$ is a Q-net such that
	\begin{enumerate}
		\item for every edge the two corresponding points do not coincide, and
		\item for any three vertices of any face the corresponding points of $P$ span a plane. 
	\end{enumerate}
\end{definition}
Note that the second condition implies the first condition unless the domain of $P$ is $\Sigma_{a,0}$ or $\Sigma_{0,b}$.
We also use the following stronger notion of genericity, which is crucial to formulate our main results.

\begin{definition}\label{defn:extensive}
	We call a Q-net $P \colon \Sigma_{a,b} \rightarrow \RP^n$ \emph{extensive} if it is non-degenerate and if the image of $P$ joins an $(a+b)$-dimensional projective space. Moreover, for $c,d \in \N$ with $c \leq a, d \leq b$ we say $P$ is \emph{$\Sigma_{c,d}$-extensive} if every $\Sigma_{c,d}$ subpatch of $P$ is extensive.
\end{definition}

Note that a Q-net $P$ restricted to the set	 $\{(i,j) \in \Sigma_{a,b} \mid ij = 0\}$ joins at most $a+b$ dimensions, and the remainder of the points of $P$ is in this join. Hence, a Q-net $P$ defined on $\Sigma_{a,b}$ joins at most $a+b$ dimensions. Thus, an extensive Q-net is joining the maximal possible dimension. Also note that any restriction of an extensive Q-net to a smaller subpatch is also an extensive Q-net.

An important relation between non-degenerate and extensive Q-nets is the following lemma shown in \cite[Lemma~4.2]{AffolterFairleyKoenigsLaplace}.
\begin{lemma}\label{lem:extensive}
	For every non-degenerate Q-net $P$ defined on a finite patch $\Sigma_{a,b}$ there is an extensive Q-net $\hat P$ (called a \emph{lift} of $P$), such that there is a central projection $\pi$ with 
    \begin{align}
        P = \pi \circ \hat P.
    \end{align}
\end{lemma}
As a result, if an object (or property) is invariant under central projections it suffices to consider extensive Q-nets. Importantly, every projection and every lift of a Kœnigs net is a Kœnigs net \cite[Section~4]{AffolterFairleyKoenigsLaplace}. There are also other objects that are not necessarily preserved by central projection -- like isotropic spaces -- but the existence of which in the lift still has implications for the projected Q-net, as we will point out in several remarks throughout the paper.

\subsection{Kœnigs nets} \label{sec:resultskoenigs}

In this paper, we study \emph{(discrete) Kœnigs nets} as introduced by Bobenko and Suris \cite{BS2009Koenigsnets}, which are a discretization of \emph{smooth Kœnigs nets} \cite{BS2008DDGbook}. Note that in some of the classical literature \cite{darboux1915lecons, Tzitzeica1924geometrie} the name ``Kœnigs nets'' was not yet established, instead the nets were simply called \emph{conjugate nets with equal Laplace invariants}. Kœnigs nets are studied in differential geometry for several reasons, maybe most importantly because if a surface has a curvature-line parametrization that is also a Kœnigs net, then the surface is an isothermic surface.
In the following, when we refer to Kœnigs nets we always mean discrete Kœnigs nets.
Note that Kœnigs nets have previously appeared in \cite{sauer1933wackelige} and \cite{Doliwa2007}, but without the interpretation as a discretization of smooth Kœnigs nets.

There are multiple equivalent characterizations of Kœnigs nets, see \cite{BS2008DDGbook}. We will define Kœnigs nets with the characterization that is most useful (and crucial) to prove our main results. The characterization is comparatively novel and involves certain touching inscribed conics \cite{bobenkofairley2021nets}.

\begin{definition}
	We say that a $Q$-net $P\colon \Z^2 \to \RP^n$ has an instance of \emph{touching conics} $\mathcal C$ if each quad is equipped with an inscribed conic such that for every edge the conics of the two adjacent quads touch the line of the edge in the same point (see Figure~\ref{figure:touchingconics}).    
\end{definition}
We use the convention that $\mathcal C(i,j)$ is the conic inscribed into the quad with the vertices $P(i,j)$, $P(i+1,j)$, $P(i,j+1)$ and $P(i+1,j+1)$.

\begin{remark} \label{rem:degconics}
	A special case of a (degenerate) conic is a (double) line. We do not consider a double line to be an inscribed conic of a planar quad $P$, unless the double line is the line through the two Laplace points (Definition~\ref{def:laplace}). This convention ensures that our definition of Kœnigs nets agrees with other characterizations of Kœnigs nets. In particular, we are excluding the case that the inscribed conics are diagonals of the quads. Otherwise, for any Q-net there would be two instances of touching conics that are determined by the diagonals.
\end{remark}

We are ready to give our definition of Kœnigs nets.

\begin{definition}\label{thm:BSKoenigstouchingconics}
	A non-degenerate Q-net $P\colon \Z^2 \to \RP^n$ is a Kœnigs net if and only if it admits an instance of touching conics. Moreover, every Kœnigs net admits a $1$-parameter family of touching conics (see Figure~\ref{figure:touchingconics}).
\end{definition}

The equivalence of this definition to previously found characterizations of Kœnigs nets was proven in \cite[Thm 3.5]{bobenkofairley2021nets}. Note that the characterization was also independently found by Christian Müller and Martin Peternell [unpublished].

\begin{remark}\label{rem:Koenigsconicssmoothdiscrete}
    In the smooth theory, there is a theorem called \emph{Kœnigs' theorem}, see \cite[p.~79]{Tzitzeica1924geometrie}. The theorem characterizes K{\oe}nigs nets as conjugate nets $x \colon \R^2 \to \RP^n$ that satisfy a condition involving the existence of certain conics. Let us briefly explain Kœnigs' condition. 
    Let $x_{\pm 1}$ be the two Laplace transforms of $x$. The condition is that for each $(u,v) \in \R^2$ there exists a conic that has second order contact at the two points $x_{\pm 1}(u,v)$ with the two tangent lines of the parameter curves at $x(u,v)$. Let us compare this to the discrete theory. According to Definition~\ref{thm:BSKoenigstouchingconics}, each (discrete) K{\oe}nigs net $P$ has an instance of touching conics $\mathcal{C}$. Each inscribed conic $\mathcal{C}(i,j)$ is tangent to the four edge-lines of its corresponding planar quad of $P$, which is quite similar to the smooth theory. However, the conic $\mathcal{C}(i,j)$  does not contain the Laplace points $P_{\pm 1}(i,j)$ if $\mathcal{C}(i,j)$ is non-degenerate. Therefore, it is not clear whether any instance of touching conics $\mathcal{C}$ is a counterpart of the conics in the smooth theory. Note that Doliwa's discretization of Kœnigs nets (which are the diagonal intersection nets of the discrete Kœnigs nets we consider) comes with a discretization of the aforementioned conics in the smooth theory \cite{doliwa2003} -- see also Remark~\ref{rem:doliwaconic}.
\end{remark}

In the following, we want to generalize the notion of inscribed conics to inscribed quadrics on larger patches. With that in mind, we introduce a few useful spaces associated to Kœnigs nets.
First of all, for any net $P\colon \Sigma \to \RP^n$ we use the notation
\begin{align}
	\pah{P} (j) := \mathrm{join}\{P(i,j) \mid i : (i,j) \in \Sigma \}, \qquad \pav{P}(i) := \mathrm{join}\{P(i,j) \mid j : (i,j) \in \Sigma)\},
\end{align}
and we call each $\pah{P} (j)$ and each $\pav P(i)$ a \emph{parameter space} of $P$.
Moreover, for a Kœnigs net $P$ with touching conics $\mathcal C$, we define the \emph{touching points} $S \colon \Sigma_{a-1, b} \to \RP^n$ and $T \colon \Sigma_{a, b-1} \to \RP^n$ such that:
\begin{enumerate}
	\item $S(i,j)$ is the touching point of $\inc(i,j)$ with the line $P(i,j) \vee P(i+1,j)$,
	\item $T(i,j)$ is the touching point of $\inc(i,j)$ with the line $P(i,j) \vee P(i,j+1)$.
\end{enumerate}

Accordingly, $\pah{S}(j)$, $\pav S(i)$ and $\pah{T}(j)$, $\pav T(i)$ denote the spaces joined by the parameter lines of $S$ and $T$.

\begin{definition} \label{def:inscribedquadric}
	Let $P\colon \Sigma_{a,b} \to \RP^{a+b}$ be an extensive Kœnigs net with an instance of touching conics $\mathcal C$. A quadric $\mathcal{Q}$ is an \emph{inscribed quadric} if 
	\begin{enumerate}
		\item For all $(i,j) \in \Sigma_{a-1,b-1}$ the plane 
		\begin{align}
\Pi(i,j):= P(i,j) \vee P(i+1,j) \vee P(i,j+1)
		\end{align}
intersects $\mathcal{Q}$ in the corresponding inscribed conic
		\begin{align}
			\mathcal C(i,j) =  \mathcal Q \cap \Pi(i,j),
		\end{align}
		\item the space $\pav{T}(i)$ is contained in $\mathcal{Q}$ for all $i\in \Sigma_{a}$, and 
		\item the spaces $\pah{S}(j)$ is contained in $\mathcal{Q}$ for all $j\in \Sigma_{b}$.
	\end{enumerate}
\end{definition}

Let us explain why we call $\mathcal Q$ an \emph{inscribed} quadric. We will show in Lemma~\ref{lem:BSquadriclocalpatch} that Definition~\ref{def:inscribedquadric} implies that each $\pav P(i)$ is tangent to $\mathcal Q$ along the space $\pav T(i)$, and each $\pah P(j)$ is tangent to $\mathcal Q$ along the space $\pah S(j)$. In particular, if $P$ is defined on $\Sigma_{1,1}$, then $P$ is simply one quad. In this case, $\mathcal Q$ coincides with $\mathcal C(0,0)$. Moreover, the spaces $\pah{S}(0)$, $\pah{S}(1)$, $\pav{T}(0)$ and $\pav{T(1)}$ are the four touching points of $C(0,0)$ with the four tangents $\pah{P}(0)$, $\pah{P}(1)$, $\pav{P}(0)$ and $\pav{P}(1)$. Hence, we view inscribed quadrics as a generalization of inscribed conics.

We are now able to give a precise formulation of our first main result.
\begin{theorem}\label{thm:BSquadriclocalpatch}
	Let $P\colon \Sigma_{a,b} \to \RP^{a+b}$ be an extensive Kœnigs net with an instance of touching conics $\mathcal C$. Then, there is a unique inscribed quadric $\mathcal Q$.
\end{theorem}

Of course, since there is a 1-parameter family of touching conics $\mathcal C$, there is also a 1-parameter family of inscribed quadrics $\mathcal Q$. Note that while a choice of touching conics $\mathcal C$ consists of a set of $a \times b$ many conics, $\mathcal Q$ is just one quadric. 

Let us also add that we show that if $a \neq b$, that is if $P$ is not defined on a square-patch, then $\mathcal Q$ will be a degenerate quadric. In the square-patch case $\mathcal Q$ may be non-degenerate, but does not have to be. For example, in the $\Sigma_{2,1}$ case $\mathcal Q$ will be a cone or a (double) plane. In the $\Sigma_{1,1}$ case $\mathcal Q$ coincides with the inscribed conic $\mathcal C$, which may be a non-degenerate conic or a (degenerate) double line. 

Moreover, let us note that although Theorem~\ref{thm:BSquadriclocalpatch} only considers extensive Kœnigs nets, there are implications for non-extensive Kœnigs nets. For example, consider a Kœnigs net $P$ defined on $\Sigma_{2,2}$, such that the points only join $\RP^3$ instead of $\RP^4$. Let us also fix an instance of touching conics $\mathcal C$. A consequence of Theorem~\ref{thm:BSquadriclocalpatch} is that generally there will be a quadric $\mathcal Q$ that is tangent to each parameter space in a point, and such that each conic of $\mathcal C$ is tangent in two points to $\mathcal Q$. See Remark~\ref{rem:twicetangetnconics} for more details.

\begin{remark}\label{rem:inscribedquadriclowerdim}
	Let us mention that restrictions of extensive Kœnigs nets to subpatches are also extensive Kœnigs nets. Hence, it is also possible to define touching quadrics that exist on subpatches. For example, one might consider all the 3-dimensional inscribed quadrics defined on $\Sigma_{2,2}$ patches. Now, consider two adjacent $\Sigma_{2,2}$ patches, which are patches that overlap in a $\Sigma_{2,1}$ or $\Sigma_{1,2}$ patch. In this case, adjacent $\Sigma_{2,2}$-quadrics have two 2-dimensional tangent planes in common, and in each tangent plane they share an isotropic line. Moreover, the two quadrics have the same $3$-dimensional tangent space at the intersection point of the two aforementioned isotropic lines.
\end{remark}

\subsection{Kœnigs grids} \label{sec:resultsgrids}

Let us turn to Kœnigs nets with constrained parameter spaces. Specifically, we are interested in Kœnigs nets such that the parameter spaces are projective subspaces of some fixed dimension $d \in \N$. 

\begin{definition}\label{defn:ddimparameterlines}
	For $d\in \N$ we say that a map $P\colon \Sigma \to \RP^n$ is a \emph{$d$-grid} if the spaces $\pah{P}(j)$ and $\pav{P}(i)$ are $d$-dimensional for all $i,j$.
\end{definition}

To give our second main result we need a genericity condition on $d$-grids. The condition involves the notion of the $d$-th \emph{Laplace transforms} $P_d$ and $P_{-d}$ and the non-degeneracy condition that $P_d$ and $P_{-d}$ are \emph{nowhere Laplace degenerate}. Since this involves some technical notions, we postpone the introduction of these terms to Section~\ref{sec:laplace}, and more specifically Definition~\ref{def:laplace} and Definition~\ref{defn:nowherelaplace}.

\begin{definition}\label{def:genericgrid}
	Let $P \colon \Sigma \to \RP^n$ be a $d$-grid. Then $P$ is a \emph{generic $d$-grid} if $P$ is $\Sigma_{d,d}$-extensive and both $P_d$ and $P_{-d}$ exist and are nowhere Laplace degenerate.
\end{definition}

We now give a precise formulation of our second main result.

\begin{theorem}\label{thm:BSKoenigsconstrainedtangentquadric}
	Let $P\colon \Sigma \to \RP^{2d}$ be a generic Kœnigs $d$-grid. Then, there is a unique non-degenerate quadric $\mathcal{Q}$ such that each $\pav{P}(i)$ and each $\pah{P}(j)$ is tangent to $\mathcal{Q}$ along an isotropic $(d-1)$-space of $\mathcal{Q}$.
\end{theorem}

\begin{figure}[tb] 
	\centering
	\includegraphics[width=0.95\textwidth]{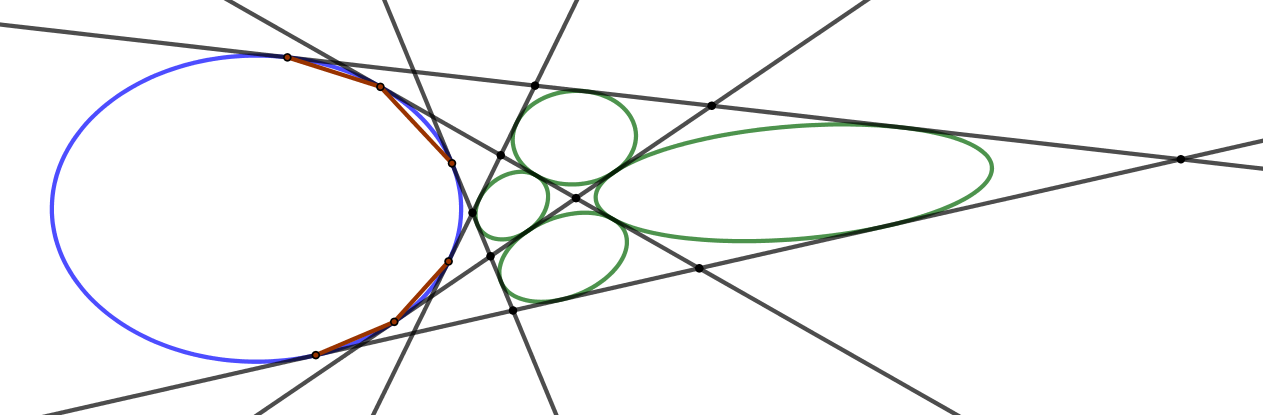}
	\caption{A generic K{\oe}nigs $1$-grid with an instance of touching conics. All the parameter spaces (the gridlines) are tangent to a non-degenerate conic.}  
	\label{figure:Koenigs1grid}
\end{figure}

Note that Theorem~\ref{thm:BSKoenigsconstrainedtangentquadric} is a generalization of a result in \cite{bobenkofairley2021nets} which deals with the $d=1$ case (see Figure~\ref{figure:Koenigs1grid} for an illustration). We also note that Theorem~\ref{thm:BSKoenigsconstrainedtangentquadric} has applications to non-generic Kœnigs $d$-grids. See Remark~\ref{rem:koenigsgridRP3} for an example.

\begin{remark}
	Note that if $P$ is a generic $d$-grid, then it turns out that the Laplace transforms $P_d$ and $P_{-d}$ are \emph{Goursat degenerate}, see \cite{AffolterFairleyKoenigsLaplace} for more background. In that sense, Theorem~\ref{thm:BSKoenigsconstrainedtangentquadric} provides a discretization of an analogous theorem for smooth Kœnigs nets with Goursat degenerations, which can be found in \cite[p.~156]{Tzitzeica1924geometrie}. However, note that the classical theory in \cite{Tzitzeica1924geometrie} does not precisely formulate necessary or sufficient genericity conditions for their results.
\end{remark}

\subsection{Autoconjugate curves} \label{sec:resultsauto}

We proceed to introduce (discrete) autoconjugate curves. We will show that pairs of such curves are in bijection to Kœnigs $d$-grids.

Consider a discrete curve $\gamma \colon \Z \rightarrow \RP^{n}$. For $k\in \Sigma_{n}$, we call the space
\begin{align}\label{eq:osculatingkspace}
	C_{(k)}(j) = \gamma(j) \vee \gamma(j+1) \vee  \dots \vee \gamma(j + k),
\end{align}
an \emph{osculating space} of $\gamma$. By convention, $C_{(-1)}(j) = \emptyset$.
Later, we will introduce a genericity assumption that ensures that each $C_{(k)}(j)$ is $k$-dimensional.

Osculating spaces also exist for smooth curves. Moreover, there is a characterization of smooth \emph{autoconjugate curves} \cite[p.~35]{Tzitzeica1924geometrie} in terms of $(d-1)$-dimensional osculating spaces that are contained in a quadric in $\RP^{2d}$. This characterization is our motivation for Definition~\ref{defn:autoconjugate}.

\begin{definition}\label{defn:autoconjugate}
	Let $\mathcal{Q}$ be a quadric and let $\gamma \colon \Z \to \RP^{2d}$ be a (discrete) curve. 
	The curve $\gamma$ is \emph{autoconjugate} with respect to $\mathcal{Q}$ if every osculating space $C_{(d-1)}(i)$ is contained in $\mathcal Q$.
\end{definition}

Note that the original definition of smooth autoconjugate curves in \cite[p.~35]{Tzitzeica1924geometrie} is different in character, and is not immediately suited for discretization. We discuss this in more detail in Remark~\ref{rem:autoconjugateothercharacterisation}.

In the following, we will need the following two genericity definitions.

\begin{definition}\label{def:genericcurve}
    We call a curve $\gamma \colon \Z \to \RP^{n}$ \emph{generic} if, for all $k \in \Sigma_n$, the osculating space $C_{(k)}$ has dimension $k$.
\end{definition}

Note that generic curves are exactly those that are $\Sigma_{n,0}$-extensive, when considered as Q-nets restricted to $\Z \times \{0\}$. Moreover, Definition~\ref{def:genericcurve} implies that all $C_{(n-1)}(j)$ are hyperplanes, and that consecutive hyperplanes are distinct -- for example $C_{(n-1)}(0)$ is not equal to $C_{(n-1)}(1)$.

\begin{definition} \label{def:genericpair}
	Let $\sigma, \tau \colon \Z \to \RP^{n}$ be a pair of curves, and let $S_{(k)}(j)$ and $T_{(k)}(i)$ be the osculating $k$-spaces of $\sigma$ and $\tau$, respectively. We call the pair \emph{generic} if for all $i,j \in \Z$ and all $k$ with $-1 \leq k \leq n$
	\begin{align}
		\dim S_{(k)}(j) \vee T_{(n - k - 1)}(i) = n.
	\end{align}
\end{definition}
Note that Definition~\ref{def:genericpair} implies that both curves $\sigma$ and $\tau$ are generic in the sense of Definition~\ref{def:genericcurve}. Additionally, the definition also implies that all ``smaller'' joins have maximal dimension, that is
\begin{align}
	\dim S_{(a)}(j) \vee T_{(b)}(i) = a + b + 1, \label{eq:autogendim}
\end{align}
whenever $a + b + 1 \leq n$.

\begin{theorem}\label{thm:bijection}
	Generic Kœnigs $d$-grids correspond bijectively to generic pairs of autoconjugate curves of non-degenerate quadrics in $\RP^{2d}$.
\end{theorem}

We will give an explicit construction of both directions of the bijection in Section~\ref{sec:autoconjugate}, see Theorem~\ref{thm:pairauutoconjugategivesgenericBSKoenigsgrid} and Theorem~\ref{thm:genericBSKoenigsdgridgivesgenericpairautoconjugate}. As a result, all that is needed to construct a generic Kœnigs $d$-grid is a generic pair of autoconjugate curves -- and every generic Kœnigs $d$-grid may be constructed in this manner. Also note that Theorem~\ref{thm:bijection} is a discretization of an analogous theorem in the smooth theory \cite[p.~161]{Tzitzeica1924geometrie} (except that the necessary or sufficient genericity constraints were not investigated in the smooth theory).

Note that in the $d=1$ case an autoconjugate curve is simply a curve inscribed in a conic. See also Figure~\ref{figure:Koenigs1grid}, which shows a pair of autoconjugate curves corresponding to a K{\oe}nigs $1$-grid.

\begin{remark}\label{rem:autolow}
    One may also consider Kœnigs $d$-grids in lower dimensions than $\RP^{2d}$. In the example of a Kœnigs $2$-grid in $\RP^3$, we expect there to be an associated pair of discrete curves such that each edge is tangent to a quadric $\mathcal Q$. However, it is less clear if the converse direction also holds, whether the correspondence is $1:1$, and what the appropriate genericity conditions are.
\end{remark}

\section{Laplace transformations} \label{sec:laplace}

\begin{figure}
	\centering
	\begin{tikzpicture}[line cap=round,line join=round,>=triangle 45,x=1.0cm,y=1.0cm,scale=0.8]
		\definecolor{qqwuqq}{rgb}{0.,0.39215686274509803,0.}
		\clip(-2.06,-2.3) rectangle (12.42,5.56);
		%		\draw (current bounding box.north east) -- (current bounding box.north west) -- (current bounding box.south west) -- (current bounding box.south east) -- cycle;
		\draw [line width=1pt,densely dashed] (5.64,2.46)-- (-1.62,-1.);
		\draw [line width=1pt,densely dashed] (-1.62,-1.)-- (5.5,5.16);
		\draw [line width=1pt,densely dashed] (5.5,5.16)-- (5.64,2.46);
		\draw [line width=1pt,densely dashed] (5.64,2.46)-- (12.16,-1.48);
		\draw [line width=1pt,densely dashed] (12.16,-1.48)-- (5.5,5.16);
		\draw [line width=1pt,densely dashed] (5.64,2.46)-- (5.3,0.74);
		\draw [line width=1pt,densely dashed] (5.3,0.74)-- (0.5278102388403072,0.023612042202129757);
		\draw [line width=1pt,densely dashed] (5.3,0.74)-- (10.566791115099843,-0.5172326677137082);
		\draw [line width=1pt,densely dashed] (3.749048234236283,3.6451316183842)-- (3.905296931048227,1.6332682343563176);
		\draw [line width=1pt,densely dashed] (3.905296931048227,1.6332682343563176)-- (3.8914355385497297,0.5285501900662476);
		\draw [line width=1pt,densely dashed] (7.633390140055851,0.182997786225348)-- (7.151434597262734,1.546648418218532);
		\draw [line width=1pt,densely dashed] (7.151434597262734,1.546648418218532)-- (7.819811895183212,2.847154506904426);
		\draw [line width=1.5pt,color=qqwuqq] (0.5278102388403072,0.023612042202129757)-- (-1.62,-1.);
		\draw [line width=1.5pt,color=qqwuqq] (-1.62,-1.)-- (12.16,-1.48);
		\draw [line width=1.5pt,color=qqwuqq] (12.16,-1.48)-- (10.566791115099843,-0.5172326677137082);
		\draw [line width=1.5pt,color=qqwuqq] (10.566791115099843,-0.5172326677137082)-- (0.5278102388403072,0.023612042202129757);
		\draw [line width=1.5pt] (3.749048234236283,3.6451316183842)-- (5.5,5.16);
		\draw [line width=1.5pt] (5.5,5.16)-- (5.64,2.46);
		\draw [line width=1.5pt] (5.64,2.46)-- (3.905296931048227,1.6332682343563176);
		\draw [line width=1.5pt] (3.905296931048227,1.6332682343563176)-- (3.749048234236283,3.6451316183842);
		\draw [line width=1.5pt] (3.905296931048227,1.6332682343563176)-- (3.8914355385497297,0.5285501900662476);
		\draw [line width=1.5pt] (3.8914355385497297,0.5285501900662476)-- (5.3,0.74);
		\draw [line width=1.5pt] (5.3,0.74)-- (5.64,2.46);
		\draw [line width=1.5pt] (5.3,0.74)-- (7.633390140055851,0.182997786225348);
		\draw [line width=1.5pt] (7.633390140055851,0.182997786225348)-- (7.151434597262734,1.546648418218532);
		\draw [line width=1.5pt] (7.151434597262734,1.546648418218532)-- (5.64,2.46);
		\draw [line width=1.5pt] (5.5,5.16)-- (7.819811895183212,2.847154506904426);
		\draw [line width=1.5pt] (7.819811895183212,2.847154506904426)-- (7.151434597262734,1.546648418218532);
		\begin{small}
			\draw [fill=black] (5.64,2.46) circle (2.5pt);
			\draw[color=black] (6.4,2.73) node {$P(1,1)$};
			\draw [fill=qqwuqq] (-1.62,-1.) circle (2.5pt);
			\draw[color=qqwuqq] (-1.0,-1.4) node {$P_{-1}(0,1)$};
			\draw [fill=black] (5.5,5.16) circle (2.5pt);
			\draw[color=black] (6.5,5.17) node {$P(1,2)$};
			\draw [fill=qqwuqq] (12.16,-1.48) circle (2.5pt);
			\draw[color=qqwuqq] (11.29,-1.76) node {$P_{-1}(1,1)$};
			\draw [fill=black] (5.3,0.74) circle (2.5pt);
			\draw[color=black] (6.2,0.9) node {$P(1,0)$};
			\draw [fill=qqwuqq] (0.5278102388403072,0.023612042202129757) circle (2.5pt);
			\draw[color=qqwuqq] (1.0,-0.4) node {$P_{-1}(0,0)$};
			\draw [fill=qqwuqq] (10.566791115099843,-0.5172326677137082) circle (2.5pt);
			\draw[color=qqwuqq] (9.67,-0.8) node {$P_{-1}(1,0)$};
			\draw [fill=black] (3.749048234236283,3.6451316183842) circle (2.5pt);
			\draw[color=black] (2.9,3.83) node {$P(0,2)$};
			\draw [fill=black] (3.905296931048227,1.6332682343563176) circle (2.5pt);
			\draw[color=black] (3,1.83) node {$P(0,1)$};
			\draw [fill=black] (3.8914355385497297,0.5285501900662476) circle (2.5pt);
			\draw[color=black] (3.9,0.1) node {$P(0,0)$};
			\draw [fill=black] (7.633390140055851,0.182997786225348) circle (2.5pt);
			\draw[color=black] (6.8,-0.05) node {$P(2,0)$};
			\draw [fill=black] (7.151434597262734,1.546648418218532) circle (2.5pt);
			\draw[color=black] (8,1.69) node {$P(2,1)$};
			\draw [fill=black] (7.819811895183212,2.847154506904426) circle (2.5pt);
			\draw[color=black] (8.4,3.21) node {$P(2,2)$};
		\end{small}
	\end{tikzpicture}
	\caption{The Laplace transform $\mathcal L_-P = P_{-1}$ (green) of a Q-net $P$ (black).}
	\label{fig:laplacetransform}
\end{figure}
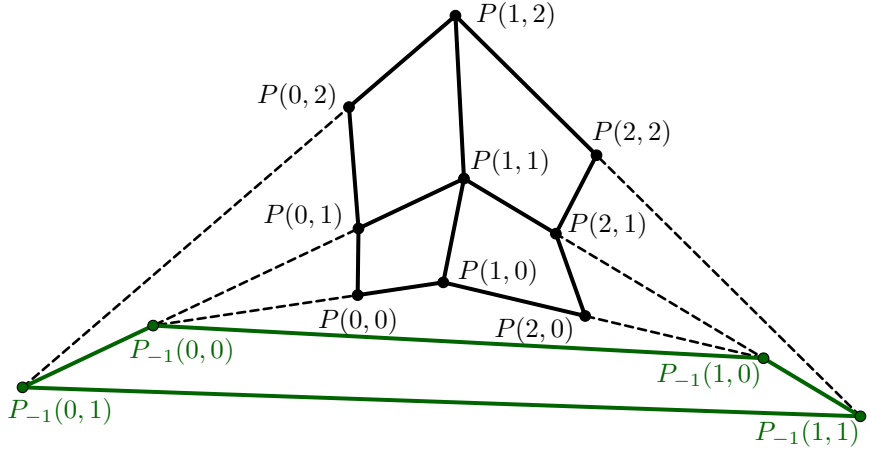

Since we used the language of Laplace sequences to precisely formulate the genericity assumptions for Theorem~\ref{thm:BSKoenigsconstrainedtangentquadric} and Theorem~\ref{thm:bijection}, let us briefly recall some theory. Laplace transformations of Q-nets were introduced by Doliwa \cite{doliwa1997geometricToda} as follows.

\begin{definition}\label{def:laplace}
	For a non-degenerate Q-net $P\colon \Sigma \rightarrow \RP^n$, the Laplace transforms $\mathcal{L}_{+}P$ and $\mathcal{L}_{-}P$ are defined by 
	\begin{align}
		\mathcal{L}_{+}P(i,j) := (P(i,j)\vee P(i,j+1))\cap (P(i+1,j)\vee P(i+1,j+1)), \label{eq:laplaceforward}\\
		\mathcal{L}_{-}P(i,j):= (P(i,j) \vee P(i+1,j))\cap (P(i,j+1) \vee P(i+1,j+1)).\label{eq:laplacebackward}
	\end{align}
\end{definition}
The intersection points in Definition~\ref{def:laplace} exist since the corresponding two lines are contained in the plane of the corresponding quad of the Q-net $P$, see Figure~\ref{fig:laplacetransform}. They are well-defined since $P$ is non-degenerate. Note that we only consider $\mathcal L_\pm P$ to be well-defined if $P$ is indeed non-degenerate, even though the formulas might also make sense in some degenerate cases. Also, for a Q-net $P$ defined on $\Sigma_{a,b}$ the Laplace transforms $\mathcal{L}_\pm P$ are defined on $\Sigma_{a-1,b-1}$.

It is straightforward to check that the Laplace transforms $\mathcal{L}_+ P$ and $\mathcal{L}_- P$ are Q-nets. %; it is an exercise in \cite{BS2008DDGbook}. 
Assuming that the Laplace transforms are non-degenerate, Laplace transformations can be iterated. To lighten the notation, for $m \geq 0$ we write
\begin{align}
	P_m := (\mathcal L_+)^m P \quad \mbox{and} \quad P_{-m} := (\mathcal L_-)^{m}P
\end{align}
for the $m$-times iterated Laplace transforms.

It readily follows from Definition~\ref{def:laplace} that the Laplace transformations $\mathcal{L}_+$ and $\mathcal{L}_-$ are mutually inverse up to an index shift, that is:
\begin{align}
\mathcal{L}_+ \circ \mathcal{L}_- P (i,j) = \mathcal{L}_- \circ \mathcal{L}_+ P (i,j) = P(i+1,j+1).    
\end{align}

The iterated Laplace transforms determine the {\em Laplace sequence}
\begin{align}
\ldots  \leftarrow P_{-3} \leftarrow  P_{-2}  \leftarrow P_{-1}  \leftarrow P \rightarrow P_{1} \rightarrow P_{2} \rightarrow P_{3} \rightarrow \ldots
\end{align}

Generically, the Laplace sequence is infinite in both directions. However, sometimes the Laplace sequence terminates in one direction, which means that some kind of degeneracy occurs in the Laplace sequence. There are two main types of degeneracies studied in the literature, called \emph{Laplace} and \emph{Goursat degenerate}. See \cite{AffolterFairleyKoenigsLaplace} for more details. In the present paper, we only consider the first type of degeneracy.

\begin{definition}\label{defn:laplace}
	Let $P\colon \Sigma \to \RP^n$ be a non-degenerate Q-net.
	\begin{enumerate}
		\item $P_m$ is \emph{Laplace degenerate} if $P_m(i,j)$ is independent of $i$ for all $j$.
		\item $P_{-m}$ is \emph{Laplace degenerate} if $P_{-m}(i,j)$ is independent of $j$ for all $i$.
	\end{enumerate}
\end{definition}

We will need the following strengthening of the condition that certain Laplace transforms are \emph{not} Laplace degenerate.

\begin{definition}\label{defn:nowherelaplace}
	Let $P\colon \Sigma \to \RP^n$ be a non-degenerate Q-net.
	\begin{enumerate}
		\item $P_m$ is \emph{nowhere Laplace degenerate} if $P_m(i,j) \neq P_m(i+1,j)$ for all $i,j$.
		\item $P_{-m}$ is \emph{nowhere Laplace degenerate} if $P_{-m}(i,j) \neq P_{-m}(i,j+1)$ for all $i,j$.
	\end{enumerate}
\end{definition}

The condition for $P_m$ not to be nowhere Laplace degenerate is one of the assumptions we used to define generic $d$-grids, see  Definition~\ref{def:genericgrid}.
We will also need the following lemma which is a combination of \cite[Lemma 5.5, Corollaries 5.6, 5.7]{AffolterFairleyKoenigsLaplace}.
\begin{lemma}\label{lem:expllaplacetransform}
	Let $P\colon \Sigma_{m,m} \to \RP^{2m}$ be an extensive Q-net. Suppose that $P_m$ exists. Then
	we get the explicit formula
	\begin{align}
		P_m(0,0) = \bigcap_{k = 0}^{m} \pav{P}(k).
	\end{align}
	Moreover, for $0 \leq u \leq m$ we have that
	\begin{align}
		\dim \left( \bigcap_{k = 0}^{u} \pav{P}(k) \right) = m - u. \label{eq:laplacedim}
	\end{align}
	Furthermore, if $P\colon \Sigma_{m+1,m+1} \to \RP^{2m}$ is a generic $m$-grid, then
	\begin{align}
		\bigcap_{k = 0}^{m+1} \pav{P}(k) = \emptyset.
	\end{align}
\end{lemma}

\section{Touching conics}\label{sec:inscribedconics}

In this section, we discuss some basic properties of touching conics that we need for our main results.

Let us briefly recall what it means that two pairs of lines $\{L_1,L_3\}$ and $\{L_2,L_4\} $ in $\RP^2$ are \emph{harmonically separated}. It means that the four lines are concurrent in a point $X$ and that the cross-ratio $\cro(L_1,L_2,L_3,L_4)$ of the four lines is $-1$. The cross-ratio of four concurrent lines can be expressed by choosing another line $H$ that does not contain $X$ via the formula
\begin{align}
	\cro(L_1,L_2,L_3,L_4) = \cro(H \cap P_1, H \cap P_2, H \cap P_3, H \cap P_4).
\end{align}
It turns out that this expression does not depend on the choice of $H$.

\begin{lemma}\label{lem:harmonicdiagonals}
	Let $P\colon \Sigma_{1,1} \to \R \mathrm{P}^2$ be a planar quad with a non-degenerate inscribed conic. The two lines $\pav{S}(0)$ and $\pah{T}(0)$ harmonically separate the two diagonals
	\begin{align}
		P(0,0)\vee P(1,1), \qquad P(1,0) \vee P(0,1).
	\end{align}
	In particular, the four lines are concurrent (see Figure~\ref{figure:Koenigsbinetface}).
\end{lemma}

\begin{proof}
	We choose an affine chart such that the points of $P$ are the vertices of a square. Any non-degenerate inscribed conic of a square is symmetric with respect to both diagonals. Therefore, the four lines are concurrent. Moreover, the diagonals are the angle bisectors of $\pav{S}(0)$ and $\pah{T}(0)$. Consequently, the lines $\pav{S}(0)$ and $\pah{T}(0)$ harmonically separate the two diagonals. 
\end{proof}

A property of the touching nets $S,T$ that we will use repeatedly is the following.

\begin{lemma}\label{lem:touchingpointsQnet}
	Let $P \colon \Sigma_{a,b} \to \RP^n$ be a Kœnigs net with an instance of touching conics. Then $S$ and $T$ are Q-nets. 
\end{lemma}
\begin{proof}
	Let us begin with $S$, for which it suffices to consider $P$ defined on a $\Sigma_{2,1}$ patch. Suppose that $\mathcal{C}(0,0)$ or $\mathcal{C}(1,0)$ is degenerate. Then, $S(0,0) = S(0,1)$ or $S(1,0) = S(1,1)$ and thus the points $S(0,0), S(1,0)$, $S(1,1)$, $S(0,1)$ are coplanar. It only remains to consider the case when $\mathcal{C}(0,0)$ and $\mathcal{C}(1,0)$ are both non-degenerate.
    Define the points
	\begin{align}
		X &= \pav S(0) \cap \pav P(1), & X' &= \pav S(1) \cap \pav P(1).
	\end{align}
	Note that $P(1,0), P(1,1)$ are the intersection of the line $P(1)$ with the diagonals of the first quad (and the second quad). Thus, Lemma~\ref{lem:harmonicdiagonals} implies that
	\begin{align}
		\cro(P(1,0), T(1,0), P(1,1), X) & =-1, & \cro(P(1,0), T(1,0), P(1,1), X') & =-1.
	\end{align}
	Therefore, $X$ and $X'$ coincide, which in turn shows that the four points of $S$ are in a plane. The proof for $T$ proceeds analogously on a $\Sigma_{1,2}$ patch. 
\end{proof}

\begin{remark}\label{rem:koenigshyperplanesconics}
	Let us point out that if $P$ is an extensive Kœnigs net then in fact $S$ and $T$ arise quite naturally as follows. In \cite[Section~4]{AffolterFairleyKoenigsLaplace} we showed that $P$ is inscribed in a degenerate quadric $\mathcal U$ consisting of two hyperplanes $U_1$ and $U_2$. These two hyperplanes determine a pencil of hyperplanes $(U_s)_{s\in \RP^1}$ containing $U_1 \cap U_2$. Moreover, for each $U_s$ there is a hyperplane $U_t$ in the same pencil such that the cross-ratio $\cro(U_1,U_s,U_2,U_t)$ is $-1$. One can show that
	\begin{align}
		S(i,j) &= (P(i,j) \vee P(i+1,j)) \cap U_s, & T(i,j) &= (P(i,j) \vee P(i,j+1)) \cap U_t.
	\end{align}
\end{remark}

\section{A connection to Kœnigs binets}
\label{sec:binets}

In this section we prove an interesting result that arose naturally in the process of obtaining our main results, but which is otherwise independent from our main results. Hence, the reader interested primarily in the main results may skip this section.
The additional result concerns certain Laplace invariants of Q-nets, which were introduced in \cite{doliwa1997geometricToda}.

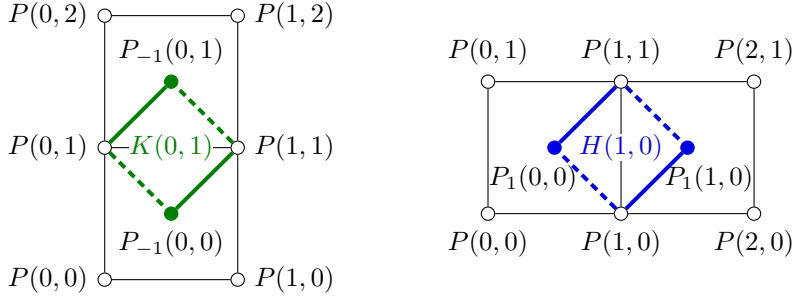
\begin{figure}[tb]
	\centering
	\begin{tikzpicture}[scale=1.75,baseline=(current bounding box.center)]
		\small
		\node[wvert, label={left:$P(0,0)$}] (p00) at (0,0) {};
		\node[wvert, label={left:$P(0,1)$}] (p01) at (0,1) {};
		\node[wvert, label={left:$P(0,2)$}] (p02) at (0,2) {};
		\node[wvert, label={right:$P(1,0)$}] (p10) at (1,0) {};
		\node[wvert, label={right:$P(1,1)$}] (p11) at (1,1) {};
		\node[wvert, label={right:$P(1,2)$}] (p12) at (1,2) {};
		\node[bvert, black!50!green, label={below:$P_{-1}(0,0)$}] (f00) at (0.5,0.5) {};
		\node[bvert, black!50!green, label={above:$P_{-1}(0,1)$}] (f01) at (0.5,1.5) {};			
		\draw[-]
		(p00) -- (p01) -- (p02) -- (p12) -- (p11) -- (p10) -- (p00)
		(p01) edge[] node [fill=white, inner sep=0, text=black!50!green] {$K(0,1)$}  (p11)
		;
		\draw[-, line width=1.5, black!50!green]
		(p01) edge[-] (f01) edge[densely dashed] (f00)
		(p11) edge[-] (f00) edge[densely dashed] (f01)
		;
	\end{tikzpicture}
	\hspace{10mm} 
	\begin{tikzpicture}[scale=1.75,baseline=(current bounding box.center)]
		\small
		\node[wvert, label={below:$P(0,0)$}] (p00) at (0,0) {};
		\node[wvert, label={above:$P(0,1)$}] (p01) at (0,1) {};
		\node[wvert, label={below:$P(2,0)$}] (p20) at (2,0) {};
		\node[wvert, label={below:$P(1,0)$}] (p10) at (1,0) {};
		\node[wvert, label={above:$P(1,1)$}] (p11) at (1,1) {};
		\node[wvert, label={above:$P(2,1)$}] (p21) at (2,1) {};
		\node[bvert, black!10!blue, label={[xshift=14,yshift=-1]below left:$P_{1}(0,0)$}] (f00) at (0.5,0.5) {};
		\node[bvert, black!10!blue, label={[xshift=-14,yshift=-1]below right:$P_{1}(1,0)$}] (f10) at (1.5,0.5) {};			
		\draw[-]
		(p00) -- (p10) -- (p20) -- (p21) -- (p11) -- (p01) -- (p00)
		(p10) edge[] node [fill=white, inner sep=0, text=black!10!blue] {$H(1,0)$}  (p11)
		;
		\draw[-, line width=1.5, black!10!blue]
		(p10) edge[-] (f10) edge[densely dashed] (f00)
		(p11) edge[-] (f00) edge[densely dashed] (f10)
		;
	\end{tikzpicture}	
	\caption{Combinatorial picture of the Laplace invariants $H$ (left) and $K$ (right).}
	\label{fig:combinatorialinvariants}
\end{figure}

\begin{definition}\label{def:discreteLaplaceinvariants}
	Let $P\colon \Z^2  \to \RP^n$ be a non-degenerate Q-net. The Laplace invariants $H\colon \Z^2 \to \R$ and $K\colon \Z^2 \to \R$ are defined as
	\begin{align*}
		H(i,j)&:= \cro(P(i,j), P_{1}(i,j), P(i,j+1), P_{1}(i-1,j)), \\
		 K(i,j) &:= \cro(P(i,j), P_{-1}(i,j) ,P(i+1,j) , P_{-1}(i,j-1)).
	\end{align*}
\end{definition}

As shown in Figure~\ref{fig:laplacetransform}, the points involved in the cross-ratio are on a line. Moreover, as shown in Figure~\ref{fig:combinatorialinvariants}, the Laplace invariant $H(i,j)$ is assigned to the vertical edge joining $(i,j)$ and $(i,j+1)$ whereas $K(i,j)$ is assigned to the horizontal edge joining $(i,j)$ and $(i+1,j)$. The following theorem shows that the pair $S,T$ have equal Laplace invariants per edge.

\begin{figure}[tb] 
	\centering
	\includegraphics[width=0.6\textwidth]{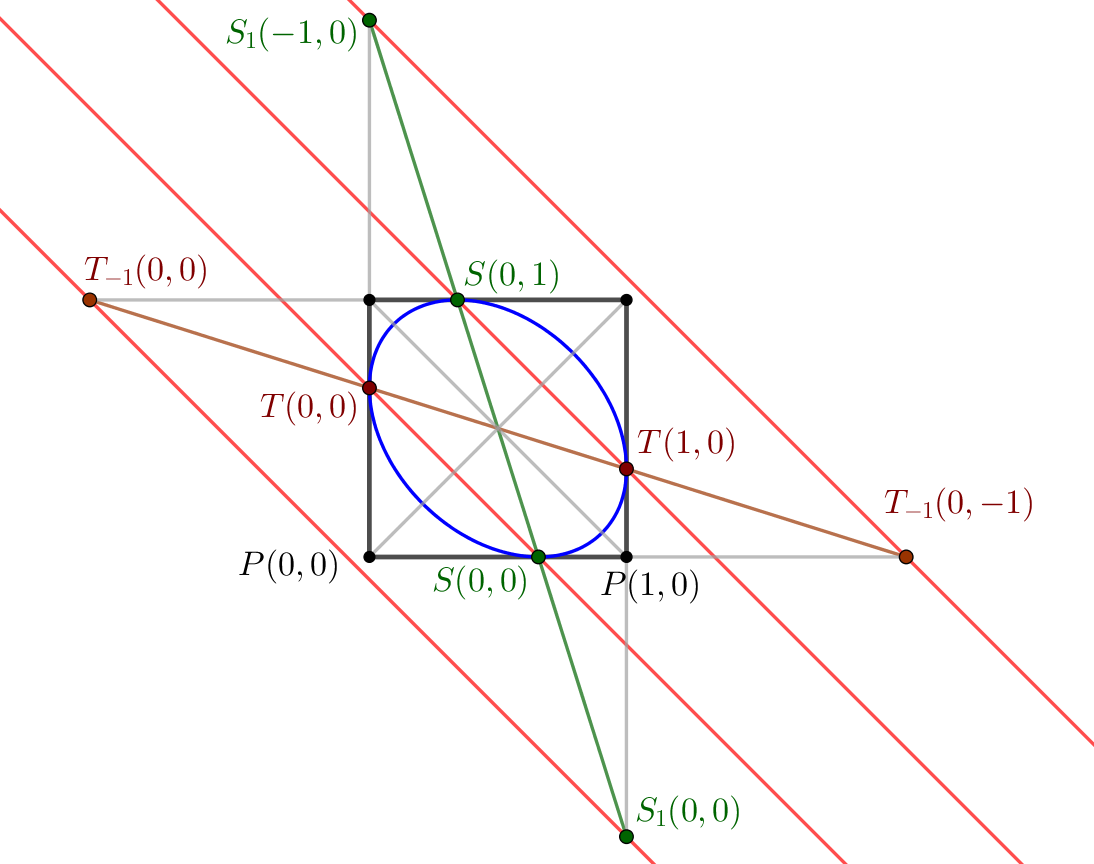}
	\caption{ The four brown points on the line $\pah{T}(0)$ determine $K^T(1,1)$ and the four green points on the line $\pav{S}(0)$ determine $H^S(1,1)$.}
	\label{figure:Koenigsbinetface}
\end{figure}

\begin{theorem}\label{thm:Keonigsbinet}
	Let $P \colon \Sigma \to \RP^n$ be a Kœnigs net with an instance of touching conics. Let $H^{S}, K^S$ and $H^{T}, K^T$ be the Laplace invariants of the Q-nets $S, T$ determined by the touching points. Then, 
	\begin{align}
		H^S(i,j) &= K^T (i,j), &
		H^T(i+1,j) &= K^S(i,j+1).
	\end{align}
\end{theorem}
Before we give a proof, let us briefly discuss the statement of Theorem~\ref{thm:Keonigsbinet}. First of all, the analogous property for Kœnigs nets paired with their diagonal intersection nets was discovered recently in \cite[Theorem~5.15]{AffolterFairleyKoenigsLaplace} and independently in \cite{adtbinets}. This property is also used as a characterization for \emph{Kœnigs binets} in the language of the upcoming publication \cite{adtbinets} -- or the \emph{control net} of a \emph{Kœnigs checkerboard pattern} \cite{dellingercb}. Hence, the pair of touching nets $(S,T)$ provides a new example of such Kœnigs binets.

\begin{proof}[Proof of Theorem~\ref{thm:Keonigsbinet}]
	We need to prove two cases: the case that corresponding Laplace invariants live on edges of $S$ and $T$ that cross each other in a face of $P$, and the case where the edges of $S$ and $T$ cross at a vertex of $P$. We begin with the first case.
	
	Without loss of generality, we can choose an affine image of $\RP^n$ such that $P(i,j), P(i+1,j) , P(i+1,j+1), P(i, j+1)$ are the vertices of a square as shown in Figure~\ref{figure:Koenigsbinetface}. In the proof of Lemma~\ref{lem:touchingpointsQnet} we have observed that the points $S_{1}(i,j), S_{1}(i-1,j), T_{-1}(i,j), T_{-1}(i,j-1)$ are contained in the edge-lines of the square. Moreover, the same lemma shows that the four points on each edge-line have cross-ratio $-1$. Since the square is symmetric about the diagonal $P(i,j) \vee P(i+1,j+1)$, the reflection about this diagonal interchanges
	\begin{align}
		S(i,j) &\leftrightarrow T(i,j), & S_1(i-1,j) &\leftrightarrow T_{-1}(i,j-1),\\  S(i,j+1) &\leftrightarrow T(i+1,j), & S_1(i,j) &\leftrightarrow T_{-1}(i,j).
	\end{align}
	Thus, the two cross-ratios 
	\begin{align}
		H^S(i,j)&:= \cro(S(i,j), S_{1}(i,j), S(i,j+1), S_{1}(i-1,j)), \\
		K^T(i,j) &:= \cro(T(i,j), T_{-1}(i,j), T(i+1,j) , T_{-1}(i,j-1)),
	\end{align}
	are equal as the corresponding two quadruples of points involved are related by a reflection (which is a projective transformation). Alternatively, one can see that, due to the symmetry, point pairs are on parallel lines (red lines in Figure~\ref{figure:Koenigsbinetface}), which also shows that the cross-ratios are equal.
	
	Now for the second case.
	Consider a patch of $P$ as shown in Figure~\ref{figure:Koenigsbinetvertex}. Let $a,b,c,d,x \in \R^{n+1}$ be representative vectors for 
	\begin{align}
		P(i,j+1), P(i+1,j), P(i+2,j+1), P(i+1,j+2), P(i+1,j+1),
	\end{align}
	respectively. We can rescale $a,b,c,d$ such that $[a+b], [b+c], [c+d]$ are the intersection points of the diagonals of the corresponding quads. Then, $[d+a]$ is the intersection point of the diagonals of the last quad because the four diagonal intersection points are coplanar for any Kœnigs net \cite[Theorem 2.26]{BS2008DDGbook}. Strictly speaking, the last argument assumes some genericity. Instead, one may use that the fourth diagonal point is uniquely determined by a multi-ratio condition for the four diagonal points which are on the edge-lines of the quad with vertices $P(i,j+1), P(i+1,j), P(i+2,j+1), P(i+1,j+2)$ \cite[Theorem 2.25]{BS2008DDGbook}. We can still rescale $x$ such that 
    \begin{align}
		S(i,j+1)= [a+x].
	\end{align} 
    Furthermore, we use the observation \cite[Theorem 2.4]{bobenkofairley2021nets} that the two touching points $S(i,j+1)$, $T(i+1,j)$ and the diagonal intersection point $[a+b]$ form a Ceva configuration for the triangle with the vertices 
    \begin{align}
        P(i,j+1)=[a], \quad P(i+1,j+1) = [x], \quad P(i+1,j)=[b].
    \end{align}
    Then, a simple computation shows that $[a+b+x]$ is the concurrency point of the Ceva configuration, and that $T(i+1,j) = [b+x]$. With analogous arguments we obtain that 
	\begin{align}
		S(i+1,j+1) = [c+x], \qquad T(i+1,j+1) = [d+x].
	\end{align}
	Moreover, we may choose $\kappa, \lambda, \mu, \nu \in \R$ such that
	\begin{align}
		P(i,j) &= [a+b+\kappa x],&  P(i+2,j) &= [b+c+\lambda x],\\
		P(i+2,j+2) &= [c+d+\mu x], & P(i,j+2) &= [d+a+\nu x].
	\end{align}
	With all representative vectors now chosen, a direct computation shows that
	\begin{align}
		S_{-1}(i,j) &= [(1-\lambda)(a+x) + (\kappa-1)(c+x)],\\
		S_{-1}(i,j+1)&= [(1-\mu)(a+x) + (\nu-1)(c+x)],\\
		T_1(i,j)&= [(\nu -1)(b+x)+(1-\kappa)(d+x)],\\
		T_{1}(i+1,j)&= [(\mu -1)(b+x)+(1-\lambda)(d+x)].
	\end{align}
	With another computation we find that
	\begin{align}
		K^S(i,j+1) = \cro(S(i,j+1), S_{-1}(i,j+1) ,S(i+1,j+1) , S_{-1}(i,j)) = \frac{\kappa -1 }{\lambda -1} \frac{\mu -1}{\nu -1}.
	\end{align} 
	Finally, a last computation shows that 
	%relative to the basis $b+x, d+x$, we find that $$$$
	\begin{align}
		H^T(i+1,j) =  \cro(T(i+1,j), T_{1}(i+1,j), T(i+1,j+1), T_{1}(i,j)) = \frac{\kappa -1 }{\lambda -1} \frac{\mu -1}{\nu -1}.
	\end{align}
	Therefore, we have shown $K^S(i,j+1) = H^T(i+1,j)$ and thereby the claim of the theorem.
\end{proof}

\begin{figure}[tb] 
	\centering
	\includegraphics[width=0.85\textwidth]{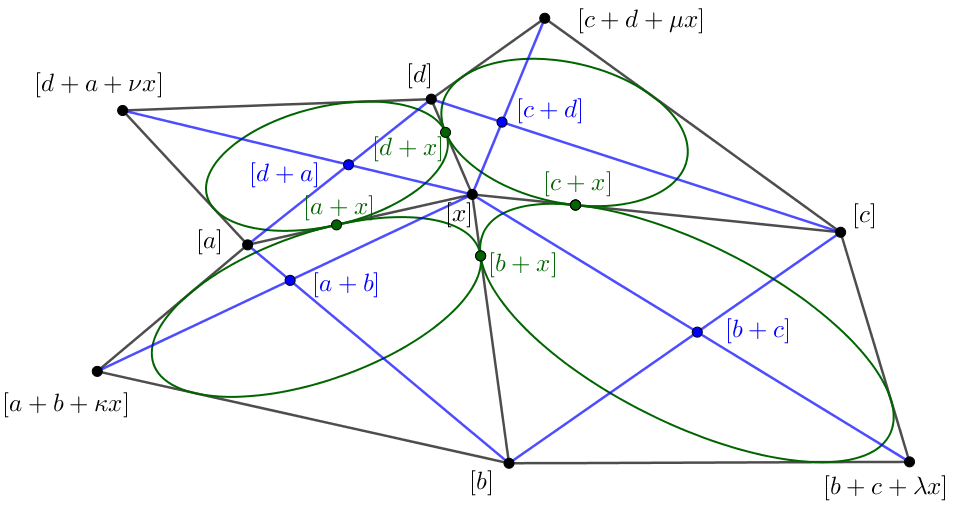}
	\caption{The choice of representative vectors used for the proof of Theorem~\ref{thm:Keonigsbinet}.}  
	\label{figure:Koenigsbinetvertex}
\end{figure}

\begin{remark}
	\emph{S-isothermic nets} \cite{bpdiscsurfaces, bhsminimal} are a discretization of surfaces in isothermic parametrization. One characterization of S-isothermic nets is as Q-nets with touching inscribed circles (there is a later generalization of S-isothermic nets in \cite{BS2008DDGbook} which we do not consider here). Hence, we can consider an S-isothermic net as a Kœnigs net $P$ with touching conics $\mathcal C$ given by the touching inscribed circles. By Lemma~\ref{lem:touchingpointsQnet}, the touching nets $S$ and $T$ are Q-nets. Moreover, we claim that $S$ and $T$ are actually \emph{circular nets}, that is the four points of each face are contained in a circle. To see this, note that two adjacent circles of $\mathcal C$ are contained in a sphere. Consequently, the four points of a quad of $S$ (or $T$) are contained in a plane and a sphere -- and thus in a circle. Combining the observation that $S$ and $T$ are circular nets with Theorem~\ref{thm:Keonigsbinet}, we see that $(S,T)$ are a Kœnigs binet consisting of two circular nets. On the one hand, circular nets are a discretization of \emph{curvature-line parametrizations} \cite{cdscircular, bobenkocircular}. On the other hand, curvature-line parametrizations that are Kœnigs nets are \emph{isothermic parametrizations}. Thus, it appears that the touching nets of S-isothermic nets also constitute a discretization of isothermic parametrizations.
\end{remark}

\section{The quadric gluing lemma} \label{sec:gluing}

Let us briefly revisit the theory of quadrics and pencils of quadrics before giving a gluing lemma (Lemma~\ref{lem:2quadricsdefinepencil}), which will be crucial in Section~\ref{sec:isncribedquadrics}. Recall (from Section~\ref{sec:basics}) that a quadric $\mathcal Q$ in $\RP^n$ is defined as the zero-set of a symmetric bilinear form $\p$, that is
\begin{align}
	\mathcal{Q} := \{[x] \in \RP^n \mid \p(x,x) = 0\}.
\end{align}
If $\p$ defines $\mathcal Q$ in this manner we say $\p$ is a \emph{corresponding bilinear form}. Moreover, if $\p' = \lambda \p$ for some (non-zero) $\lambda \in \R$ then $\p'$ and $\p$ define the same quadric. We say $\p'$ and $\p$ are the same \emph{up to rescaling}. However, even up to rescaling, a quadric does not always define its corresponding bilinear form uniquely. For example, if $\mathcal{Q} = \emptyset$, then two forms representing $\mathcal{Q}$ are not necessarily rescalings of each other. 
Fortunately, the quadrics that we need in the remainder of the paper fall in a special class of quadrics that \emph{do} determine their corresponding bilinear forms uniquely up to rescaling.

\begin{definition}
	A quadric $\mathcal Q \subset \RP^n$ is \emph{full-dimensional} if there is a corresponding bilinear form $\p$ such that the signature of $\p$
	\begin{enumerate}
		\item contains at least one minus and at least one plus sign, or
		\item contains exactly $n$ zeroes.
	\end{enumerate}
\end{definition}

Any full-dimensional quadric $\mathcal Q$ in $\RP^n$ is locally a manifold of dimension $n-1$ at non-singular points. Note that in the second case, the quadric is a hyperplane (of dimension $n-1$) and each point is singular.

\begin{lemma}\label{lem:fulldimquadricdeterminesform}
	A full-dimensional quadric $\mathcal Q$ in $\RP^n$ determines its corresponding bilinear form $\p$ uniquely up to rescaling.
\end{lemma}
\proof{
	If the signature of $\mathcal Q$ contains exactly $n$ zeroes, then the claim is obvious. Which leaves the case that $\mathcal Q$ has at least one minus and at least one plus sign in its signature. In this case, assume that there are two bilinear forms $\p,\p'$ corresponding to $\mathcal Q$. Let us assume that the signature of $\p$ is $(p,q,r)$, and that the vectors
	\begin{align}
		e_1, \dots, e_p; \quad f_1, \dots, f_q; \quad g_1, \dots, g_r;
	\end{align}
	form an orthogonal basis of $\R^{n+1}$ with respect to $\p$, and such that 
	\begin{align}
		\p(e_i,e_i) = 1; \quad \p(f_i,f_i) = -1; \quad \p(g_i, g_i) = 0;
	\end{align}
	for all $i$ where the statements make sense. The projectivization of the span of all $g_i$ vectors is the space of singular points of $\mathcal Q$. Moreover, singular points of a quadric are independent of choice of bilinear form. Therefore,
	\begin{align}
		\p'(g_i,e_j) = \p'(g_i,f_j) = \p'(g_i,g_j) = 0,
	\end{align}	
	for all $i,j$ as well. It remains to consider the non-singular part. Moreover, due to orthogonality we obtain that 
	\begin{align}
		\p(e_i \pm f_j, e_i \pm f_j) = 0,
	\end{align}
	for all $i,j$. Therefore, the corresponding point $[e_i \pm f_j]$ is in $\mathcal Q$. Consequently, $\p'(e_i,f_j) = 0$ for all $i,j$, and
	\begin{align}
		\p'(e_i,e_i) = \lambda = - \p'(f_j,f_j),
	\end{align}
	for all $i,j$ and some non-zero $\lambda \in \R$ (that does not depend on $i,j$). Finally, if $p > 1$ note that $[e_i + e_j + \sqrt2f_k]$ is in $\mathcal Q$ for all $i\neq j$ and $k$. Plugging $[e_i + e_j + \sqrt2f_k]$ into $\p'$ observe that $\p'(e_i, e_j) = 0$ for all $i\neq j$. Analogously, if $q>1$ one obtains that $\p'(f_i, f_j) = 0$ for all $i \neq j$. Hence, $\p' = \lambda \p$.
    \qed
}

Let us add that the proof of Lemma~\ref{lem:fulldimquadricdeterminesform} is essentially a higher-dimensional version of a corresponding proof for conics in unpublished lecture notes of Bobenko, Springborn and Techter.
With this lemma we are ready to introduce pencils of quadrics. Given two different quadrics $\mathcal Q_1$, $\mathcal Q_2$ with corresponding bilinear forms $\p_1, \p_2$
and $T = [t] \in \RP^1$, we define the quadric
\begin{align}
	\mathcal Q_{T} = \{[x] \in \RP^n \mid t_1\p_1(x,x) + t_2 \p_2(x,x) = 0\}.
\end{align}
The set $\{ \mathcal Q_{T} \mid T \in \RP^1 \}$ is called a \emph{pencil (of quadrics)}. If $\mathcal Q_1$ and $\mathcal Q_2$ are both full-dimensional quadrics, then the pencil is uniquely determined by $\mathcal Q_1$ and $\mathcal Q_2$.

We are ready to give the fundamental lemma that we call the ``quadric gluing lemma'', which will be very useful for the remainder of the paper.

\begin{lemma}\label{lem:2quadricsdefinepencil}
	Let $E$ and $F$ be two distinct hyperplanes in $\RP^n$. Let  $\mathcal{Q}_E$ be a full-dimensional quadric in $E$ and $\mathcal{Q}_F$ be a full-dimensional quadric in $F$. Suppose that $\mathcal{Q}_E$ and $\mathcal{Q}_F$ coincide in $E \cap F$ and are full-dimensional in $E \cap F$. Then, there is a unique pencil of quadrics in $\RP^n$ such that all quadrics $\mathcal Q_T$ in the pencil are full-dimensional and coincide in $E \cup F$ with $\mathcal{Q}_E \cup \mathcal{Q}_F$ unless $\mathcal Q_T = E\cup F$.
\end{lemma}
\proof{
	The claim is a special case of \cite[Lemma 3.10]{bobenkofairley2023circularnets}, with the difference that we work with full-dimensional quadrics. Moreover, in \cite{bobenkofairley2023circularnets} quadrics are not viewed as point sets but as pairs of point sets and corresponding bilinear forms. This is not an issue for us since we work with full-dimensional quadrics, in which case the identification between quadrics and corresponding bilinear forms is canonical (due to Lemma~\ref{lem:fulldimquadricdeterminesform}). Hence, the same proof as in \cite{bobenkofairley2023circularnets} can be used to show the existence of the pencil in the claim. Moreover, uniqueness also follows from \cite{bobenkofairley2023circularnets}, since we assume that the restriction of $\mathcal Q_E$ and $\mathcal Q_F$ to $E\cap F$ is full-dimensional (and therefore not isotropic).
	
	It remains to show that every quadric in the pencil is full-dimensional. In the case $Q_T = E \cup F$ the quadric is clearly full-dimensional. If $\mathcal Q_E$ or $\mathcal Q_F$ are full-dimensional in the sense that the signature contains at least one minus and at least one plus sign, then $\mathcal Q$ is obviously also full-dimensional. The remaining case is that $\mathcal Q_E$ and $\mathcal Q_F$ both have signatures that consist of $n-1$ zeros and are thus $(n-2)$-dimensional projective subspaces. Let $\mathcal{Q}_T$ be a quadric in the pencil and let us assume that the signature of $\mathcal Q_T$ does not contain a plus and a minus. Consequently, every point of $\mathcal{Q}_T$ is singular. Additionally, by assumption $\mathcal Q_T$ contains both $\mathcal Q_E$ and $\mathcal Q_F$. Hence, the dimension of $\mathcal Q_T$ is at least $n-2$. However, if the dimension of $\mathcal Q_T$ were $n-2$, then $\mathcal Q_T$ would coincide with both $\mathcal Q_E$ and $\mathcal Q_F$. This would imply that $\mathcal Q_E$ coincides with $\mathcal Q_F$, which is only possible if both quadrics equal $E \cap F$. However, this is a contradiction with the assumption that the restriction of $\mathcal Q_E$ and $\mathcal Q_F$ is full-dimensional in $E\cap F$. Thus, the dimension of $\mathcal Q_T$ must be larger than $n-2$ and therefore $n-1$, which shows that $\mathcal Q_T$ is full-dimensional.\qed
	}

\begin{remark}
	Since we use the proof of \cite[Lemma 3.10]{bobenkofairley2023circularnets} in the proof of Lemma~\ref{lem:2quadricsdefinepencil}, let us point out a minor omission in the proof of \cite[Lemma 3.10]{bobenkofairley2023circularnets}. That proof constructs a $1$-parameter family of symmetric bilinear forms $\{\p_\lambda\}_{\lambda \in \R}$ representing all the quadrics in the pencil -- except $\mathcal{Q}_T = E \cup F$ which was omitted. The extra quadric appears by considering $\p = \p_{\lambda_1} - \p_{\lambda_2}$ for $\lambda_1 \neq \lambda_2$. 
\end{remark}

We use Lemma~\ref{lem:2quadricsdefinepencil} throughout the remainder of the paper to ``glue'' quadrics in smaller dimensional spaces together so as to obtain quadrics in larger dimensions. In all cases, the smallest quadrics used have signatures
\begin{align}
	\texttt{(+0)}, \texttt{(-0)}, \texttt{(++-)}, \texttt{(+--)}, \texttt{(+00)} \mbox{ or }  \texttt{(-00)},
\end{align} 
which are all full-dimensional. Thus, by using Lemma~\ref{lem:2quadricsdefinepencil} the larger quadrics are also full-dimensional. To avoid overloading the proofs with technicalities, we will not discuss the full-dimensionality of the quadrics involved in the remainder of the paper.
\section{Inscribed quadrics}\label{sec:isncribedquadrics}

In this section we investigate inscribed quadrics for finite $\Sigma_{a,b}$ patches of extensive Kœnigs nets, and prove our first main result. We also pave the way for later sections by observing certain properties of the singular points of inscribed quadrics.

\begin{figure}[tb] 
	\centering
	\includegraphics[height=0.4\textwidth]{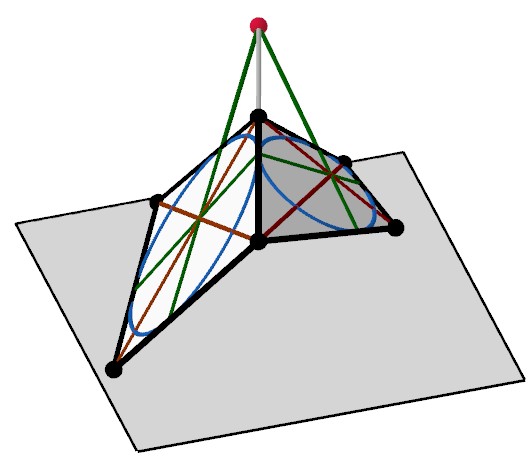}
	\includegraphics[height=0.4\textwidth]{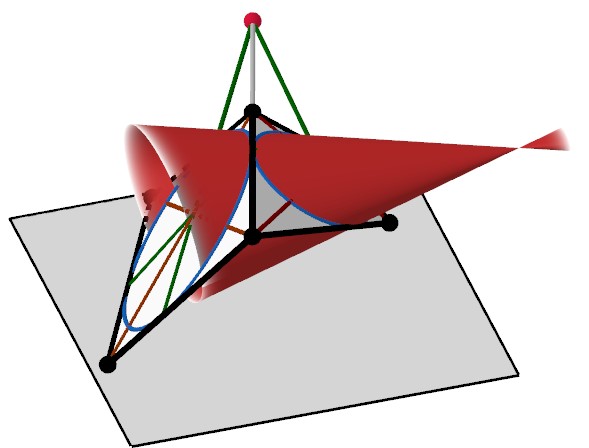}
	\caption{Two planar quads in $\RP^3$ with an instance of touching conics (blue).
         In this case the inscribed quadric $\mathcal{Q}$ (red) is a quadratic cone that contains the touching conics and the two lines $\pah{S}(j)$. The apex of $\mathcal{Q}$ is the point $S_{-1}(0,0)$.}
	\label{figure:conelemma}
\end{figure}

\subsection{Existence and Uniqueness}

Recall that $S$ and $T$ are the nets made up of the contact points of touching conics $\mathcal C$ of a Kœnigs net $P$. If $P$ is defined on a $\Sigma_{a,b}$ patch, then \emph{generically} we expect that the parameter spaces $\pah S(j)$ have dimension $a-1$, and $\pav T(i)$ have dimension $b-1$. The next lemma shows that if $P$ is extensive, then this is \emph{always} the case.

\begin{lemma}\label{lem:dimensiontspaces}
	Let $P\colon \Sigma_{a,b} \to \RP^n$ be an extensive Kœnigs net with touching conics $\mathcal C$. Then, each $\pah S(j)$ is $(a-1)$-dimensional and each $\pav T(i)$ is $(b-1)$-dimensional.
\end{lemma}
\proof{

	We begin by proving the claim for $T(0)$. Recall that for touching conics we do not allow the conics to consist of the diagonals of $P$. Therefore, the touching point $T(0,j)$ does not coincide with the vertices $P(0,j)$ or $P(0,j+1)$ of the corresponding edge. Then, $P(0,0) \vee \pav T(0)$ contains $P(0,1)$ because $P(0,0) \vee T(0,0)$ is a line that contains $P(0,1)$. Iterating further, we see that $P(0,0) \vee \pav T(0) = \pav P(0)$. From the extensivity of $P$, it follows that the dimension of $\pav P(0)$ is $b$. Consequently, the dimension of $\pav T(0)$ is $b-1$. The corresponding claims for all $\pav T(i)$ and $\pah S(j)$ follow by analogous arguments.
    \qed
}

We now give a proof of our first main result, Theorem~\ref{thm:BSquadriclocalpatch}, which states that for extensive Kœnigs nets with a choice of touching conics $\mathcal C$ there is a unique inscribed quadric $\mathcal Q$ (see Definition~\ref{def:inscribedquadric}).

\begin{proof}[Proof of Theorem~\ref{thm:BSquadriclocalpatch}]
	The proof proceeds by induction over $a+b$. The base case $\Sigma_{1,1}$ is evident because $\mathcal{Q}$ is (and must be equal to) the inscribed conic $\mathcal C(0,0)$, which contains the tangency points
	\begin{align}
		T(i,0) = \pav{T}(i), \quad S(0,j) = \pah{S}(j),
	\end{align}
	respectively. Next, we prove the $\Sigma_{a,b}$ case, without loss of generality we assume $b \geq 2$ (the $\Sigma_{a,1}$ cases follow by symmetry from the $\Sigma_{1,b}$ cases).
	We consider the two restrictions $\South P$ and $\North P$ of $P$ to the $\Sigma_{a,b-1}$ patches
	\begin{align}
		\Sigma_a \times \{0,1,\dots,b-1\}  \quad \mbox{and} \quad \Sigma_a \times \{1,2,\dots,b\},
	\end{align}
	respectively. We also denote by $\NSmid P$ the restriction of $P$ to the $\Sigma_{a,b-2}$ patch
	\begin{align}
		\Sigma_a \times \{1,\dots,b-1\}.
	\end{align}
 	By induction, there are two unique inscribed quadrics $\North \mathcal Q$, $\South \mathcal Q$ for $\North P$ and $\South P$ respectively.
	Moreover, $\North \mathcal Q$ and $\South \mathcal Q$ agree in the join of the vertices of the image of $\NSmid P$, which we denote $\NSmid \pa P$. Indeed, $\North \mathcal Q \cap \South \mathcal Q$ is the unique inscribed quadric for $\NSmid P$, provided $b>2$. For $b=2$, we also have that $\North \mathcal Q$ and $\South \mathcal Q$ agree on $\NSmid \pa P$ because then $\North \mathcal{Q} \cap\South \mathcal{Q}  = \pah{S}(1)$. For all $b\geq 2$, Lemma~\ref{lem:2quadricsdefinepencil} implies that there is a unique pencil of quadrics containing both $\North \mathcal Q$ and $\South \mathcal Q$. The inscribed quadric $\mathcal{Q}$ that we are looking for must be in this pencil.
	
	Due to the extensivity of $P$ the dimensions of both $\North \pav{P}(0)$ and $\South \pav{P}(0)$ are $b-1$, and the dimension of $\pav P(0)$ is $b$. Additionally, due to Lemma~\ref{lem:dimensiontspaces} the dimension of both $\North \pav T(0)$  and $\South \pav T(0)$ is $b-2$, and the dimension of $\pav T(0)$ is $b-1$.
	Therefore, there is a point $Y \in \pav T(0)$ that is neither in $\North \pav T(0)$ nor $\South \pav T(0)$. Furthermore, $Y$ is neither in $\North \pav{P}(0)$ nor $\South \pav{P}(0)$. Consequently, there is a unique quadric $\mathcal{Q}$ in the pencil of quadrics that contains the point $Y$. 
	By induction, we get the inclusions
	\begin{align}
		\North \pav{T}(0) \subset \North \mathcal{Q} \subset \mathcal Q, \quad \South \pav{T}(0) \subset \South \mathcal{Q} \subset \mathcal Q.
	\end{align}
	Hence, $\pav T(0)$ contains two different isotropic codimension 1 spaces and an additional distinct point $Y \in \mathcal Q$, which implies that $\pav T(0) \subset \mathcal Q$.
	
	It remains to show that $\mathcal Q$ also contains the spaces $\pav T(i)$ for $i > 0$. We show this by induction over $i$, with the base case already proven. Without loss of generality, we assume $i=1$. Let us consider the lines
	\begin{align}
		\ell(j) := P_1(0,j) \vee P_{-1}(0,j).
	\end{align}
	We start with the case $a=1$, $b=2$, for which we distinguish three cases:
	\begin{enumerate}

    		\item Suppose $\mathcal{C}(0,0)$ and $\mathcal{C}(0,1)$ are both degenerate. Then, $\mathcal{C}(0,j)$ is the double line $\ell(j)$ for $j =0,1$. It follows that 
        \begin{align}
        P_{1}(0,0) = T(0,0) = T(1,0), \qquad P_{1}(0,1) = T(0,1) = T(1,1).
        \end{align}               
            Then, $\pav T(0) = \pav T(1)$ and thus $\pav T(1)$ is also contained in $\mathcal{Q}$.          
		
        \item Suppose one of the conics $\mathcal{C}(0,0)$ and $\mathcal{C}(0,1)$ is degenerate and the other is not. Without loss of generality, suppose that $\mathcal{C}(0,1)$ is degenerate 
        and thus equals the line $\ell(1)$. Then, $T(0,1)=T(1,1)$ and this point also equals the Laplace point
		\begin{align}\label{eq:rgikbewg}
			T_1(0,0) \coloneq \pav{T}(0) \cap \pav{T}(1).
		\end{align}
        The point $T_{1}(0,0)$
		is conjugate to the plane $P(0,1)\vee P(1,1) \vee P(0,2)$ because $T_1(0,0)$ is a singular point of $\mathcal C(0,1)$. 
		The point $T_1(0,0)$ is also conjugate to $T(0,0)$ since 
        \begin{align}
			\pav T(0) = T_1(0,0) \vee T(0,0)
		\end{align}
		is contained in $\mathcal{Q}$. So, $T_1(0,0)$ is conjugate to the join of $P(0,1)\vee P(1,1) \vee P(0,2)$ and $T(0,0)$, which equals $\pav{P}(0) \vee \pav{P}(1)$ since $P$ is extensive. In other words, $T_1(0,0)$ is a singular point of $\mathcal{Q}$. Therefore, $\mathcal{Q}$ is the quadratic cone containing $\mathcal{C}(0,0)$ and with apex $T_1(0,0)$. As $T(1,0)$ is in $\mathcal C(0,0)$ and $T_1(0,0)$ is singular, the line $\pav{T}(1) = T_1(0,0) \vee T(1,0)$ is also contained in $\mathcal{Q}$. 		
		
		\item Suppose $\mathcal{C}(0,0)$ and $\mathcal{C}(0,1)$ are both non-degenerate. Then, the Laplace point $T_{1}(0,0):= \pav{T}(0) \cap \pav{T}(1)$  is contained in $\mathcal{Q}$, and is distinct from $T(1,0)$ and $T(1,1)$. It follows that $\pav T(1)$ is contained in $\mathcal{Q}$ because it contains three different collinear points in $\pav T(1)$. 
	\end{enumerate}
Therefore, we have proven the case $b=2$. Now suppose that $b > 2$. Without loss of generality, we can assume that
	\begin{align} \label{eq:rgergerg}
	    \pav T(0) \neq \pav T(1),
	\end{align}
	because otherwise $\pav T(1) \subset \mathcal{Q}$ since $\pav T(0) \subset \mathcal{Q}$.
	Suppose that
	\begin{align}\label{eq:tehrh}
		 U:= \pav T(0) \cap \pav T(1) \subseteq \North \pav T(1) \cup \South \pav T(1).
	\end{align} 
	Since $U$, $\North \pav T(1)$ and $\South \pav T(1)$ are $(b-2)$-dimensional projective spaces, \eqref{eq:tehrh} implies that $U$ equals $ \North \pav T(1)$ or $ \South \pav T(1)$. Without loss of generality, suppose that $U = \North \pav T(1)$. By considering the restriction of this equality to the join of the vertices of the image of $\North P$, we obtain $\North \pav T(0) = \North \pav T(1)$. Then, $T(0,j) = T(1,j)$ for all $j=1, 
    \ldots, b-1$. This can only happen if the conics $\mathcal{C}(0,j)$ are degenerate for $j=1, \ldots, b-1$. From the assumption \eqref{eq:rgergerg}, we deduce that $T(0,0) \neq T(0,1)$ so that $\mathcal{C}(0,0)$ is non-degenerate. Then, from case 2 (in the list above), we obtain that $T(1,0) \vee T(1,1)$ is contained in $\mathcal{Q}$. But we also know that $\North \pav T(0) = \North \pav T(1)$ is contained in $\mathcal{Q}$. Therefore, 

    \begin{align}
          \pav T(1) = T(1,0) \vee \North \pav T(1)
    \end{align}
    is contained in $\mathcal{Q}$.
    It only remains to consider the case 
    \begin{align}\label{eq:rtheh}
			U \nsubseteq \North \pav T(1) \cup \South \pav T(1). 
	\end{align} 
	By induction (over $i$), we know that $\pav T(1)$ contains the two isotropic $(b-2)$-dimensional spaces $\North \pav T(1)$ and $\South \pav T(1)$. By (\ref{eq:rtheh}), $\pav T(1)$ contains additional points from $U$ that are also contained in $\mathcal{Q}$. Therefore, $\pav T(1)$ must be an isotropic subspace of $\mathcal Q$. 
\end{proof}

\begin{remark}
It can happen that the inscribed quadric $\mathcal{Q}$ in Theorem~\ref{thm:BSquadriclocalpatch} is degenerate. For example, suppose that $P_1$ and $P_{-1}$ are Laplace degenerate. Then, there is an instance of touching conics such that each inscribed conic is degenerate and $\mathcal{Q}$ is the (double) hyperplane joined by the vertices of $P_1$ and $P_{-1}$.
\end{remark}

\begin{remark}
	As discussed in Remark~\ref{rem:degconics}, every Q-net $P$ has two instances of touching conics $\mathcal C_1$, $\mathcal C_2$ consisting only of diagonals (which we usually exclude). Moreover, if $P$ is an extensive Kœnigs net, then it follows from \cite[Lemma~4.3]{AffolterFairleyKoenigsLaplace} that there are two hyperplanes $U_1,U_2$ such that $\mathcal C_1$ is contained in $U_1$ and $\mathcal C_2$ is contained in $U_2$. Each hyperplane can be seen as a degenerate inscribed quadric.
\end{remark}

Now that we have established the existence and uniqueness of inscribed quadrics $\mathcal Q$, we justify the term ``inscribed'' a bit more. With the next lemma, we show that the parameter spaces $\pah P(j)$ and $\pav P(i)$ are tangent to $\mathcal Q$. We also provide an example in Figure~\ref{figure:conelemma} for the $\Sigma_{2,1}$ case, where we see that each plane $\pah{P}(j)$ is tangent to $\mathcal{Q}$ along the line $\pah{S}(j)$. Analogously, each line $\pav{P}(i)$ is tangent to $\mathcal{Q}$ at the point $\pav{T}(i)$ (which is just $T(i,0)$).

\begin{lemma}\label{lem:BSquadriclocalpatch}
	Let $P\colon \Sigma_{a,b} \to \RP^{a+b}$ be an extensive Kœnigs net with an inscribed quadric $\mathcal{Q}$ with respect to touching conics $\mathcal C$. Then,
	\begin{enumerate}
		\item for all $i\in \Sigma_{a}$, $\pav{P}(i)$ is tangent to $\mathcal{Q}$ along $\pav{T}(i)$,
		\item for all $j\in \Sigma_{b}$, $\pah{P}(j)$ is tangent to $\mathcal{Q}$ along $\pah{S}(j)$.
	\end{enumerate}
\end{lemma}

\proof{
	We do induction over $a+b$. The base case $\Sigma_{1,1}$ is evident because then $\mathcal{Q}$ is the inscribed conic $\mathcal{C}(0,0)$. 
	 As before, we use $\South P$, $\North P$ and $\NSmid P$ to denote the restrictions of $P$ to 
	\begin{align}
		\Sigma_a \times \{0,1,\dots,b-1\},\quad \Sigma_a \times \{1,2,\dots,b\} \quad \mbox{and} \quad \Sigma_a \times \{1,\dots,b-1\}
	\end{align}
	respectively. Due to the Lemma~\ref{lem:dimensiontspaces}, we observe that 
	\begin{align}
		\pav{P}(i) = \NSmid \pav P(i) \vee \pav{T}(i).
	\end{align}
	Therefore, to show that $\pav{P}(i)$ is conjugate to $\pav{T}(i)$ it suffices to show that both $\NSmid \pav P(i)$ and $\pav{T}(i)$ are conjugate to $\pav{T}(i)$. As $\pav{T}(i)$ is contained in $\mathcal{Q}$ it follows that $\pav{T}(i)$ is conjugate to itself. Moreover, by induction we know that $\South \pav P(i)$ is conjugate to $\South \pav{T}(i)$ with respect to $\South \mathcal Q$. Since $\South \mathcal  Q \subset \mathcal{Q}$ we also know that $\South \pav P(i)$ is conjugate to $\South \pav{T}(i)$ with respect to $\mathcal Q$.
	Analogously, $\North \pav P(i)$ is conjugate to $\North \pav{T}(i)$ with respect to $\mathcal{Q}$. Therefore, the intersection
	\begin{align}
		\NSmid \pav P(i) = \North \pav P(i) \cap \South \pav P(i)
	\end{align}
	is conjugate to the join
	\begin{align}
		\pav{T}(i) = \South \pav T(i) \vee \North \pav T(i).
	\end{align}
	In conclusion, $\pav{P}(i)$ is conjugate to $\pav{T}(i)$. Using symmetric arguments, one shows that each $\pah{P}(j)$ is conjugate to $\pah{S}(j)$.
	 \qed
}

\begin{remark}\label{rem:twicetangetnconics}
    For our proof of the existence and uniqueness of inscribed quadrics, it was essential to only consider K{\oe}nigs nets that are extensive. In this remark we provide an example of what can be shown about inscribed quadrics when K{\oe}nigs nets are not extensive. For simplicity, this remark only considers a K{\oe}nigs nets $P \colon \Sigma_{2,2} \to \RP^3$. However, similar observations can be made more generally about K{\oe}nigs nets $P \colon \Sigma_{a,b} \to \RP^n$ with $n < a+b$.
    Let $P\colon \Sigma_{2,2} \to \RP^3$ be a K{\oe}nigs net with an instance of touching conics $\mathcal{C}$. Since the codomain is $\RP^3$, $P$ is not extensive. Therefore, Theorem~\ref{thm:BSquadriclocalpatch} is not applicable. Typically, there does not exist any quadric that contains the four conics $\mathcal{C}$. Instead, as we will explain, there exists a quadric $\mathcal{Q}$ that is tangent to the planes $\pah{P}(j)$, $\pav{P}(i)$ and such that the conics $\mathcal{C}$ are twice tangent to $\mathcal{Q}$.

    Consider an extensive lift $\tilde P: \Sigma_{2,2} \rightarrow \RP^4$ of $P$, which exists due to Lemma~\ref{lem:extensive}. This means that there is a central projection $\pi$ such that $P = \pi \circ \tilde P$. Let $\tilde{\mathcal{C}}$ be the touching conics of $\tilde{P}$ that project via $\pi$ to $\mathcal{C}.$ By Theorem~\ref{thm:BSquadriclocalpatch}, the touching conics $\tilde{\mathcal{C}}$ determine an inscribed quadric $\tilde{ \mathcal{Q}}$. As we will see in Section~\ref{sec:singular}, $\tilde{ \mathcal{Q}}$ is typically non-degenerate. Hence, let us assume that $\tilde{ \mathcal{Q}}$ is non-degenerate. Let $\mathcal{R}:= Z^\perp \cap \mathcal{\tilde Q}$, where $Z$ is the point that is the center of the projection $\pi$ and $Z^\perp$ is the polar hyperplane relative to $\mathcal{\tilde Q}$. Let $\mathcal{G}$ be the quadratic cone that is the join of $\mathcal{R}$ and $Z$. Then, each $\pah{\tilde{P}}(j) \vee Z$ is tangent to $\mathcal{G}$ along the line 
	\begin{align}
		k(j) := (\pah{\tilde{S}}(j) \cap Z^\perp) \vee Z.	
	\end{align}
	The intersection of $\mathcal{G}$ with $\RP^3$ is the required quadric $\mathcal Q$ in $\RP^3$. Indeed, each $\pah{P}(j)$ is tangent at the point where $k(j)$ intersects $\RP^3$ because $\pah{P}(j)$ equals the intersection of $\pah{\tilde{P}}(j) \vee Z$ and $\RP^3$. Symmetrically, each $\pav{P}(i)$ is also tangent to the quadric $\mathcal Q$. It only remains to explain that all the conics $\mathcal{C}$ are twice tangent to $\mathcal{Q}$. For this we observe that the plane of each $\mathcal{\tilde C}(i,j)$ intersects $Z^\perp$ in a line. This line intersects $\mathcal{R}$ in two points, which can be real or imaginary. The projection via $\pi$ of the two points are the two tangency points of $\mathcal{C}(i,j)$ with $\mathcal{Q}$.
\end{remark}

\subsection{Singular points} \label{sec:singular}

In this subsection, we study the locus of singular points of inscribed quadrics. We prove two lemmas that are essential for Section~\ref{sec:constrained}. Moreover, the two lemmas demonstrate that for extensive K{\oe}nigs nets with domain $\Sigma_{a,a}$ inscribed quadrics $\mathcal Q$ are typically non-degenerate.

\begin{lemma}\label{lem:singularpointsinscribedquadric}
	Let $P \colon \Sigma_{a,b} \to \RP^{a+b}$ be an extensive Kœnigs net with an instance of touching conics $\mathcal C$ and the corresponding inscribed quadric $\mathcal{Q}$. Then, any point in
	\begin{align}
		X &:= \bigcap_{j\in \Sigma_b} \pah{S}(j), & Y &:= \bigcap_{i\in \Sigma_a} \pav{T}(i),
	\end{align}
	is a singular point of $\mathcal{Q}$. Moreover,  $X$ is at least $(a-b-1)$-dimensional and $Y$ is at least $(b-a-1)$-dimensional.
\end{lemma}

\begin{proof} 
	Let $A$ be a point in $X$. By Lemma~\ref{lem:BSquadriclocalpatch}, $A$ is conjugate to each $\pah{P}(j)$. Due to the extensivity of $P$ we have
	\begin{align}
		\RP^{a+b} = \bigvee_{j \in \Sigma_b} \pah{P}(j).
	\end{align}
	Consequently, $A$ is a singular point of $\mathcal{Q}$ because it is conjugate to all of $\RP^{a+b}$. Analogously, any point in $Y$ is a singular point.

	By Lemma~\ref{lem:dimensiontspaces}, we have that each $\pah{S}(j)$ is $(a-1)$-dimensional. Since $S$ is a Q-net by Lemma~\ref{lem:touchingpointsQnet}, we have that the join of $\pah{S}(j)$ and $\pah{S}(j+1)$ is at most $a$-dimensional. Equivalently, the intersection $\pah{S}(j) \cap \pah{S}(j+1)$
	is at least $(a-2)$-dimensional. Taking the intersection of three consecutive $\pah{S}(j)$ spaces, we may use the same arguments to see that their join is at most $(a+1)$-dimensional and their intersection is at least $(a-3)$-dimensional. Iterating further, we obtain that $X$ is at least $(a-b-1)$-dimensional. By analogous arguments, we obtain that $Y$ is at least $(b-a-1)$-dimensional.
\end{proof}

If $a > b$, then $X$ in Lemma~\ref{lem:singularpointsinscribedquadric} is non-empty and thus $\mathcal Q$ is degenerate. Analogously, $\mathcal Q$ is also degenerate if $a <b$. Therefore, every inscribed quadric of a Kœnigs net defined on a non-square patch is degenerate. For square patches the situation is different. The next lemma provides us with some insights in that regard.

\begin{lemma}\label{lem:singularpointsinscribedquadricnew}
Let $P \colon \Sigma_{a,b} \to \RP^{a+b}$ be an extensive Kœnigs net with an instance of touching conics $\mathcal C$ and the corresponding inscribed quadric $\mathcal{Q}$. Suppose that $a \geq b$. If there is a singular point of $\mathcal{Q}$ that is not contained in	
\begin{align}
		X &:= \bigcap_{j\in \Sigma_b} \pah{S}(j),
	\end{align} 
then $X$ is at least $(a-b)$-dimensional.
\end{lemma}

\begin{proof}
    We will prove the claim by induction on $b$. We start with the case $b=1$, so the ambient space is $\RP^{a+1}$. Let $A$ be a singular point of $\mathcal{Q}$ that is not contained in $X$. Then, $A \vee \pah S (0)$ is an $a$-dimensional space since $\pah S (0)$ is $(a-1)$-dimensional (by Lemma~\ref{lem:dimensiontspaces}) and does not contain $A$ by assumption. Moreover, $A \vee \pah S (0)$ is contained in $\mathcal{Q}$ because $\pah S (0)$ is contained in $\mathcal{Q}$ and $A$ is a singular point of $\mathcal{Q}$. Therefore, $\mathcal{Q}$ contains a hyperplane, which implies that $\mathcal Q$ is a double hyperplane or the union of two hyperplanes. Consequently, every inscribed conic $\mathcal{C}(i,0)$ is degenerate -- which means that each conic is the (double) line joining the two Laplace points of the corresponding quad.  In particular, $P_{-1}(i,0)=S(i,0) = S(i,1)$ for all $i$. Thus, $\pah{S}(0) = \pah{S}(1)$. Hence, $X$ is $(a-1)$-dimensional. Therefore, $X$ is indeed at least $(a-b)$-dimensional when $b=1$.

    Now suppose that $b \geq 2$. Consider the two restrictions $\South P$ and $\North P$ of $P$ to the $\Sigma_{a,b-1}$ patches 
	\begin{align}
		\Sigma_a \times \{0,1,\dots,b-1\}  \quad \mbox{and} \quad \Sigma_a \times \{1,2,\dots,b\},
	\end{align} 
	respectively. Let $\North \pa P$ be the join of the image of $\North P$ and let $\South \pa P$ be the join of the image of $\South P$. Define
    \begin{align}
		\South X \coloneq \bigcap_{j=0}^{b-1}\pah{S}(j) \quad \mbox{and} \quad  \North X \coloneq \bigcap_{j=1}^{b}\pah{S}(j).
	\end{align}
	
    Let $A$ be a singular point of $\mathcal{Q}$ that is not contained in $X$. We consider two cases. First, suppose that $A \in \South \pa P$. Then, by applying the induction hypothesis to $\South P$, $\South X$ has dimension at least $a-b+1$. Moreover, consider the space
    \begin{align}
    	U = \pah{S}(b-1) \vee \pah{S}(b).
    \end{align}
    As $S$ is a Q-net, the dimension of $U$ is at most one larger than the dimension of $\pah{S}(b)$. Also, we have that
    \begin{align}
    	X = \South X \cap \pah{S}(b).
    \end{align}
    Since both $\South X$ and $\pah{S}(b)$ are both contained in $U$, we get the inequality
    \begin{align}
        \dim X \geq \dim \South X - 1.
    \end{align}
    Hence, $X$ has dimension at least $a-b$.

    For the second case, suppose that $A \notin \South \pa P$. Moreover, by applying Lemma~\ref{lem:singularpointsinscribedquadric} to $\South P$, we have that $\South X$ is conjugate to $\South \pa P$. However, $\South X$ is also conjugate to $A \notin \South \pa P$ since $A$ is a singular point. Therefore, $\South X$ is conjugate to $A \vee  \South \pa P$ which is equal to all of  $\RP^{a+b}$. Hence, every point in $\South X$ is a singular point of $\mathcal{Q}$.

	Let us distinguish two cases. First, we consider the case that $\South X$ is not contained in $\North X$. Consequently, there is a singular point of $\South X$ that is not contained in $\North X$. Moreover, this singular point is contained in $\North P^\vee$ because $b \geq 2$. So, by applying the induction hypothesis to $\North P$, it follows that $\North X$ is at least $(a-b+1)$-dimensional. Then, using a similar dimension counting argument as above, we see that $X$ is at least $(a-b)$-dimensional.
	
	The second case is that $\South X$ is contained in $\North  X$. Then $X = \South X \cap \North X$ is at least $(a-b)$-dimensional since $\South X$ is at least $(a-b)$-dimensional.
\end{proof}

Consider a Kœnigs net defined on a square patch, that is on $\Sigma_{a,b}$ with $a=b$. In this case, Lemma~\ref{lem:singularpointsinscribedquadric} states that the dimension of $X$ has at least dimension $-1$, which is a void statement. However, if $X$ is a point or a larger dimensional space, then this implies that the $\pah S(j)$ spaces have a non-empty intersection. It is not hard to see that this puts constraints on the choice of $\mathcal C$ relative to $P$. In this sense, if $X$ is non-empty then $\mathcal C$ is a ``non-generic'' choice of touching conics. Moreover, if $X$ is empty, Lemma~\ref{lem:singularpointsinscribedquadricnew} shows that there is no singular point of $\mathcal Q$ that is not contained in $X$ -- and therefore $\mathcal Q$ has no singular points. Hence, generically $\mathcal Q$ is non-degenerate for square patches. Although we think it would be interesting to know how many instances of non-generic touching conics exist for a given Kœnigs net, this is not relevant for our main results.

\subsection{Diagonal intersection nets}

In this subsection we introduce diagonal intersections nets and we investigate how the Laplace transforms of the diagonal intersection nets of K{\oe}nigs nets are related to inscribed quadrics.

\begin{definition}\label{defn:diagintersectionnet}
	Let $P\colon \Sigma_{a,b} \to \RP^n$ be a non-degenerate Q-net. The \emph{diagonal intersection net $D$ of $P$} is the map
	\begin{align}
		D\colon  \Sigma_{a-1,b-1} \rightarrow \RP^n, \quad (i,j) \mapsto \big(P(i,j) \vee P(i+1,j+1)\big) \cap \big(P(i+1,j)\vee P(i,j+1)\big).
	\end{align}	
\end{definition}

For any Kœnigs net $P$, the diagonal intersection net is a Q-net \cite[Theorem 2.26]{BS2008DDGbook}. Thus, we may consider the Laplace sequence of $D$.

\begin{lemma}\label{lem:concurrencyDST}
	Let $P \colon \Sigma_{2,2} \to \RP^n$ be a Kœnigs net with touching conics $\mathcal{C}$. Let the diagonal intersection net $D\colon \Sigma_{1,1} \to \RP^n$ of $P$ be non-degenerate. Then,
	\begin{align}
		D_1(0,0)  \in \pav{T}(1) \quad \mbox{and} \quad D_{-1}(0,0)  \in \pah{S}(1).
	\end{align}
\end{lemma}

\begin{proof}
	As already explained in the proof of Theorem~\ref{thm:Keonigsbinet}, there exist $a,b,c,d,x$ such that the labeled points in Figure~\ref{figure:Koenigsbinetvertex} have representative vectors as shown in Figure~\ref{figure:Koenigsbinetvertex}. Then, $D_{1}(0,0) = [b+d]$ is contained in the line  
	\begin{align}
		\pav{T}(1) = [b+x] \vee [d+x].	
	\end{align}
	Symmetrically, $D_{-1}(0,0) = [a+c]$ is contained in the line  
	\begin{align}
		\pah{S}(1) = [a+x] \vee [c+x].
	\end{align}
\end{proof}

\begin{corollary}\label{cor:DnetofBSnetandQuadric}
	Let $P \colon \Sigma_{a,b} \to \RP^n$  be an extensive Kœnigs net with touching conics $\mathcal{C}$ and corresponding inscribed quadric $\mathcal{Q}$. Let $D$ be the diagonal intersection net of $P$. Then, for all $k >0$ such that $D_{\pm k}$ exists,
	\begin{enumerate}
		\item $\pav{{(D_k)}}(i) \subset \mathcal{Q}$ for all $i$,
		\item $\pah{{(D_{-k})}}(j) \subset \mathcal{Q}$ for all $j$.	
	\end{enumerate}
\end{corollary}

\begin{proof}
	By Lemma~\ref{lem:concurrencyDST} the point $D_1(i,j)$ is contained in $\pav{T}(i+1) \subset \mathcal{Q}$. This shows that each $\pav{{(D_1)}}(i)$ is contained in $\mathcal{Q}$. Moreover, by definition of the Laplace transformation we have the inclusion
	\begin{align} 
		\pav{{(D_k)}}(i) \subset \pav{{(D_{k-1})}}(i) \cap \pav{{(D_{k-1})}}(i+1).
	\end{align}
	Therefore, via induction over $k$ we conclude that $\pav{{(D_k)}}(i)$ is contained in $\mathcal{Q}$ for all $k$. The $-k$ cases follow by symmetry.
\end{proof}

\begin{remark} \label{rem:doliwaconic}
	In \cite{doliwa2003}, a Q-net $D$ is called a \emph{Kœnigs lattice} if for every $i,j$ there is a conic $\mathcal K(i,j)$ through the six Laplace points
	\begin{align}\label{eq:gfqibhzeg}
		D_1(i,j-1),\ D_1(i,j),\ D_1(i,j+1),\ D_{-1}(i-1,j),\ D_{-1}(i,j),\ D_{-1}(i+1,j).
	\end{align}
    To be clear, Kœnigs lattices are not equivalent to the (discrete) Kœnigs nets that we considered so far in this article. Clearly, the six points of \eqref{eq:gfqibhzeg} are in the plane joined by $D(i,j)$, $D(i+1,j)$, $D(i,j+1)$. Consider the situation that $D$ is the diagonal intersection net of an extensive Kœnigs net $P$ with inscribed quadric $\mathcal Q$ determined by any instance of touching conics. Then, it is known that $D$ is a K{\oe}nigs lattice \cite{BS2008DDGbook}. By Corollary~\ref{cor:DnetofBSnetandQuadric}, we see that the conic $\mathcal K(i,j)$ through the six points of \eqref{eq:gfqibhzeg} is precisely the intersection of $\mathcal Q$ with the plane of the quad joined by $D(i,j)$, $D(i+1,j)$, $D(i,j+1)$. 
\end{remark}

\section{Kœnigs grids} \label{sec:constrained}

\subsection{Construction}\label{subsection:construction}

In this section, we investigate Kœnigs $d$-grids, that is Kœnigs nets with $d$-dimensional parameter spaces. Therefore, let us give a brief explanation of how to construct such grids.
We begin with an extensive Kœnigs net $P$ that is defined on a $\Sigma_{d,d}$-patch. The parameter spaces $\pav P(i)$, $\pah P(j)$ of this net are necessarily $d$-dimensional, so it is also a Kœnigs $d$-grid. We also fix touching conics $\mathcal C$, which are contained in a unique inscribed quadric $\mathcal Q$ due to Theorem~\ref{thm:BSquadriclocalpatch}. Also recall that the inscribed quadric $\mathcal Q$ contains all the parameter spaces (of the touching points) $\pav T(i)$ and $\pah S(j)$, which have dimension $d-1$. 

Next, we want to extend our net $P$ to a Kœnigs $d$-grid $\bar P$ defined on $\Sigma_{d,d+1}$, which requires that $\bar P$ still has $d$-dimensional parameter spaces. In particular, this means that $\pav P(i)$ and $\pav{\bar P}(i)$ coincide for all $i \in \Sigma_d$. We choose $\bar P(0,d+1)$ in $\pav P(0)$. The line 
\begin{align}
	L = P(0,d) \vee P(0,d+1),
\end{align}
is contained in $\pav P(0)$ and therefore intersects $\pav T(0)$ in a point $\bar T(0,d)$. Moreover, the plane
\begin{align}
	E = P(0,d) \vee P(0,d+1) \vee P(1,d)
\end{align}
intersects $\mathcal Q$ in a conic $\mathcal C(0,d)$, which touches $L$ in $\bar T(0,d)$. Additionally, a dimension counting argument shows that $\pav T(1)$ intersects $E$ in a point $\bar T(1,d)$, which is contained in $\mathcal C(0,d)$ since $\pav T(1)$ is contained in $\mathcal Q$. This allows us to define $\bar P(1,d+1)$ uniquely as the point on the line $\bar T(1,d) \vee P(1,d)$ intersected with the other tangent to $\mathcal C(0,d)$ from $\bar P(0,d+1)$. Moreover, since both $\bar T(1,d)$ and $P(1,d)$ are contained in $\pav P(1)$, so is the new point $\bar P(1,d+1)$. Finally, we may iterate the procedure to obtain the points $\bar P(2,d+1)$ to $\bar P(d,d+1)$. The net $\bar P$ is a $d$-grid and -- as it has touching conics -- it is also a Kœnigs net (due to Definition~\ref{thm:BSKoenigstouchingconics}).

The same steps may be taken to construct Kœnigs $d$-grids on larger and larger patches. Consequently, Kœnigs $d$-grids are defined by initial data on the set
\begin{align}
	\Sigma_{d,d} \cup (\Sigma_{0} \times \Z) \cup (\Z \times \Sigma_{0}), \label{eq:initialkoenigsgriddata}
\end{align}
plus the 1-parameter choice of $\mathcal C$ on $\Sigma_{d,d}$ (with the assumption that the data along the coordinate axes is contained in $d$-dimensional spaces). Let us add that the extensivity requirement on the initial data can be circumvented via Lemma~\ref{lem:extensive}, by considering a lift first and then projecting back in the end.

An interesting property of the construction procedure for a Kœnigs $d$-grid is that the spaces $\pav T(i)$ and $\pah S(j)$ are only $(d-1)$-dimensional. One might suspect that not all Kœnigs $d$-grids are obtained in this manner. However, it turns out that this is not a special property of the Kœnigs $d$-grid -- instead it is a special property of the instance $\mathcal C$ of touching conics that is used for the construction.

\subsection{Special touching conics and special grids}

For a Kœnigs $d$-grid, the spaces $\pav T(i)$ and $\pah S(j)$ are at most $d$-dimensional, and in general they are exactly $d$-dimensional. In this section, we deal with special cases where the dimension is $d-1$ instead.

\begin{definition}\label{def:special}
	Let $P\colon \Sigma \to \RP^{2d}$ be a  $\Sigma_{d,d}$-extensive Kœnigs $d$-grid. An instance of touching conics $\mathcal C_0$ is called \emph{special} if each $\pav T(i)$ and each $\pah S(j)$ is $(d-1)$-dimensional.
\end{definition}

Of course, if $\Sigma = \Sigma_{d,d}$, then every instance of touching conics $\mathcal C_0$ is special. Therefore, Definition~\ref{def:special} is only interesting for larger patches.

The significance of the following Lemma~\ref{lem:uniquespecialtouchingconics} is that it implies that special touching conics are uniquely determined for K{\oe}nigs $d$-grids, provided we assume some mild genericity conditions. Later, in Section~\ref{subsec:existencespecialtouchinconics}, we prove the existence of special touching conics for generic K{\oe}nigs $d$-grids.

Note, for Q-nets defined on patches of the form $\Sigma_{a,1}$ with $a \geq 1$, there is always a $1$-parameter family of touching conics. In the following Lemma~\ref{lem:uniquespecialtouchingconics}, one particular instance of touching conics is distinguished.

Throughout the remainder of Section~\ref{sec:constrained}, $D$ always denotes the diagonal intersection net of a Q-net $P$, as in Definition~\ref{defn:diagintersectionnet}.

\begin{lemma}\label{lem:uniquespecialtouchingconics}
    Let $P \colon \Sigma_{d+1,1} \to \RP^{n}$ be a non-degenerate Q-net that is $\Sigma_{d,1}$-extensive and such that $P_1$ is nowhere Laplace degenerate. Suppose that $\pah{P}(0), \pah{P}(1)$ are $d$-dimensional.  If $\pah{D}(0)$ is $d$-dimensional, then there is a unique instance of touching conics $\mathcal{C}$ such that $\pah{S}(0)$ is $(d-1)$-dimensional. 
    For this instance, $\pah{S}(1)$ is also $(d-1)$-dimensional and
    \begin{align}\label{eq:wehrg}
        \pah{D}(0) = \pah{S}(0) \vee \pah{S}(1).
    \end{align}
\end{lemma}

\begin{proof}
    Suppose that $\mathcal{C}$ is an instance of touching conics such that $\pah{S}(0)$ and $\pah{S}(1)$ are $(d-1)$-dimensional. Suppose that $\pah{S}(0)$ equals $\pah{S}(1)$. Then, $S(i,0)$ equals $S(i,1)$ for all $i$. Consequently, 
    \begin{align}
        \mathcal{C}(i,0) = P_{1}(i,0) \vee P_{-1}(i,0),
    \end{align}
    that is, all the inscribed conics $\mathcal{C}(i,0)$ are degenerate and
    \begin{align}
        T(i,0)=T(i+1,0) = P_{1}(i,0).
    \end{align} 
    Because the conics  $\mathcal{C}(i,0)$ are touching, it follows that $P_1(i,0)$ is independent of $i$. This contradicts the assumption that $P_1$ is nowhere Laplace degenerate. Therefore $\pah{S}(0) \neq \pah{S}(1)$. Since $S$ is a Q-net and $\pah{S}(0), \pah{S}(1)$ are distinct $(d-1)$-spaces, it follows that $\pah{S}(0) \vee \pah{S}(1)$ is $d$-dimensional.

    Next, we show that $\pah{S}(0) \vee \pah{S}(1)$ equals the $d$-space $\pah{D}(0)$. Since, $\pah{S}(0)$ and $\pah{S}(1)$ are distinct, there is an $i_0$ such that $S(i_0,0)$ and $S(i_0,1)$ are distinct. By Lemma~\ref{lem:harmonicdiagonals}, $D(i_0,0)$ is on the line $S(i_0,0) \vee S(i_0,1)$. Therefore,
    \begin{align}
        \pah{S}(0) \vee D(i_0,0) = \pah{S}(0) \vee S(i_0,1).
    \end{align}
    Moreover,
    \begin{align}
        \pah{S}(0) \vee S(i_0,1) = \pah{S}(0) \vee \pah{S}(1),
    \end{align}
    since $S$ is a Q-net. By Lemma~\ref{lem:harmonicdiagonals},  the line $S(i_0,0)\vee D(i_0,0)$ (which is contained in $\pah{S}(0) \vee \pah{S}(1)$) intersects the line $S(i_0+1,0)\vee D(i_0+1,0)$. This intersection point is contained in $\pah{S}(0) \vee \pah{S}(1)$. Its join with $S(i_0+1,0)$ is a line that is also contained in $\pah{S}(0) \vee \pah{S}(1)$. This line contains $D(i_0+1,0)$. Therefore, $D(i_0+1,0)$ is contained in $\pah{S}(0) \vee \pah{S}(1)$. By iterating this argument, we obtain that all the points $D(i,0)$ are contained in $\pah{S}(0) \vee \pah{S}(1)$. Thus, $\pah{D}(0)$ is contained in $\pah{S}(0) \vee \pah{S}(1)$. Therefore, $\pah{S}(0) \vee \pah{S}(1)$ equals $\pah{D}(0)$ because both spaces are $d$-dimensional. This shows Equation~\eqref{eq:wehrg} in the claim. Then, we must have that 
    \begin{align}\label{eq:formulaforSpointsofspecialtouchingconics}
        S(i,0)= \pah{D}(0) \cap \big( P(i,0)\vee P(i+1,0)\big).
    \end{align}
    This shows that the instance of special touching conics is uniquely determined. 
\end{proof}

We now show that there are certain non-generic $d$-grids for which every instance of touching conics is special.

\begin{definition}\label{def:specialgrid}
	A Kœnigs $d$-grid $P$ is called \emph{special} if $P$ is $\Sigma_{d,d}$-extensive and all $\pah D(j)$, $\pav D(i)$ spaces are $(d-1)$-dimensional.
\end{definition}

\begin{lemma}\label{lemma:specialgridpspecialconics}
    Let $P \colon \Sigma_{d+1,1} \to \RP^n$ be a non-degenerate Q-net that is $\Sigma_{d,1}$-extensive. Suppose that $\pah{P}(0)$, $\pah{P}(1)$ are $d$-dimensional. %Let $D$ be the diagonal intersection net. 
    If $\pah{D}(0)$ is $(d-1)$-dimensional, then for any instance of touching conics $\mathcal{C}$, we have that $\pah{S}(0)$ and $\pah{S}(1)$ are $(d-1)$-dimensional.
\end{lemma}

\begin{proof}
    Let $\mathcal{C}$ be an instance of touching conics. We know that the lines $S(i,0) \vee D(i,0)$ and $S(i+1,0) \vee D(i+1,0)$ intersect at a point in the line $P(i+1,0)\vee P(i+1,1)$, as shown in Figure~\ref{figure:conelemma}. Therefore, $S(i,0)\vee \pah{D}(0)$ contains the point $S(i+1,0)$. Therefore, the space $S(0,0)\vee \pah{D}(0)$ contains all the points $S(i,0)$. Thus, 
    \begin{align}
        S(0,0)\vee \pah{D}(0) = \pah{S}(0) \vee \pah{D}(0).
    \end{align}
    Since $\pah{D}(0)$ is $(d-1)$-dimensional and $P$ is $\Sigma_{d,1}$-extensive, we have that $S(0,0)\vee \pah{D}(0)$ is $d$-dimensional. Equivalently, 
    \begin{align}
        E:= \pah{D}(0)\vee \pah{S}(0)
    \end{align}
    is $d$-dimensional. Since $P$ is $\Sigma_{d,1}$-extensive, the join of the $d$-spaces $\pah{P}(0)$ and $\pah{P}(1)$ is $(d+1)$-dimensional. Then, $\pah{P}(0) \cap E$ is a $(d-1)$-space which contains $\pah{S}(0)$. Consequently, $\pah{S}(0)$ is at most $(d-1)$-dimensional. By Lemma~\ref{lem:dimensiontspaces}, $\pah{S}(0)$ is at least $(d-1)$-dimensional. Therefore, $\pah{S}(0)$ is exactly $(d-1)$-dimensional. By Lemma~\ref{lem:harmonicdiagonals}, $S(i,0)$, $D(i,0)$ and $S(i, 1)$ are collinear. Thus, the $d$-space $E$ also contains all the points $S(i,1)$. Therefore, $\pah{S}(1)$ is contained in $E$. We deduce that $\pah{S}(1)$ is $(d-1)$-dimensional because it is the intersection $E \cap \pah{P}(1)$.
\end{proof}

As a consequence of Definition~\ref{def:specialgrid} and Lemma~\ref{lemma:specialgridpspecialconics} we obtain the following corollary.

\begin{corollary}\label{cor:specialgridshavespecialconics}
    For special K{\oe}nigs $d$-grids, any instance of touching conics is special.
\end{corollary}

In order to prove our second main result, Theorem~\ref{thm:BSKoenigsconstrainedtangentquadric}, we need to show that special $d$-grids are not generic $d$-grids. To this end, we need to recall \cite[Lemma~4.3]{AffolterFairleyKoenigsLaplace}, which states that for an extensive Kœnigs net $P$ there are two different hyperplanes $U_1, U_2$ such that
\begin{align}
	P(i,j) \in \begin{cases}
		U_1 & \mbox{if } i+j\in 2\Z, \\
		U_2 & \mbox{if } i+j \in 2\Z + 1.
	\end{cases}\label{eq:koenigshyperplanes}
\end{align}
We may view $\mathcal{U}:= U_1 \cup U_2$ as a degenerate quadric with singular locus  $U_1 \cap U_2$. Moreover, all points of the diagonal intersection net $D$ of $P$ are contained in $U_1 \cap U_2$. We use the quadric $\mathcal{U}$ in the proof of Lemma~\ref{lem:specialdegenerate}.

Recall that for generic Kœnigs $d$-grids we require that the two Laplace transforms $P_{\pm d}$ exist and are nowhere Laplace degenerate (see Definition~\ref{def:genericgrid}).
Next, we show why special Kœnigs grids are not generic Kœnigs grids.

\begin{lemma}\label{lem:specialdegenerate}
	Let $P\colon \Sigma \rightarrow \RP^{2d}$ be a special Kœnigs $d$-grid. Suppose that $P_d$ and $P_{-d}$ exist. Then, $P_d$ and $P_{-d}$ are Laplace degenerate.
\end{lemma}

\proof{
	It suffices to consider a $\Sigma_{d,d+1}$-patch and to show that $P_{-d}$ is Laplace degenerate. Due to Lemma~\ref{lem:extensive}, we may consider a lift $\hat P$ of $P$ that is extensive. Let $C$ be the center of the corresponding central projection $\pi$ with $\pi(\hat P) = P$. Moreover, since $P$ is special, $P$ is $\Sigma_{d,d}$-extensive. Therefore, the dimension of $C$ is 0, that is $C$ is a point. Because $P$ is a $d$-grid, the dimension of all the $\pav P(i)$ spaces is $d$. On the other hand, because $\hat P$ is extensive, the dimension of all the $\pav{\hat P}(i)$ spaces is $d+1$. Thus, $C$ is contained in every $\pav{\hat P}(i)$ space. Equivalently,
	\begin{align}
		C \in \bigcap_{i \in \Sigma_d} \pav{\hat P}(i) =: L.
	\end{align}   
    Since $P_d$ exists, Lemma~\ref{lem:expllaplacetransform} shows that $\bigcap_{i \in \Sigma_d} \pav{P}(i)$ is a point, namely $P_{d}(0,0)$ (which coincides with $P_d(0,1)$). Then, the join of this point with $C$ equals $L$. Thus, $L$ is a line. On the other hand, since we are assuming that $P_{d}$ exists and hence that $\hat P_{d}$ also exists, we may use \cite[Lemma 5.10]{AffolterFairleyKoenigsLaplace} applied to $P_{d-1}$, which shows that
	\begin{align}
		\hat P_{d}(0,0) \neq \hat P_{d}(0,1).
	\end{align}
	Moreover, both points $\hat P_{d}(0,0)$ and $\hat P_{d}(0,1)$ are contained in $L$. Therefore,
    \begin{align}
		L = \hat P_{d}(0,0) \vee \hat P_{d}(0,1).
	\end{align}    
	Furthermore, let $\hat D$ denote the diagonal intersection net of $\hat P$ (which is also a lift of $D$). It is not difficult to see that if every space $\pav{\hat D}(i)$ has dimension $d-1$, then $\hat P$ is not extensive. Therefore, we may assume that there is $k \in \Sigma_{d-1}$ such that $\pav{\hat D}(k)$ has dimension $d$. On the other hand, since $P$ is special $\pav D(k)$ has dimension $d-1$. Consequently, $C$ is contained in $\pav{\hat D}(k)$. Using \cite[Lemma~4.3]{AffolterFairleyKoenigsLaplace} (which we recalled before Lemma~\ref{lem:specialdegenerate}) we see that $C$ is contained in the singular locus $U_1 \cap U_2$ of the quadric $U_1 \cup U_2$. As we have shown that $C$ is in $L$, the assumptions of \cite[Lemma~6.4]{AffolterFairleyKoenigsLaplace} are satisfied. As a consequence we obtain that
   	\begin{align}
		\hat P_{-d}(0,0) = \hat P_{-d}(0,1).
	\end{align} 
	Since $P$ is the central projection of $\hat P$, we also get that $P_{-}(0,0)$ and $P_{-d}(0,1)$ coincide, which shows that $P_{-d}$ is Laplace degenerate.\qed
}

\subsection{Existence of special touching conics}\label{subsec:existencespecialtouchinconics}

Throughout this subsection, we only consider generic K{\oe}nigs grids, see Definition~\ref{def:genericgrid}.

The construction that we explained in Section~\ref{subsection:construction} gives Kœnigs grids together with special touching conics $\mathcal C_0$. In the following we will see that every generic Kœnigs grid comes with a unique 
instance of special touching conics. In fact, in Lemma~\ref{lem:PdgridwithDdgrid} we are able to give an explicit formula for the touching points of the special touching conics of any generic Kœnigs grid. A direct proof of the existence and uniqueness of special touching conics using the explicit formulas appears to be difficult. Therefore, we take another route, which considers weaker criteria that characterize special touching conics. First, we need the following lemma.

\begin{lemma}\label{lem:Disadgrid}
    Let $P\colon \Sigma_{d+1,d+1} \to \RP^{2d}$ be a generic Kœnigs $d$-grid. Then, all the parameter spaces $\pav{D}(i)$ and $\pah{D}(j)$ are $d$-dimensional. %\sout{$D$ is also a $d$-grid}.
\end{lemma}
\begin{proof}

Let $\South P$ be the restriction of $P$ to $\Sigma_{d+1,d}$ and let $\South{D}$ be the diagonal intersection net of $\South{P}$. By the $\Sigma_{d,d}$-extensivity of $P$ (and thus also of $\South{P}$), each $\pav{\South{D}}(i)$ is $(d-1)$-dimensional. Suppose that all $\pah{\South{D}}(j)$ spaces are $(d-1)$-dimensional, which would imply that $\South{P}$ is special. We have that $P_d$ and $P_{-d}$ (and thus also $\South{P}_d$ and $\South{P}_{-d}$) exist because $P$ is generic. Then, Lemma~\ref{lem:specialdegenerate} shows that $\South{P}_d$ is Laplace degenerate. This contradicts our assumption that $P$ is a generic $d$-grid. Therefore, there is some $j_0$ such that $\pah{\South D}(j_0)$, which coincides with $\pah{D}(j_0)$, is not $(d-1)$-dimensional. Furthermore, $\pah{D}(j_0)$ cannot be less than $(d-1)$-dimensional due to the $\Sigma_{d,d}$-extensivity of $P$. Moreover, $\pah{D}(j_0)$ is at most $d$-dimensional since $\pah{D}(j_0)$ is defined as the join of $d+1$ points. Therefore, $\pah{D}(j_0)$ is $d$-dimensional. By Lemma~\ref{lem:uniquespecialtouchingconics}, it follows that there is a unique instance of touching conics such that $\pah{S}(j_0)$ and $\pah{S}(j_0+1)$ are both $(d-1)$-dimensional. Suppose that $\pah{D}(j_0+1)$ is not $d$-dimensional. Then, it must be $(d-1)$-dimensional due to the $\Sigma_{d,d}$-extensivity of $P$. By Lemma~\ref{lem:uniquespecialtouchingconics}, we deduce that $\pah{S}(j_0+1)$ is $(d-1)$-dimensional for any instance of touching conics. This contradicts that $\pah{S}(j_0+1)$ is $(d-1)$-dimensional for only one instance of touching conics. Therefore, $\pah{D}(j_0+1)$ must be $d$-dimensional. Iterating further, we have that $\pah{D}(j)$ is $d$-dimensional for all $j$. By analogous arguments, $\pav{D}(i)$ is $d$-dimensional for all $i$.
\end{proof}

In the next lemma it is not strictly necessary to assume that the $d$-grid is generic. However, we only need the generic case, which also shortens the proof.

\begin{lemma}\label{lem:specialsolo}
	Let $P\colon \Sigma_{d+1,d+1} \to \RP^{2d}$ be a generic Kœnigs $d$-grid, and let $\mathcal C$ be touching conics. 
	\begin{enumerate}
		\item If $\pah{S}(j)$ is $(d-1)$-dimensional for some $j$, then $\pah{S}(j)$ is $(d-1)$-dimensional for all $j$.
		\item If $\pav{T}(i)$ is $(d-1)$-dimensional for some $i$, then $\pav{T}(i)$ is $(d-1)$-dimensional for all $i$.
	\end{enumerate}	
\end{lemma}

\proof{
    By Lemma~\ref{lem:Disadgrid}, we have that all $\pah{D}(j)$ and $\pav{D}(i)$ spaces are $d$-dimensional. Suppose that $\pah{S}(j_0)$ is $(d-1)$-dimensional.  Then, by Lemma~\ref{lem:uniquespecialtouchingconics}, we have that $\pah{S}(j_0\pm 1)$ is also $(d-1)$-dimensional. Iterating further, we have that all $\pah{S}(j)$ spaces are $(d-1)$-dimensional. By symmetric arguments we also obtain the second statement in the claim of the lemma. \qed
}

As a consequence of Lemma~\ref{lem:specialsolo}, it suffices to check the dimension condition for only one $\pah S(j)$ and only one $\pav T(i)$.

\begin{lemma} \label{lem:specialsym}
	Let $P\colon \Sigma_{d+1,d+1} \to \RP^{2d}$ be a generic Kœnigs $d$-grid. Let $\mathcal{C}$ be touching conics. The following two conditions are equivalent:
	\begin{enumerate}
		\item $\pah{S}(j)$ is $(d-1)$-dimensional for all $j$,
		\item $\pav{T}(i)$ is $(d-1)$-dimensional for all $i$.
	\end{enumerate}
\end{lemma}

\begin{proof}
	Let $\mathcal{C}$ be an instance of touching conics such that statement $2$ is satisfied. We will show that each $\pah{S}(j)$ is $(d-1)$-dimensional.
	It is not difficult to see that there is a $\Sigma_{d+1,d}$-extensive Kœnigs net $\tilde{P}: \Sigma_{d+1,d+1} \to \RP^{2d+1}$ and a central projection $\pi$ with a point $Z$ as center, such that
	\begin{enumerate}
		\item $\pi(\tilde P) = P$,
		\item each $\pah{\tilde{P}}(j)$ is $(d+1)$-dimensional, and 
		\item each $\pav{\tilde{P}}(i)$ is $d$-dimensional.
	\end{enumerate}

     Let $\tilde{\mathcal{C}}(i,j)$ be the touching conics of $\tilde{P}$ that project via $\pi$ to $\mathcal{C}(i,j)$. Let $\tilde{S}(i,j)$ and $\tilde{T}(i,j)$ be the tangency points of the touching conics $\tilde{\mathcal{C}}(i,j)$.
	We consider the two restrictions $\South \tilde P$, $\North \tilde P$ of $\tilde{P}$ defined respectively on the $\Sigma_{d+1,d}$ patches 
    \begin{align}
		\Sigma_{d+1} \times \{0,1, \ldots, d\} , & & \Sigma_{d+1} \times \{1,2, \ldots, d+1\}.
	\end{align}
    
    Theorem~\ref{thm:BSquadriclocalpatch} determines unique inscribed quadrics $\North \mathcal{\tilde{Q}}$ and $\South \mathcal{\tilde{Q}}$ for $\North \tilde P$ and $\South \tilde P$, respectively. 
    Next, we will show that $\North \mathcal{\tilde{Q}} = \South \mathcal{\tilde{Q}}$. Because of the uniqueness of inscribed quadrics, it suffices to show that $\South \mathcal{\tilde{Q}}$ is also an inscribed quadric for $\North \tilde P$. Consider the plane
	    \begin{align}
		\tilde \Pi(0,d) = \tilde P(0,d) \vee \tilde P(1,d) \vee \tilde P(0,d+1),
	\end{align}
	which belongs to a quad of $\North \tilde P$ but not $\South \tilde P$. This quad comes with the conic ${\mathcal{\tilde C}}(0,d)$ which lies in $\North \mathcal {\tilde Q}$ by definition of $\North \mathcal{\tilde Q}$, but a priori not in $\South \mathcal{\tilde Q}$. So let us also consider the conic 
	\begin{align}
		\mathcal{\tilde K}(0,d) := \South \mathcal{\tilde Q} \cap \tilde{\Pi}(0,d).
	\end{align}
	Clearly, the point $\tilde S(0,d)$ is both in $\mathcal{\tilde C}(0,d)$ and $\mathcal{\tilde K}(0,d)$ since this point is contained in both $\South \tilde P$ and $\North \tilde P$. Moreover, the tangent in $\tilde S(0,d)$ is the same for $\mathcal{\tilde C}(0,d)$ and $\mathcal{\tilde K}(0,d)$ because it is the line $\tilde{P}(0,d) \vee \tilde{P}(1,d)$ which is contained in both $\South \tilde P$ and $\North \tilde P$.
	
	Because the lift has the property that each $\pav {\tilde P}(i)$ is $d$-dimensional, each $\pav {\tilde T}(i)$ has dimension $d-1$ (the same dimension as each $\pav{T}(i)$ by assumption). Then, $\North \pav {\tilde T}(i)$ equals $\South \pav {\tilde T}(i)$ for all $i$ because they are both equal to $\pav {\tilde T}(i)$. This implies that the point ${\tilde T}(0,d)$ is contained in $\South \mathcal{\tilde Q}$ and also $\North \mathcal{\tilde Q}$. It also implies that the line $\tilde P(0,d) \vee  \tilde T(0,d)$ is tangent to $\South \mathcal{\tilde Q}$ and also $\North \mathcal{\tilde Q}$. Therefore, $ \tilde T(0,d)$ is also contained in $\mathcal{\tilde K}(0,d)$, and the line $\tilde P(0,d) \vee \tilde T(0,d)$ is tangent to $\mathcal{\tilde K}(0,d)$. Furthermore, $\mathcal{\tilde C}(0,d)$ also contains $\tilde T(0,d)$ with the same tangent line. Finally, the same arguments apply to $\tilde T(1,d)$ and the tangent line $\tilde P(1,d) \vee \tilde T(1,d)$. Consequently, the two conics $\mathcal{\tilde C}(0,d)$ and $\mathcal{\tilde K}(0,d)$ share three points and the tangents in these three points. Therefore, the two conics coincide. Analogous arguments show that
	\begin{align}
		 \mathcal{\tilde C}(i,d) = \mathcal{\tilde K}(i,d)
	\end{align} 
	for all $i$. Hence, $\South \mathcal{\tilde Q}$ contains all touching conics of $\North \tilde P$ as required for an inscribed quadric of $\North \tilde{P}$. 
	
	Next, let us show that $\South \mathcal{\tilde Q}$ contains $\pah {\tilde S}(d+1)$.  Without loss of generality, we assume that $\pah {\tilde S}(d) \neq \pah {\tilde S}(d+1)$. Then, the space
	\begin{align}
		\pah{\tilde{S}}(d) \vee \pah{\tilde{S}}(d+1)
	\end{align}
	is $(d+1)$-dimensional. In this space, $\South \mathcal{\tilde Q}$ contains the $d$-space $\pah {\tilde S}(d)$ which has codimension 1. Therefore, the restriction of $\South \mathcal{\tilde Q}$ to this space must be the union of two $d$-spaces. As $\South \mathcal{\tilde Q}$ also contains $d+1$ points in  $\pah {\tilde S}(d+1)$, the second $d$-space must be $\pah {\tilde S}(d+1)$. Therefore, $\South \mathcal{\tilde Q}$ contains $\pah {\tilde S}(d+1)$.

	By definition, $\South \mathcal{\tilde Q}$ also contains $\pah {\tilde S}(j)$ for $j \in \Sigma_d$. Moreover, since $\North \pav {\tilde T}(i)$ equals $\South \pav {\tilde T}(i)$ for all $i$, the quadric $\South \mathcal{\tilde Q}$ contains all $\North \pav {\tilde T}(i)$ spaces. So far, we have shown all the conditions of Definition~\ref{def:inscribedquadric} saying that $\South \mathcal{\tilde Q}$ is an inscribed quadric for $\North \tilde P$. Recall, Theorem~\ref{thm:BSquadriclocalpatch} states that inscribed quadrics are unique. Thus, we conclude that $\North \mathcal{\tilde Q} = \South \mathcal{\tilde Q}$, and we write
	\begin{align}
		\mathcal{\tilde Q} := \North \mathcal{\tilde Q} = \South \mathcal{\tilde Q}.
	\end{align}
    By applying Lemma~\ref{lem:singularpointsinscribedquadric} to $\South \tilde{P}$ and $\North \tilde{P}$, it follows that any point in
	\begin{align}\label{eq:space1}
		\South\tilde{X} := \bigcap_{j=0}^{d} \pah{\tilde{S}}(j) \quad \mbox{and} \quad \North\tilde{X} = \bigcap_{j =1}^{d+1} \pah{\tilde{S}}(j)
	\end{align} 
	is a singular point of $\tilde{\mathcal{Q}}$. Also using Lemma~\ref{lem:singularpointsinscribedquadric} and that we are working with $\Sigma_{d+1,d}$ restrictions, we know that $\South\tilde{X}$ and $\North\tilde{X}$ are both at least $0$-dimensional and thus non-empty.

	Because we are assuming that $P_{-d}$ is nowhere Laplace degenerate and using Lemma~\ref{lem:expllaplacetransform},
we have that $\cap_{j \in \Sigma_{d+1}} \pah{P}(j)$ is empty. 	
	Therefore, as $\pah{\tilde{P}}(j)$ equals $\pah P(j) \vee Z$ for all $j$, it follows that
	\begin{align}\label{eq:rehkig}
	    Z= \bigcap_{j\in \Sigma_{d+1}} \pah{\tilde{P}}(j).
	\end{align}	
	Note that
	\begin{align}\label{eq:rshgreah}
		\South\tilde{X}  \subset \bigcap_{j\in \Sigma_{d}} \pah{\tilde{P}}(j).
	\end{align}
	Moreover, as we are assuming that $P$ is a generic $d$-grid, $\cap_{j \in \Sigma_{d}} \pah{P}(j)$ is a point. Hence, the right-hand side of \eqref{eq:rshgreah} is a $1$-dimensional space, which contains $Z$ by definition of the lift $\hat P$. Consequently, $\South\tilde{X}$ is at most $1$-dimensional.
	
	Let us distinguish two cases. First, let us assume $\South\tilde{X}$ is $1$-dimensional. Then, the $d$-space $\pah{\tilde{S}}(0)$ contains $Z$ and thus the projection $\pah{S}(0)$ is $(d-1)$-dimensional. Furthermore, by Lemma~\ref{lem:specialsolo} we have that each $\pah{S}(j)$ is $(d-1)$-dimensional. Hence, this case is fine. Analogously, if $\North\tilde{X}$ is $1$-dimensional, it also follows that $\pah{S}(j)$ is $(d-1)$-dimensional.
	
	For the second case, we assume that $\South\tilde{X}$ and  $\North\tilde{X}$ are both $0$-dimensional, which means that both spaces have only one point. Then, by applying Lemma~\ref{lem:singularpointsinscribedquadricnew} to $\South \tilde{P}$ and $\North \tilde{P}$, we have that any singular point of $\tilde{\mathcal{Q}}$ must be contained in $\South\tilde{X}$ and also $\North\tilde{X}$ -- as otherwise Lemma~\ref{lem:singularpointsinscribedquadricnew} would imply that both $\South\tilde{X}$ and also $\North\tilde{X}$ would be 1-dimensional. Since $\South\tilde{X}$ and $\North\tilde{X}$ are singular points, and every singular point has to be in both $\South\tilde{X}$ and $\North\tilde{X}$, we see that $\South\tilde{X} = \North\tilde{X}$. Moreover, both points $\South\tilde{X}$ and $\North\tilde{X}$ are contained in \eqref{eq:rehkig} and thus equal $Z$. Then, each $d$-space $\pah{\tilde{S}}(j)$ contains $Z$ and thus each projection $\pah{S}(j)$ is $(d-1)$-dimensional. This means that statement 1 is satisfied, as claimed. Symmetrically, statement 1 implies statement 2.
\end{proof}

With Lemma~\ref{lem:specialsym} we have further weakened the assumptions that we need for touching conics $\mathcal C$ to be special. Specifically -- combined with Lemma~\ref{lem:specialsolo} -- it is sufficient that one $\pah S(j)$ is $(d-1)$-dimensional for $\mathcal C$ to be special.

Recall that Lemma~\ref{lem:BSquadriclocalpatch} states that an extensive Kœnigs net $P$ on $\Sigma_{d,d}$ with an instance of touching conics $\mathcal C$ determines a unique inscribed quadric $\mathcal Q$. However, Kœnigs $d$-grids can be at most $\Sigma_{d,d}$-extensive. Thus, Lemma~\ref{lem:BSquadriclocalpatch} does not apply to $\Sigma_{d+1,d+1}$ patches of Kœnigs $d$-grids. However, the following corollary says that for $\Sigma_{d+1,d+1}$ patches of generic Kœnigs $d$-grids, the special instance of touching conics has the property that the inscribed quadrics on the four $\Sigma_{d,d}$ subpatches (which are determined by Lemma~\ref{lem:BSquadriclocalpatch}) coincide and thus determine an inscribed quadric on $\Sigma_{d+1,d+1}$.

\begin{corollary}\label{cor:gridlargerspecial}
	Let $P\colon \Sigma_{d+1,d+1} \to \RP^{2d}$ be a generic Kœnigs $d$-grid with special touching conics $\mathcal C$. Then, there is a unique inscribed quadric $\mathcal Q$, and $\mathcal Q$ is non-degenerate.
\end{corollary}
\proof{
	The quadric $\mathcal Q$ is the quadric $\tilde {\mathcal Q}$ (defined in $\RP^{2d +1}$ in the proof of Lemma~\ref{lem:specialsym}) restricted to $\RP^{2d}$. Note that  $\tilde {\mathcal Q}$ has only one singular point, namely the center of projection $Z$. Since $Z$ is not in $\RP^{2d}$, the singular point disappears when we intersect $\tilde{\mathcal{Q}}$ with $\RP^{2d}$.\qed
}

\begin{remark}
For non-generic Kœnigs grids, degenerations occur. For example, let $P \colon \Z^2 \to \RP^2$ be a $1$-grid such that all lines $\pav{P}(i)$ are concurrent at some point, say $Y$, and such that all lines $\pah{P}(j)$ are concurrent at some point, say $X$. Then, $P$ is a Kœnigs $1$-grid but it is not generic. The gridlines are not tangent to a non-degenerate conic. Instead, the gridlines are tangent to the degenerate conic $X \vee Y$.
\end{remark}

Corollary~\ref{cor:gridlargerspecial} gives us access to the special inscribed quadric $\mathcal Q$ on $\Sigma_{d+1,d+1}$ patches of Kœnigs grids (earlier we only had inscribed quadrics on $\Sigma_{d,d}$ patches). However, Corollary~\ref{cor:gridlargerspecial} relies on the existence of special conics on $\Sigma_{d+1,d+1}$ patches, which we establish next.

\begin{lemma}\label{lem:uniqueconics}
	Let $P\colon \Sigma_{d+1,d+1} \to \RP^{2d}$ be a generic Kœnigs $d$-grid. Among, the $1$-parameter family of touching conics, there is exactly one instance of special touching conics.
\end{lemma}
\proof{
    By Lemma~\ref{lem:Disadgrid}, all the parameter spaces of $D$ are $d$-dimensional. By Lemma~\ref{lem:uniquespecialtouchingconics}, there is a unique instance of touching conics such that $\pah{S}(0)$ and $\pah{S}(1)$ are $(d-1)$-dimensional. By Lemma~\ref{lem:specialsolo}, all the spaces $\pah{S}(j)$ are $(d-1)$-dimensional. By Lemma~\ref{lem:specialsym}, all the spaces $\pav{T}(i)$ are also $(d-1)$-dimensional. Thus, we obtain the instance of special touching conics. \qed
}

\begin{theorem}\label{th:uniquespeciality}
	Let $P\colon \Sigma \to \RP^{2d}$ be a generic Kœnigs $d$-grid, where we assume that $\Sigma_{d+1,d+1}$ is contained in  $\Sigma$. 
	\begin{enumerate}
		\item There is a unique instance of special touching conics $\mathcal C$.
		\item There is a unique inscribed quadric $\mathcal Q$ corresponding to the special touching conics $\mathcal C$, and we call $\mathcal Q$ the \emph{special inscribed quadric}.
	\end{enumerate}
\end{theorem}
\proof{
	Let us prove the first claim first. If $\Sigma = \Sigma_{d+1,d+1}$ then the claim follows from Lemma~\ref{lem:uniqueconics}. For larger patches, we use an iterative gluing method. Assume the claim is proven on a $\Sigma_{a,b}$ patch with $a,b \geq d+1$. We want to show that the claim also holds on a $\Sigma_{a+1,b}$ patch. Let us denote by $\West P$, $\WEmid P$ and $\East P$ the three restrictions of $P$ to the subpatches
	\begin{align}
		&\{(i,j) \mid i,j \in \Sigma_{a,b}\},  & & \{(i+1,j) \mid i,j \in \Sigma_{a-1,b}\}, & & \{(i+1,j) \mid i,j \in \Sigma_{a,b}\}.
	\end{align}
	By applying the inductive assumption to $\West P$ and $\East P$ we obtain two sets of special conics $\West \mathcal C$ and $\East \mathcal C$. Moreover, the $(d-1)$-spaces $\West \pah S(j)$ and $\East \pah S(j)$ of touching points are already determined by $\WEmid \pah S(j)$. Thus, we deduce that $\West \pah S(j)= \East \pah S(j)$. Subsequently, since the touching points agree, we also conclude that $\West \mathcal C = \East \mathcal C = \mathcal C$ (wherever they overlap). This concludes the proof of the first claim.

	The arguments for the second claim are similar. If $\Sigma = \Sigma_{d+1,d+1}$ then the claim follows from Corollary~\ref{cor:gridlargerspecial}. Using the same restrictions as before, the inductive claim provides us with three special inscribed quadrics $\West \mathcal Q$, $\WEmid \mathcal Q$ and $\East \mathcal Q$. The three quadrics coincide because the inscribed quadric for the touching conics $\WEmid \mathcal C$ is unique (by Theorem~\ref{thm:BSquadriclocalpatch}), which proves the second claim. 
    \qed
}

Combining Theorem~\ref{th:uniquespeciality} and Lemma~\ref{lem:BSquadriclocalpatch}, this completes the proof of Theorem~\ref{thm:BSKoenigsconstrainedtangentquadric}. Of course, Lemma~\ref{lem:BSquadriclocalpatch} still applies, therefore the $d$-spaces $\pah P(j)$ are tangent to the special inscribed quadric $\mathcal Q$ along the $(d-1)$-spaces $\pah S(j)$. Analogously, the $d$-spaces $\pah P(i)$ are tangent to the special inscribed quadric $\mathcal Q$ along the $(d-1)$-spaces $\pav T(i)$. Moreover, since Corollary~\ref{cor:gridlargerspecial} provides non-degenerate quadrics, this implies that $\mathcal Q$ is also non-degenerate.

The following Lemma~\ref{lem:PdgridwithDdgrid} is a generalization of Lemma~\ref{lem:Disadgrid}, which only considered the case $\Sigma_{d+1,d+1}$.

\begin{lemma} \label{lem:PdgridwithDdgrid}
	Let $P\colon \Sigma \to \RP^{2d}$ be a generic Kœnigs $d$-grid, with $\Sigma_{d+1,d+1} \subset \Sigma$. Then, the parameter spaces $\pah{D}(j)$ and $\pav{D}(i)$ of the diagonal intersection net are $d$-dimensional. Moreover, the touching points of the special touching conics $\mathcal C$ are given by
	\begin{align}
		S(i,j) &= \pah D(j) \cap (P(i,j) \vee P(i+1,j)) \label{eq:igbrgi}, \\
		T(i,j) &= \pav D(i) \cap (P(i,j) \vee P(i,j+1)) \label{eq:rogubrgi}.
	\end{align}

\end{lemma}
\proof{    
    The case $\Sigma_{d+1,d+1}$ follows from Lemma~\ref{lem:Disadgrid} and Lemma~\ref{lem:uniquespecialtouchingconics}. We now prove the case $\Sigma_{d+2,d+1}$. The proof of the general case proceeds analogously.
    
    Let $\West P$ and $\East P$ be the restrictions of $P$ defined respectively on $\Sigma_{d+1,d+1}$ and $\{1, \ldots, d+2\} \times \Sigma_{d+1}$. By Theorem~\ref{th:uniquespeciality}, we know that $P$ has a unique instance of special touching conics. Since $\pah{S}(0) =\pah{\West S}(0)$ and $\pah{P}(0) = \West \pah{P}(0)$, applying Lemma~\ref{lem:uniquespecialtouchingconics} to $\West P$ shows that \begin{align}
        \pah{S}(0) = \West \pah{D}(0) \cap \pah{P}(0).
    \end{align} Analogously, applying Lemma~\ref{lem:uniquespecialtouchingconics} to $\East P$, we have that \begin{align}\pah{S}(0) = \East \pah{D}(0) \cap \pah{P}(0).\end{align} By applying Lemma~\ref{lem:Disadgrid} to $\West P$ and also $\East P$, we know that $\West \pah{D}(0)$ and $\East \pah{D}(0)$ are both $d$-dimensional and thus they both intersect $\pah{P}(0)$ in a $(d-1)$-space since they are both contained in the $(d+1)$-space $\pah{P}(0) \vee \pah{P}(1)$. Lemma~\ref{lem:Disadgrid} also tells us that both $\West \pah{D}(0)$ and $\East \pah{D}(0)$ contain the $(d-1)$-space $\pah{S}(0)$. Therefore, $\West \pah{D}(0)= \East \pah{D}(0)$ which then also equals $\pah{D}(0)$. Therefore, $\pah{D}(0)$ is $d$-dimensional. Then, Equation~\eqref{eq:igbrgi} follows from Equation~\eqref{eq:formulaforSpointsofspecialtouchingconics}.
    By analogous arguments, all the spaces $\pah{D}(j)$ are $d$-dimensional. By symmetry, all the spaces $\pav{D}(i)$ are also $d$-dimensional and we also obtain Equation~\eqref{eq:rogubrgi}. \qed 
}

Note that a symmetric version of the formulas in Lemma~\ref{lem:PdgridwithDdgrid} also yields the equations
\begin{align}
	S(i,j) &= \pah D(j-1) \cap (P(i,j) \vee P(i+1,j)), \label{eq:wogiue}\\
	T(i,j) &= \pav D(i-1) \cap (P(i,j) \vee P(i,j+1)).\label{eq:ewiofbh}
\end{align}

For a K{\oe}nigs $d$-grid defined on a finite patch $\Sigma_{a,b}$, formulas \eqref{eq:wogiue} and \eqref{eq:ewiofbh} are useful to obtain the points $S(i,b)$ and $T(a,j)$, which cannot be obtained from the formulas \eqref{eq:igbrgi} and \eqref{eq:rogubrgi}.

\subsection{An incidence theorem}

Finally, our initial observations about the construction of Kœnigs $d$-grids together with Theorem~\ref{th:uniquespeciality} give rise to the following incidence theorem.

\begin{theorem}\label{th:bsgridincidence}
	Let $P \colon \Sigma_{d+2,d+2} \rightarrow \RP^{2d}$ be a generic $d$-grid such that its restrictions to
	\begin{align}
		\Sigma_{d+2} \times \Sigma_{d+1} \quad \mbox{and} \quad \Sigma_{d+1} \times \Sigma_{d+2}
	\end{align}
	are Kœnigs nets. Then, $P$ is also a K{\oe}nigs net.
\end{theorem}

\begin{figure}[tb] 
	\[\includegraphics[width=0.6\textwidth]{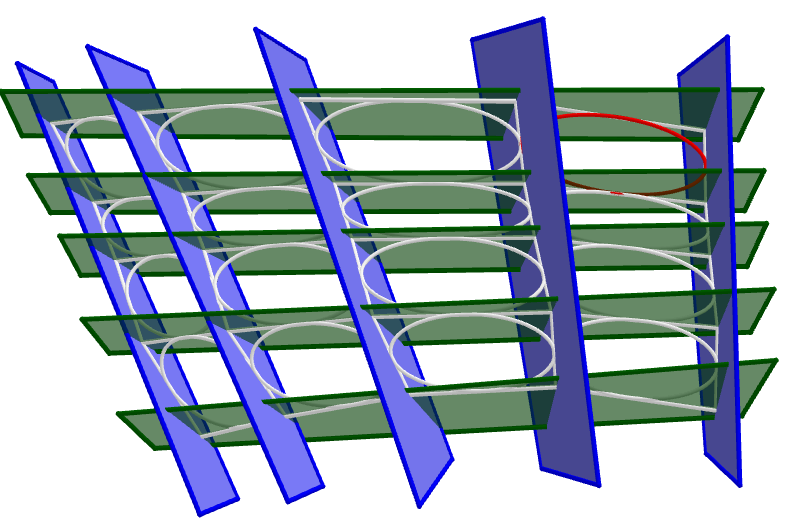}\]
	\caption{A Kœnigs grid $P \colon \Sigma_{4,4} \to \RP^3$ with two families of planar parameter lines. By Theorem~\ref{th:bsgridincidence}, the existence of the white touching conics guarantees the existence of the red touching conic.}
	\label{figure:koenigs5x5_10planes}
\end{figure}

\proof{
	Let us denote the two restrictions of $P$ in the theorem by $\South P$ and $\West P$, respectively. Due to Theorem~\ref{th:uniquespeciality} we obtain the two unique special inscribed quadrics $\South \mathcal Q$ and $\West \mathcal Q$. Moreover, on the restriction $\SW P$ to $\Sigma_{d+1,d+1}$ the two restrictions $\South P$ and $\West P$ agree. But Theorem~\ref{th:uniquespeciality} also gives a unique special quadric $\SW \mathcal Q$. Since all three quadrics  $\South \mathcal Q$, $\West \mathcal Q$ and $\SW \mathcal Q$ have codimension 1 in $\RP^{2d}$, the three quadrics coincide. Therefore, let us write $\mathcal Q := \South \mathcal Q = \West \mathcal Q = \SW \mathcal Q$. Analogously, we obtain a unique instance of special conics $\mathcal C := \South \mathcal C = \West \mathcal C = \SE \mathcal C$ for all quads but the top-right one. We define the final conic $\mathcal C(d+1,d+1)$ as the intersection of $\mathcal Q$ with the plane of the top-right quad
	\begin{align}
		P(d+1,d+1) \vee P(d+2,d+1) \vee P(d+2,d+2) \vee P(d+1,d+2).
	\end{align}
	Because $\mathcal Q$ is tangent to each $\pah{P}(j)$ and $\pav{P}(i)$, the conic $\mathcal C(d+1,d+1)$ is an inscribed conic that touches the conics $\mathcal{C}(d+1,d)$ and $\mathcal{C}(d,d+1)$. Therefore, $\mathcal C$ is an instance of touching conics in $P$. Due to Definition~\ref{thm:BSKoenigstouchingconics} we infer that $P$ is a Kœnigs net.\qed
}

The case $d=1$ of Theorem~\ref{th:bsgridincidence} was discovered in \cite[Cor.~4.3, Fig.~14]{bobenkofairley2021nets}, so Theorem~\ref{th:bsgridincidence} is a higher-dimensional generalization. The case $d=2$ is illustrated in Figure~\ref{figure:koenigs5x5_10planes}.

Note that since the restriction $\SW P$ of $P$ to $\Sigma_{d+1,d+1}$ is a generic Kœnigs $d$-grid, Theorem~\ref{th:uniquespeciality} says that $\SW P$ has an instance of special touching conics $\mathcal{C}$. Then, we can use the construction method that is explained at the beginning of Section~\ref{sec:constrained} to uniquely determine $P$ from the initial data \eqref{eq:initialkoenigsgriddata} and the touching conics $\mathcal{C}$ on the restriction of $P$ to $\Sigma_{d,d}$. More generally, the existence of special touching conics in Theorem~\ref{th:uniquespeciality} shows that any Kœnigs $d$-grid can be constructed using the method that is explained at the beginning of Section~\ref{sec:constrained}.

\subsection{Non-extensive K{\oe}nigs $d$-grids}

So far, we mostly considered generic Kœnigs $d$-grids. These grids live in $\RP^{2d}$. However, our results also have applications in lower dimensions. We provide some remarks about applications in $\RP^3$.

\begin{remark}\label{rem:koenigsgridRP3}
   Consider a (non-extensive) Kœnigs 2-grid $P\colon  \Sigma \to \RP^3$. Suppose that $P_2$ and $P_{-2}$ exist and are nowhere Laplace degenerate. This is a mild genericity condition. Then, we claim there exists a quadric $\mathcal Q$, such that the planes $\pav{P}(i)$ and $\pah{P}(j)$ are tangent to $\mathcal Q$. To see this, consider a $\Sigma_{2,2}$-extensive lift $\hat P: \Sigma \rightarrow \RP^4$ of $P$, which exists due to Lemma~\ref{lem:extensive}, so that there is a central projection $\pi$ such that $P = \pi \circ \hat P$. Because the lift $\hat P$ is a generic Kœnigs $2$-grid, Theorem~\ref{th:uniquespeciality} says that there is a (unique) special inscribed quadric $\mathcal{\hat Q}$ such that each plane $\pah{\hat{P}}(j)$ is tangent along the isotropic line $\pah{\hat{S}}(j)$ and each $\pav{\hat{P}}(i)$ is tangent along the isotropic line $\pav{\hat{T}}(i)$. Then, using the same argument as in Remark~\ref{rem:twicetangetnconics}, we obtain a unique quadric $\mathcal{Q}$ that is tangent to the parameter spaces of $P$. Let $\hat{\mathcal{C}}$ be the special touching conics of $\hat{P}$ and $\mathcal{C} := \pi (\hat{\mathcal{C}})$ be the projected touching conics. Then, Remark~\ref{rem:twicetangetnconics} also tells us that all the conics $\mathcal{C}$ are twice tangent to $\mathcal{Q}$.
\end{remark}

\begin{remark}
    Suppose $P \colon \Sigma \to \RP^3$ is a K{\oe}nigs $2$-grid. Generally, as explained in Remark~\ref{rem:koenigsgridRP3}, there is a non-degenerate quadric $\mathcal{Q}$ that is tangent to the parameter spaces of $D$. If we apply projective duality, the dual $P^\ast$ of $P$ is also a $Q$-net. It has the remarkable property that its Laplace transforms $P^\ast_{\pm 1}$ are \emph{Laplace degenerate} (Definition~\ref{defn:laplace}) and the vertices of $P^\ast_{\pm 1}$ are contained in the dual quadric $\mathcal{Q}^\ast$ of $\mathcal{Q}$. 
\end{remark}

\section{Autoconjugate curves}\label{section:autoconjugate}\label{sec:autoconjugate}

In this section, we show that autoconjugate curves appear naturally in the special inscribed quadric of a generic Kœnigs $d$-grid. Indeed, we show that they occur as the $d$-th forwards and backwards Laplace transforms of the diagonal intersection net. Moreover, we show that there is actually a bijection between generic pairs of autoconjugate curves and generic Kœnigs $d$-grids. For this, it will be practical to describe the autoconjugate curves via the parameter spaces of the nets $S$ and $T$ of the special instance of touching conics.

\subsection{Laplace transforms of K{\oe}nigs $d$-grids}\label{sec:laplacetransformKoenigsgrids}

We now record some relations between the Laplace transforms of a generic Kœnigs $d$-grid and its special inscribed quadric. We think that these are interesting in their own right, but they also support our approach to defining discrete autoconjugate curves.

\begin{theorem}\label{th:diaginquadric}
	Let $P\colon \Sigma \to \RP^{2d}$ be a generic Kœnigs $d$-grid and let $\mathcal Q$ be the unique special inscribed quadric.
    Suppose that the Laplace transforms $D_{\pm d}$ exist. Then:
	\begin{enumerate}
	\item The expressions $P_d(i,j)$ and $D_d(i,j)$ are independent of $j$, whereas $P_{-d}(i,j)$ and $D_{-d}(i,j)$ are independent of $i$. Therefore, $D_{\pm d}$ and $P_{\pm d}$ are discrete curves that are well-defined. They are given by the  formulas
         \begin{align}
              P_{d}(i) &:= P_{d}(i,j),  & P_{-d}(j) &:= P_{-d}(i,j),\\
             D_{d}(i) &:= D_{d}(i,j), & D_{-d}(j) &:= D_{-d}(i,j). \label{eq:wkghibwaggg}
        \end{align}
    \item The images of $D_{\pm 1}, \ldots, D_{\pm d}$ are contained in $\mathcal{Q}$.
    \item The join of any $d$ consecutive points of $D_{\pm d}$ is contained in $\mathcal{Q}$.
    \item The vertices of $D_{\pm d}$ are contained in the edge-lines of $P_{\pm d}$. More precisely, for all $k$,
		\begin{align}
			D_{\pm d}(k) \in P_{\pm d}(k) \vee P_{\pm d}(k+1).
		\end{align}
    \item Each point $D_{\pm }(k)$ is conjugate to the $2d$ points $P_{\pm d}(k+1-d), \ldots, P_{\pm d}(k+d)$.
	\end{enumerate}	
\end{theorem}

Before we proceed to the proof, let us note that statement~3 in Theorem~\ref{th:diaginquadric} implies that  $D_d$ and $D_{-d}$ are autoconjugate curves in the sense of Definition~\ref{defn:autoconjugate}.

\proof[Proof of Theorem~\ref{th:diaginquadric}]{
    In the following, it suffices to prove the statements for Laplace transforms in the positive direction. The proofs in the negative direction are symmetric.

    Since $P$ is a $d$-grid, Lemma~\ref{lem:expllaplacetransform} implies that 
    \begin{align} \label{eq:esggw}
       P_d(i,j):=  \pav{P}(i) \cap \ldots \cap \pav{P}(i+d).
    \end{align} Therefore, $P_d(i,j)$ is independent of $j$. Thus, $P_d$ is a discrete curve and $P_d(i) := P_d(i,j)$ is well-defined. Since Lemma~\ref{lem:Disadgrid} says that $D$ is also a $d$-grid, we analogously obtain that 
    \begin{align} \label{eq:wkghibwg}
       D_d(i,j):=  \pav{D}(i) \cap \ldots \cap \pav{D}(i+d),
    \end{align} 
    and thus $D_d$ is a discrete curve. This proves statement~1.

	Corollary~\ref{cor:DnetofBSnetandQuadric} implies that $\pav{{(D_1)}}(i) \subset \pav T(i+1)$ for all $i$. Actually, we have equality since they are both $(d-1)$-dimensional.  Then, each $\pav{{(D_1)}}(i)$ is contained in $\mathcal{Q}$ because we already know that each $\pav T(j+1)$ is contained in $\mathcal{Q}$. Consequently, all points of all the Laplace transforms $D_{1}, \ldots, D_d$ are contained in $\mathcal Q$, as their vertices are contained in the intersections of some consecutive $\pav{{(D_1)}}(i)$ spaces. This proves statement~2. 
    
	   Since, $\pav{{(D_1)}}(i) = \pav{D}(i) \cap \pav{D}(i+1)$, Equation~\eqref{eq:wkghibwg} is equivalent to \begin{align}\label{eq:rkigwhbyg}
		 D_d(i)= \pav{{(D_1)}}(i) \cap \ldots \cap \pav{{(D_1)}}(i+d-1).
	\end{align}
    
	Equation~\eqref{eq:rkigwhbyg} shows that the $d$ consecutive points $D_d(i), \ldots, D_d(i+d-1)$ are all contained in $\pav{{(D_1)}}(i+d-1)$ and thus are also contained in $\mathcal{Q}$, which proves statement~3.

	Since each $\pav{{D_1}}(i)$ equals $\pav T(i+1)$ and is thus contained in $\pav P(i+1)$, we have that 
	\begin{align}\label{eq:aksrbg}
	   D_d(i) = \pav T(i+1) \cap \ldots \cap \pav T(i+d),
	\end{align}
	is contained in 
	\begin{align}
		\pav P(i+1) \cap \ldots \cap \pav{P}(i+d). \label{eq:rgiour}
	\end{align}
	By applying Lemma~\ref{lem:expllaplacetransform} to $P$, the space \eqref{eq:rgiour} is the line $P_d(i) \vee P_d(i+1)$. This proves statement~4.

    By Lemma~\ref{lem:BSquadriclocalpatch}, each $\pav T(i)$ is a subset of $\pav P(i)$ that is conjugate to $\pav P(i)$. So, Equation~\eqref{eq:aksrbg} shows that $D_d(i)$ is conjugate to the hyperplane joining $\pav{P}(i+1), \ldots, \pav{P}(i+d)$. This hyperplane contains the points $P_{ d}(i+1-d), \ldots, P_{ d}(i+d)$. So, these points are also conjugate to $D_d(i)$, which proves statement~5.
\qed	
}

\begin{remark}\label{rem:autoconjugateothercharacterisation}
    In (smooth) differential geometry, autoconjugate curves of a non-degenerate quadric $\mathcal{Q} \subset \RP^{2d}$ can be characterised as smooth curves such that the polar hyperplane and the osculating hyperplane coincide for each point of the curve. This characterisation is non-trivial to discretise because the osculating hyperplanes of any discrete curve $ \gamma \colon \Z \to \R^{2d}$ are naturally associated to the edges rather than the vertices. For example, if $d=1$, then the osculating hyperplanes are tangent lines of $\gamma$ which are associated to edges rather than vertices. However, in Theorem~\ref{th:diaginquadric} we have the pair $D_{d}, P_{d}$ of curves, which enables us to naturally define osculating hyperplanes as follows. Consider the combined curve given by
    \begin{align}
        \dots,\  P_d(i),\ D_d(i),\  P_d(i+1),\  D_d(i+1),\  \dots
    \end{align}
    which makes combinatorically sense since $D$ is the diagonal intersection net of $P$, so $D_d(i)$ lives between $P_d(i)$ and $P_d(i+1)$. Moreover, due to Theorem~\ref{th:diaginquadric}, the point $D_d(i)$ is actually contained in the line $P_d(i) \vee P_d(i+1)$. Thus, we can also think of the combined curve as the curve $P_d$ with marked points on the edges given by $D_d$. Now, we are able to define the osculating hyperplanes of $P$ using the points of $D$, and the osculating hyperplanes of $D$ using the points of $P$ in a symmetric fashion. Specifically, we define the osculating hyperplane at $D_{d}(k)$ to be the join 
    \begin{align}
        P_{d}(k+1-d) \vee \ldots \vee P_{d}(k+d),
    \end{align}
    and symmetrically we define the osculating hyperplane at $P_{d}(k)$ to be
    \begin{align}
        D_{d}(k-d) \vee \ldots \vee D_{d}(k+d-1).\label{eq:dosculating}
    \end{align}
    Using these definitions, Theorem~\ref{th:diaginquadric} shows that the osculating hyperplane at $D_d(k)$ is indeed the polar hyperplane of $D_{d}(k)$.
    We claim without proof that one can readily show that for each vertex $P_{d} (k)$ the osculating hyperplane \eqref{eq:dosculating} is the polar hyperplane of $P_{d}(k)$. Thus, the combined curve is autoconjugate in a sense that is analogous to the smooth setup in a more immediate sense. However, note that $D_d$ (so half of the combined curve) is an autoconjugate curve in the sense of Definition~\ref{defn:autoconjugate}, while $P_d$ (the other half) is not. Of course, the whole remark applies analogously to the pair $D_{-d}, P_{-d}$.
\end{remark}

\subsection{A bijection between K{\oe}nigs $d$-grids and pairs of autoconjugate curves}

In this subsection we establish our third main result, which is a bijection between generic K{\oe}nigs $d$-grids and generic pairs of autoconjugate curves. We begin with two lemmas that establish necessary properties of generic curves (in the sense of Definitions~\ref{def:genericcurve} and \ref{def:genericpair}).

\begin{lemma} \label{lem:justanothergenericitylemma}
	Let $\gamma \colon \Z \to \RP^{n}$ be a generic curve. Then, for all $j, k,\ell$ with $\ell \leq k+1 \leq n$ we have
	\begin{align}
		\cap^{\ell}_{b=0} C_{(k)}(j+b) = C_{(k-\ell)}(j+\ell),
	\end{align}
	which is $(k-\ell)$-dimensional.
\end{lemma}

\proof{
	Without loss of generality we assume $j=0$.
	The base case 
	\begin{align}
		C_{(k)}(0) \cap  C_{(k)}(1) = C_{(k -1)}(1),
	\end{align}
	follows immediately from the genericity assumption that
	\begin{align}
		C_{(k)}(0) \vee  C_{(k)}(1) = C_{(k + 1)}(0).
	\end{align}
	By induction on $\ell$, we get that 
	\begin{align}
		\cap^{\ell-1}_{b=0} C_{(k)}(b) = C_{(k-\ell+1)}(\ell-1), \qquad \cap^{\ell}_{b=1} C_{(k)}(b) = C_{(k-\ell+1)}(\ell).	
	\end{align}
	Using genericity again, we get 
	\begin{align}
		\cap^{\ell}_{b=0} C_{(k)}(b) = C_{(k-\ell+1)}(\ell-1) \cap C_{(k-\ell+1)}(\ell) = C_{(k-\ell)}(\ell).
	\end{align}
	\qed
}

Given a generic pair of autoconjugate curves $\sigma, \tau$ of a non-degenerate quadric in $\RP^{2d}$, they generate a sequence of maps
\begin{align}
P_k \colon \Z^2 &\to \RP^{2d}, \quad
	P_k(i,j) := S^\perp_{(d-k-1)}(j+k) \cap T^\perp_{(d+k-1)}(i), \label{eq:explicitpk}
\end{align}
for all $k$ such that $0 \leq k \leq d$. An immediate consequence of the genericity condition (polar version of Equation~\eqref{eq:autogendim}) is that each $P_{k}(i,j)$ is a point. Note that we use the convention that the polar of the empty set is the entire space $\RP^{2d}$. Then, \eqref{eq:explicitpk} also makes sense for $k=d$. In Theorem~\ref{thm:pairauutoconjugategivesgenericBSKoenigsgrid}, we will show that each $P_k$ is a Q-net and that $P_{k+1}$ is a Laplace transform of $P_{k}$. Then, the notation $P_k$ is justified.

\begin{lemma}\label{lem:autogenextensivity}
	Consider a generic pair of autoconjugate curves $\sigma, \tau \colon \Z \to \RP^{2d}$ and the generated sequence of maps $P_{k}$.	Then for all $i,j,k,l$ with $k \leq d$ and $l \leq d-k$,
\begin{align}\label{eq:dfnjjstj}
\bigvee^{l}_{b=0}P_k(i,j+b) = T^\perp_{(d+k-1)}(i) \cap S^\perp_{(d-k-l-1)}(j+k+l).
\end{align}
Moreover, $\dim \bigvee_{b=0}^l P_k(i,j+b) = l$.
\end{lemma}

\proof{
	Without loss of generality we assume $i=j=0$. The base case $l=0$ holds since each 
    \begin{align}
        P_k(0,0) = T^\perp_{(d+k-1)}(0) \cap S^\perp_{(d-k-1)}(k)    
    \end{align}
    is a point, as was explained earlier. By induction, we assume the case $l-1$ and prove the case $l>1$.	
	By our induction hypothesis, the joins
	\begin{align}\label{eq:rwhrwh}
		\bigvee^{l-1}_{b=0} P_k(0,b) \quad \text{ and } \qquad \bigvee^{l}_{b=1} P_k(0,b)
	\end{align}
	are $(l-1)$-dimensional.
	Using Equation~\eqref{eq:explicitpk} we have 
	\begin{align}
		\bigvee^{l-1}_{b=0}P_k(0,b) &= \bigvee^{l-1}_{b=0} \left( T^\perp_{(d+k-1)}(0) \cap S^\perp_{(d-k-1)}(b+k) \right)
		\subseteq T^\perp_{(d+k-1)}(0) \cap \left( \bigvee^{l-1}_{b=0} S^\perp_{(d-k-1)}(b+k)\right).\label{eq:thmoaergn}
	\end{align}

	By Lemma~\ref{lem:justanothergenericitylemma}, 
	\begin{align}
		\bigcap^{l-1}_{b=0} S_{(d-k-1)}(b+k) = S_{(d-k-l)}(k+l-1).
	\end{align}
	Equivalently,
	\begin{align}
		\bigvee^{l-1}_{b=0} S^\perp_{(d-k-1)}(b+k) = S^\perp_{(d-k-l)}(k+l-1).		
	\end{align}	
	Substituting this into Equation~\eqref{eq:thmoaergn} gives
	\begin{align}\label{eq:erhqerhqh}
		\bigvee^{l-1}_{b=0}P_k(0,b) = T^\perp_{(d+k-1)}(0) \cap S^\perp_{(d-k-l)}(k+l-1),
	\end{align}
	because our genericity assumption ensures that the right-hand side is $(l-1)$-dimensional, which is the same dimension as the subset on the left-hand side.
	Analogously,
	\begin{align}\label{eq:wrhbqrh}
		\bigvee^{l}_{b=1}P_k(0,b) = T^\perp_{(d+k-1)}(0) \cap S^\perp_{(d-k-l)}(k+l).
	\end{align}
	Looking for a contradiction, suppose that $\bigvee^{l}_{b=0}P_k(0,b)$ is not $l$-dimensional. Then -- using the induction assumption -- this join is $(l-1)$-dimensional and equals both spaces in \eqref{eq:rwhrwh}. Thus, from Equation~\eqref{eq:erhqerhqh} and Equation~\eqref{eq:wrhbqrh} we get
	\begin{align}
		\bigvee^{l}_{b=0}P_k(0,b) = T^\perp_{(d+k-1)}(0) \cap S^\perp_{(d-k-l)}(k+l-1) \cap S^\perp_{(d-k-l)}(k+l).
	\end{align}
	However, 
	\begin{align}
		S^\perp_{(d-k-l)}(k+l-1) \cap S^\perp_{(d-k-l)}(k+l) = S^\perp_{(d-k-l+1)}(k+l-1).
	\end{align}
	By our genericity assumption, 
	\begin{align}
		T_{(d+k-1)}(0) \vee S_{(d-k-l+1)}(k+l-1)	
	\end{align}
	is $(2d-l+1)$-dimensional since $l \geq 1$.
	 Equivalently, 
	\begin{align}
		T^\perp_{(d+k-1)}(0) \cap S^\perp_{(d-k-l+1)}(k+l-1)	
	\end{align}
	is $(l-2)$-dimensional. This leads to a contradiction because it cannot be both $(l-2)$-dimensional and $(l-1)$-dimensional. Therefore, $\bigvee^{l}_{b=0}P_k(0,b)$ is $l$-dimensional. From \eqref{eq:erhqerhqh}, we obtain \eqref{eq:dfnjjstj} by replacing $l-1$ with $l$. From \eqref{eq:dfnjjstj}, we see that the genericity assumption ensures that $\bigvee_{b=0}^l P_k(i,j+b)$ is $l$-dimensional.
	\qed
}

With these definitions, we are ready to give the bijection between autoconjugate curves and Kœnigs $d$-grids.

\begin{theorem}\label{thm:pairauutoconjugategivesgenericBSKoenigsgrid}
    Let $\sigma , \tau  \colon \Z \to \RP^{2d}$ be a generic pair of autoconjugate curves of a non-degenerate quadric $\mathcal{Q}$. Then, \begin{align} \label{eq:PijfromSandT}
    P(i,j) := S^\perp_{(d-1)}(j) \cap T^\perp_{(d-1)}(i)
    \end{align} determines a generic K{\oe}nigs $d$-grid $P \colon \Z^2 \to \RP^{2d}$. For $k \in \Sigma_{d}$, the Laplace transforms $P_{\pm k}$ exist and they are determined by the formulas
    \begin{align}
        P_{k}(i,j) &= S^\perp_{(d-k-1)}(j+k) \cap T^\perp_{(d+k-1)}(i), \label{eq:laplacetransormPk}\\
        P_{-k}(i,j) &= T^\perp_{(d-k-1)}(i+k) \cap S^\perp_{(d+k-1)}(j). \label{eq:laplacetransormPminusk}
    \end{align}
\end{theorem}

\begin{proof}
Throughout the proof, let us denote
\begin{align}
    A_k(i,j) &:= P_{k}(i,j) \vee P_k(i+1,j), & B_k(i,j) &:= P_{k}(i,j) \vee P_k(i,j+1).
\end{align}

The genericity assumption ensures that each $P_k(i,j)$ is a point. By Lemma~\ref{lem:autogenextensivity}, for all $k \in \Sigma_{d-1}$ the join $B_{k}(i,j)$ is $1$-dimensional and  
\begin{align}\label{eq:lineerheh}
    B_{k}(i,j) &=  S^\perp_{(d-k-2)}(j+k+1) \cap T^\perp_{(d+k-1)}(i),
\end{align}
and with an index shift we also get
\begin{align}\label{eq:linestwhra}
    B_{k}(i+1,j)  &=  S^\perp_{(d-k-2)}(j+k+1) \cap T^\perp_{(d+k-1)}(i+1).
\end{align}
Now, we want to check that $P_{k+1}$ as given in the right-hand side in \eqref{eq:laplacetransormPk} is indeed the Laplace transform of $P_k$.
For this,
we substitute \eqref{eq:lineerheh} and \eqref{eq:linestwhra} into the defining equation of the Laplace transform, which shows that $P_{k+1}(i,j)$ should be
\begin{align}
     B_k(i,j) \cap B_{k}(i+1,j)  &= S^\perp_{(d-k-2)}(j+k+1)  \cap  T^\perp_{(d+k-1)}(i) \cap T^\perp_{(d+k-1)}(i+1).\label{eq:wrshrhs}
\end{align}
From the genericity of $\tau$ we obtain that
\begin{align}
    T^\perp_{(d+k)}(i) = T^\perp_{(d+k-1)}(i) \cap T^\perp_{(d+k-1)}(i+1).
\end{align}
Substituting this into \eqref{eq:wrshrhs} shows that \eqref{eq:laplacetransormPk} is the correct formula for $P_{k+1}(i,j)$. The proof of \eqref{eq:laplacetransormPminusk} is analogous.

Since the two lines in \eqref{eq:wrshrhs} intersect in a point, it follows that $P_k$ is a Q-net. In particular, $P=P_0$ is a Q-net. Next, we show that $P$ is non-degenerate. This requires that no three of the points $P(i,j), P(i+1,j), P(i+1,j+1), P(i,j+1)$ are collinear.
It suffices to show that $P(i,j), P(i+1,j), P(i,j+1)$ are not collinear -- the other cases are analogous. Specializing \eqref{eq:lineerheh} to $k=0$ we obtain
\begin{align}
    B(i,j) = S^\perp_{(d-2)}(j+1) \cap T^\perp_{(d-1)}(i). \label{eq:rsihnrelg}
\end{align}
Symmetrically, we get 
\begin{align}
        A(i,j) = S^\perp_{(d-1)}(j) \cap T^\perp_{(d-2)}(i+1). \label{eq:woguhwerig}
\end{align}
We already showed that each $B(i,j)$ is $1$-dimensional. Symmetrically, each $A(i,j)$ is also $1$-dimensional. Now, looking for a contradiction, suppose $P(i,j), P(i+1,j), P(i,j+1)$ are collinear. Then, the lines $A(i,j)$ and $B(i,j)$ are equal. So,
\begin{align}\label{eq:thqeh}
    S^\perp_{(d-1)}(j) \cap T^\perp_{(d-2)}(i+1) = S^\perp_{(d-2)}(j+1) \cap T^\perp_{(d-1)}(i).
\end{align} 
Intersecting both sides with $S^\perp_{(d-1)}(j)$ and using $S^\perp_{(d-1)}(j) \subset S^\perp_{(d-2)}(j+1)$, we obtain
\begin{align}\label{eq:fnbsftn}
S^\perp_{(d-1)}(j) \cap T^\perp_{(d-2)}(i+1)= S^\perp_{(d-1)}(j) \cap T^\perp_{(d-1)}(i).
\end{align}
The left-hand side of \eqref{eq:fnbsftn} is $1$-dimensional. However, the genericity assumption implies that the right-hand side of \eqref{eq:fnbsftn} is $0$-dimensional. This is a contradiction. Therefore, $P$ is a non-degenerate Q-net. 

By Lemma~\ref{lem:autogenextensivity}, $P$ has parameter lines that are at least $d$-dimensional. Moreover, each $\pav{P}(i)$ is contained in $T^\perp_{(d-1)}(i)$ which is $d$-dimensional. Therefore, each $\pav{P}(i)$ is $d$-dimensional and equals $T^\perp_{(d-1)}(i)$. Analogously, each $\pah{P}(j)$ is $d$-dimensional and equals $S^\perp_{(d-1)}(j)$. Thus, $P$ is a $d$-grid.

The join $\pav{P}(i) \vee \pah{P}(j)$ is $2d$-dimensional because our genericity assumption ensures that
\begin{align}
    S_{(d-1)}(j) \vee T_{(d-1)}(i)
\end{align}
is $(2d-1)$-dimensional so that $S^\perp_{(d-1)}(j) \cap T^\perp_{(d-1)}(i)$ is $0$-dimensional. Equivalently, $S^\perp_{(d-1)}(j) \vee T^\perp_{(d-1)}(i)$ is $2d$-dimensional. Therefore, $P$ is a $\Sigma_{d,d}$-extensive $d$-grid.

To show that $P$ is a generic $d$-grid, we need to show that $P_{\pm d}$ are nowhere Laplace degenerate. We have already shown that
\begin{align}
    P_d(i,j) = T_{(2d-1)}^\perp (i), \quad P_d(i+1,j) = T_{(2d-1)}^\perp (i+1). 
\end{align}
We have $P_d(i,j) \neq P_d(i+1,j)$  because $T_{(2d-1)}(i) \neq T_{(2d-1)}(i+1)$ since the genericity assumption says that $T_{(2d)}(i)$ is $2d$-dimensional. Therefore, $P_d$ is nowhere Laplace degenerate. Symmetrically, $P_{-d}$ is also nowhere Laplace degenerate.

It only remains to show that $P$ is a Kœnigs net. As we observed above, $\pah{P}(j)$ is the $d$-dimensional space $S_{(d-1)}^{\perp}(j)$, which contains the isotropic space $S_{(d-1)}(j)$. Moreover, since $\mathcal Q$ is non-degenerate, $S_{(d-1)}(j)$ is the restriction of $\pah{P}(j)$ to $\mathcal Q$.
Now, let 
\begin{align}
    E(i,j):= P(i,j) \vee P(i+1,j) \vee P(i,j+1)  .  
\end{align}
This is a $2$-space since $P$ is a non-degenerate Q-net. Let $\mathcal{C}(i,j):= \mathcal{Q}\cap E(i,j)$. We claim this is an inscribed conic of the quad $P(i,j), P(i+1,j), P(i+1,j+1), P(i,j+1)$. To see this, consider
\begin{align}
    S(i,j) := A(i,j) \cap S_{(d-1)}(j) = A(i,j) \cap \mathcal Q.
\end{align}
This intersection is $0$-dimensional because $A(i,j)$ is a line in $\pah{P}(j)$ that is not contained in $\pah{S}(j)$ (since $P$ is a $\Sigma_{d,d}$-extensive $d$-grid). Therefore, $S(i,j)$ is the touching point of $\mathcal C(i,j)$ on the line $A(i,j)$. Analogously, the touching point on the line $B(i,j)$ is the intersection point of $B(i,j)$ with $T_{(d-1)}(i)$. Thus, the inscribed conics $\mathcal{C}$ are touching conics. Finally, Definition~\ref{thm:BSKoenigstouchingconics} shows that $P$ is a Kœnigs net.
\end{proof}

\begin{theorem}\label{thm:genericBSKoenigsdgridgivesgenericpairautoconjugate}
    Let $P \colon \Z^2 \rightarrow \RP^{2d}$ be a generic Kœnigs $d$-grid. Let $\mathcal{Q}$ be the special inscribed quadric. Define $\sigma, \tau \colon \Z \to \RP^{2d}$ such that 
    \begin{align}
    \sigma (j) &:= \pah{S}(j) \cap \ldots \cap \pah{S}(j+d-1),\label{eq:sigmaformula} \\
    \tau(i) &:= \pav{T}(i) \cap \ldots \cap \pav{T}(i+d-1).\label{eq:tauformula}
    \end{align}
    The two curves $\sigma,\tau$ are a pair of generic autoconjugate curves of $\mathcal{Q}$.
\end{theorem}

Before we give a proof of the theorem, let us point out that the pair of autoconjugate curves $\sigma, \tau$ coincides with the pair of Laplace transforms $D_{\pm d}$ of the diagonal intersection net $D$ of $P$. Indeed, comparing Equation~\eqref{eq:tauformula} from Theorem~\ref{thm:genericBSKoenigsdgridgivesgenericpairautoconjugate} with Equation~\eqref{eq:aksrbg}, we have that $D_{d}(i) = \tau(i+1)$. Analogously, $D_{-d}(j) = \sigma(j+1)$.

\begin{proof}[Proof of Theorem~\ref{thm:genericBSKoenigsdgridgivesgenericpairautoconjugate}]

From Corollary~\ref{cor:gridlargerspecial}, we obtain that $\mathcal Q$ is non-degenerate. Using that the polars of $\pah{S}(j)$ and $\pav{T}(i)$ are $\pah{P}(j)$ and $\pav{P}(i)$, respectively (by Lemma~\ref{lem:BSquadriclocalpatch}), we have that 
\begin{align}
    \sigma^\perp(j) = \pah{P}(j) \vee \ldots \vee \pah{P}(j+d-1), \label{eq:sigmafromp} \\ 
    \tau^\perp(i) = \pav{P}(i) \vee \ldots \vee \pav{P}(i+d-1),
\end{align}
which are both $(2d-1)$-dimensional since $P$ is $\Sigma_{d,d}$-extensive. Consequently, $\sigma(j)$ and $\tau(i)$ are points. Thus, $\sigma$ and $\tau$ are well defined.

Moreover, the osculating space
\begin{align}
    S_{(d-1)}(j) = \sigma(j) \vee \ldots \vee \sigma(j+d-1)
\end{align}
is contained in $\pah{S}(j+d-1)$, which -- by definition of the inscribed quadric $\mathcal Q$ -- is contained in $\mathcal{Q}$. So, each $S_{(d-1)}(j)$ is contained in $\mathcal{Q}$. Analogously, each $T_{(d-1)}(i)$ is contained in $\mathcal{Q}$. Therefore, $\sigma$ and $\tau$ are autoconjugate curves. It only remains to show that $\sigma$ and $\tau$ are a generic pair of curves.

Next, using induction on $k \in \Sigma_{d-1}$, we show that
\begin{align}\label{eq:rehewrh}
    \sigma^\perp(j) \cap \ldots \cap \sigma^\perp(j+k) = \pah{P}(j+k) \vee \ldots \vee \pah{P}(j+d-1),
\end{align}
which is then $(2d-k-1)$-dimensional by $\Sigma_{d,d}$-extensivity of $P$. The base case $k=0$ is \eqref{eq:sigmafromp}. By induction, 
\begin{align}
    \sigma^\perp(j) \cap \ldots \cap \sigma^\perp(j+k-1) &= \pah{P}(j+k-1) \vee \ldots \vee \pah{P}(j+d-1),\\
    \sigma^\perp(j+1) \cap \ldots \cap \sigma^\perp(j+k) &= \pah{P}(j+k) \vee \ldots \vee \pah{P}(j+d).
\end{align}
Hence, 

\begin{align}
    \sigma^\perp(j) \cap \ldots \cap \sigma^\perp(j+k) &= \left( \bigcap_{b=0}^{k-1} \sigma^\perp(j+b) \right) \bigcap  \left( \bigcap_{b=1}^{k} \sigma^\perp(j+b) \right)\\
     &= \left( \bigvee_{b=k-1}^{d-1} \pah P(j+b) \right) \bigcap \left(\bigvee_{b=k}^{d} \pah P(j+b) \right) \label{eq:ergkjengli}
\end{align}

By $\Sigma_{d,d}$-extensivity, $\pah{P}(j+k-1) \vee \ldots \vee \pah{P}(j+d-1)$ and $\pah{P}(j+k) \vee \ldots \vee \pah{P}(j+d)$ are distinct spaces that are both $(2d-k)$-dimensional and that are contained in the $(2d-k+1)$-dimensional space
\begin{align}
    \pah{P}(j+k-1) \vee \ldots \vee \pah{P}(j+d).
\end{align}
Therefore, their intersection \eqref{eq:ergkjengli} is the $(2d-k-1)$-space $\pah{P}(j+k) \vee \ldots \vee \pah{P}(j+d-1)$. This completes the proof of \eqref{eq:rehewrh}.

Next, we show that $\sigma$ (and by symmetric arguments $\tau$) is generic.
Note that the case $k=d-1$ of \eqref{eq:rehewrh} is
\begin{align} \label{eq:wiungrpieg}
    \pah{P}(j+d-1) = \sigma^\perp(j) \cap \ldots \cap \sigma^\perp(j+d-1).
\end{align}
By polarizing \eqref{eq:wiungrpieg} -- and using that $(\pah{P}(j+d-1))^\perp$ is $\pah{S}(j+d-1)$ due to Lemma~\ref{lem:BSquadriclocalpatch} --   we get that
\begin{align}
    S_{(d-1)}(j) := \sigma(j) \vee \ldots \vee \sigma(j+d-1)  =   \pah{S}(j+d-1),  \label{eq:blablibluzehn}
\end{align}
is $(d-1)$-dimensional since its polar is $d$-dimensional. Symmetrically, 
\begin{align}\label{eq:lirufgbwlirygb}
    T_{(d-1)}(i) := \tau(i) \vee \ldots \vee \tau(i+d-1) =    \pav{T}(i+d-1) 
\end{align}
is also $(d-1)$-dimensional.
Due to the $\Sigma_{d,d}$-extensivity of $P$, the dimension of 
\begin{align}
    \pah{P}(j+d-1) \cap \ldots \cap \pah{P}(j+d+k-1), \label{eq:thisspace}
\end{align}
is at least $d-k$ for any $k \in \Sigma_{d+1}$. Moreover, due to the genericity of $P$, (compare with Lemma~\ref{lem:expllaplacetransform}) the dimension is exactly $d-k$. The polar of \eqref{eq:thisspace} is 
\begin{align}
    \pah{S}(j+d-1) \vee \ldots \vee \pah{S}(j+d+k-1), \label{eq:rhwiqgubh}
\end{align}
which (using \eqref{eq:blablibluzehn}) is  
\begin{align}
    \sigma(j) \vee \ldots \vee \sigma(j+d+k-1) = S_{(d+k-1)}(j).
\end{align}
Since the dimension of \eqref{eq:thisspace} is $d-k$, the dimension of $S_{(d+k-1)}(j)$ is $d+k-1$.
By symmetric arguments, $T_{(d+k-1)}(i)$ is also $(d+k-1)$-dimensional. Therefore, we have shown that both $\sigma$ and $\tau$ are generic in the sense of Definition~\ref{def:genericcurve}. It only remains to show that the pair of curves is generic in the sense of Definition~\ref{def:genericpair}.

Since $P$ is a generic $d$-grid, its iterated Laplace transforms $P_{k}$ and $P_{-k}$ exist for all $k \in \Sigma_{d}$. By Lemma~\ref{lem:expllaplacetransform},
\begin{align}
    P_k(i,j) = \big( \pah{P}(j) \vee \ldots \vee \pah{P}(j+k)\big) \cap \big(  \pav{P}(i) \cap \ldots \cap \pav{P}(i+k) \big).
\end{align}
Polarising, we obtain 
\begin{align} \label{eq:rfhwethh}
    P_k^\perp(i,j) = \big( \pah{S}(j) \cap \ldots \cap \pah{S}(j+k)\big) \vee \big(  \pav{T}(i) \vee \ldots \vee \pav{T}(i+k) \big).
\end{align}
From \eqref{eq:blablibluzehn} we obtain (by an index shift) that $\pah{S}(j) = S_{(d-1)}(j-d+1)$. Combining this with Lemma~\ref{lem:justanothergenericitylemma}, which applies since $\sigma$ is generic, it follows that 
\begin{align}\label{eq:dfighgkyu}
    \pah{S}(j) \cap \ldots \cap \pah{S}(j+k) &= S_{(d-1)}(j-d+1) \cap \ldots \cap S_{(d-1)}(j-d+1+k) \\
    &= S_{(d-k-1)}(j-d+1+k).
\end{align}
Analogously, from \eqref{eq:lirufgbwlirygb} we obtain $\pav{T}(i) = T_{(d-1)}(i-d+1)$. Combining this with Lemma~\ref{lem:justanothergenericitylemma}, which applies since $\tau$ is generic, it follows that
\begin{align}\label{eq:sdgerwgh}
    \pav{T}(i) \vee \ldots \vee \pav{T}(i+k) &= T_{(d-1)}(i-d+1) \vee \ldots \vee T_{(d-1)}(i-d+1+k)\\
    &= T_{(d+k-1)}(i-d+1).
\end{align}
Inserting these two expressions into \eqref{eq:rfhwethh} we obtain 
\begin{align}\label{eq:lifbquqf}
    P_k^\perp (i,j) = S_{(d-k-1)}(j-d+1+k) \vee T_{(d+k-1)}(i-d+1).
\end{align}

Applying shifts to the indices in \eqref{eq:lifbquqf}, we obtain the following equations
\begin{align}\label{eq:lrigbqwkjg}
    S_{(d-k-1)}(j) \vee T_{(d+k-1)}(i) &= P^\perp_k(i+d-1, j+d-k-1), \\
    S_{(d-k-1)}(j) \vee T_{(d+k-1)}(i+1) &= P^\perp_k(i+d, j+d-k-1)\label{eq:asrkgnjr}
\end{align}
which are distinct hyperplanes because $P_k$ is nowhere Laplace degenerate since $P$ is a generic $d$-grid. Using that
\begin{align}
    T_{(d+k)}(i) = T_{(d+k-1)}(i) \vee T_{(d+k-1)}(i+1)
\end{align}
by the genericity of $\tau$, it follows that
\begin{align}\label{eq:iakhbglrehg}
    S_{(d-k-1)}(j) \vee T_{(d+k)}(i)
\end{align}
is $2d$-dimensional because it equals the join of the two distinct hyperplanes \eqref{eq:lrigbqwkjg} and \eqref{eq:asrkgnjr}. By symmetric arguments, using that $P_{-k}$ exists and is nowhere Laplace degenerate, we obtain that
\begin{align}\label{eq:rogjnqresef}
    S_{(d+k)}(j) \vee T_{(d-k-1)}(i)
\end{align} 
is $2d$-dimensional. Together, \eqref{eq:iakhbglrehg} and \eqref{eq:rogjnqresef} imply that 
\begin{align}\label{eq:sirdhgbrhgu}
    S_{(\ell)}(j) \vee T_{(2d-\ell-1)}(i)
\end{align} 
is $2d$-dimensional for all $\ell$ such that $-1 \leq \ell \leq 2d$. This shows that $\sigma, \tau$ are a generic pair of curves.
\end{proof}

Together, Theorem~\ref{thm:pairauutoconjugategivesgenericBSKoenigsgrid} and Theorem~\ref{thm:genericBSKoenigsdgridgivesgenericpairautoconjugate} prove our third main result (Theorem~\ref{thm:bijection}), that there is a bijection between generic autoconjugate curves and generic Kœnigs $d$-grids.

\section{Concluding Remarks}\label{sec:isothermic}

Let us finish by discussing a few directions of future research that we think are of interest.

\subsection{The family of inscribed quadrics}\label{section:conclusioninsrcibedquadrics}

Consider one quad of a Q-net. There is a 1-parameter family of inscribed conics. As already pointed out in Remark~\ref{rem:Koenigsconicssmoothdiscrete} it is not clear whether the $1$-parameter family of these conics has a counterpart in the smooth theory. Let us note though that the 1-parameter family of conics constitutes a dual pencil of conics.

More generally, let $P \colon \Sigma \to \RP^n$ be a $\Sigma_{a,b}$-extensive K{\oe}nigs net. We recall that by Theorem~\ref{thm:BSquadriclocalpatch}, it has a 1-parameter family of touching inscribed quadrics, each quadric is locally defined on a $\Sigma_{a,b}$ patch. One may ask the following question.

\begin{statement}
    For K{\oe}nigs nets, does the $1$-parameter family of touching inscribed quadrics have a counterpart in the smooth theory?
\end{statement}
It would also be interesting to understand whether there are further characterizations of the 1-parameter family of quadrics on the discrete level. For example, one could ask whether the family of quadrics is a dual pencil in the $\Sigma_{2,2}$ case.

\subsection{Isothermic nets}

Let us briefly discuss potential applications of our results to discrete isothermic surfaces. These nets are defined in \cite{bpdisosurfaces} as Kœnigs nets that are circular nets. By definition, a circular net is a discrete parametrized surface $\Z^2 \to \R^n$ such that, for each quad, the four points of the quad are contained in a circle. Circular nets with spherical parameter lines correspond to Q-nets in the M{\"o}bius quadric that have parameter spaces that are $3$-dimensional \cite{bobenkofairley2023circularnets}. If the circular net is isothermic, then the corresponding $Q$-net is also a K{\oe}nigs net. Therefore, circular nets with spherical parameter lines correspond to K{\oe}nigs $3$-grids that are inscribed in the M{\"o}bius quadric. Then, our results may have applications to discrete isothermic surfaces with two families of spherical parameter lines.  For instance, the combination of Theorem~\ref{th:bsgridincidence} and \cite[Proposition 4.23]{bobenkofairley2023circularnets} suggests that an analogue of Theorem~\ref{th:bsgridincidence} also holds for discrete isothermic surfaces with spherical curvature lines. However, the corresponding $3$-grids are typically not generic in the sense of Definition~\ref{def:genericgrid}. The parameter lines of such nets must be quite special curves. This leads to the following question:
\begin{statement}
	 Is there an efficient way to construct  all discrete isothermic nets with two families of spherical parameter lines?
\end{statement}

More generally, as mentioned in \cite{hoffmannSzewieczek2024isothermic}, it is an open problem to obtain a full description of isothermic nets with only one family of spherical parameter lines.

\subsection{Doubly periodic Kœnigs grids}

It may be interesting to understand the space of doubly periodic Kœnigs grids. In particular, using Theorem~\ref{thm:bijection} and the explicit form of the bijection given in Section~\ref{sec:autoconjugate}, it is clear that there is also a bijection between doubly periodic Kœnigs grids and pairs of periodic autoconjugate curves. Hence, one may ask the following question:
\begin{statement}
	What is the moduli space of discrete periodic autoconjugate curves of fixed length?
\end{statement}

\subsection{Kœnigs dual}
There is another characterization of Kœnigs nets in terms of the existence of a \emph{Kœnigs dual} \cite{BS2008DDGbook} -- which specializes to the \emph{Christoffel dual} for isothermic nets. The Kœnigs dual of a Kœnigs net is also a Kœnigs net, and the Kœnigs dual of a Kœnigs grid is also a Kœnigs grid. Hence, one may ask the following question:
\begin{statement}
	Is there a relation induced by the Kœnigs dual for inscribed quadrics and autoconjugate curves?
\end{statement}

\bibliographystyle{alpha}
\bibliography{main_laplacekoenigs}

\newcommand{\etalchar}[1]{$^{#1}$}
\begin{thebibliography}{BCHJ{\etalchar{+}}23}

\bibitem[Abr87]{Abresch1987}
Uwe Abresch.
\newblock Constant mean curvature tori in terms of elliptic functions.
\newblock {\em J. Reine Angew. Math.}, 374:169--192, 1987.

\bibitem[ADT25]{adtbinets}
Niklas~C. Affolter, Felix Dellinger, and Jan Techter.
\newblock {K{\oe}nigs binets}, 2025+.
\newblock In preparation.

\bibitem[AF25]{AffolterFairleyKoenigsLaplace}
Niklas~C. Affolter and Alexander~Y. Fairley.
\newblock {Discrete K{\oe}nigs nets and finite Laplace sequences}, 2025.
\newblock arXiv:2508.02851.

\bibitem[AGPR23]{agprvrc}
Niklas Affolter, Max Glick, Pavlo Pylyavskyy, and Sanjay Ramassamy.
\newblock Vector-relation configurations and plabic graphs.
\newblock {\em Selecta Mathematica}, 30(1):9, Dec 2023.

\bibitem[BCHJ{\etalchar{+}}23]{bchjpromega}
F.E. Burstall, J.~Cho, U.~Hertrich-Jeromin, M.~Pember, and W.~Rossman.
\newblock Discrete {$\Omega$}-nets and {G}uichard nets via discrete {K}oenigs
  nets.
\newblock {\em Proceedings of the London Mathematical Society},
  126(2):790--836, 2023.

\bibitem[Ber01]{Bernstein2001}
H.~Bernstein.
\newblock Non-special, non-canal isothermic tori with spherical lines of
  curvature.
\newblock {\em Trans. Amer. Math. Soc.}, 353(6):2245--2274, 2001.

\bibitem[BF21]{bobenkofairley2021nets}
Alexander~I. Bobenko and Alexander~Y. Fairley.
\newblock Nets of lines with the combinatorics of the square grid and with
  touching inscribed conics.
\newblock {\em Discrete Comput. Geom.}, 66(4):1382--1400, 2021.

\bibitem[BF25]{bobenkofairley2023circularnets}
Alexander~I. Bobenko and Alexander~Y. Fairley.
\newblock {Circular nets with spherical parameter lines and terminating Laplace
  sequences}.
\newblock {\em Discrete Comput. Geom.}, 2025.
\newblock to appear.

\bibitem[BH16]{bhsconical}
Alexander~I. Bobenko and Tim Hoffmann.
\newblock {\em S-Conical CMC Surfaces. Towards a Unified Theory of Discrete
  Surfaces with Constant Mean Curvature}, pages 287--308.
\newblock Springer Berlin Heidelberg, Berlin, Heidelberg, 2016.

\bibitem[BHS06]{bhsminimal}
Alexander~I. Bobenko, Tim Hoffmann, and Boris~A. Springborn.
\newblock Minimal surfaces from circle patterns: {G}eometry from combinatorics.
\newblock {\em Ann. Math. (2)}, 164(1):231--264, 2006.

\bibitem[BHSF23]{bobenko2023isothermictorifamilyplanar}
Alexander~I. Bobenko, Tim Hoffmann, and Andrew~O. Sageman-Furnas.
\newblock Isothermic tori with one family of planar curvature lines and area
  constrained hyperbolic elastica, 2023.
\newblock arXiv:2312.14956.

\bibitem[BHSF25]{bobenko2023compact}
Alexander~I. Bobenko, Tim Hoffmann, and Andrew~O. Sageman-Furnas.
\newblock {Compact Bonnet pairs: isometric tori with the same curvatures}.
\newblock {\em Publications math{\'e}matiques de l'IH{\'E}S}, Oct 2025.

\bibitem[Bob99]{bobenkocircular}
Alexander~I. Bobenko.
\newblock {\em Discrete conformal maps and surfaces}, page 97–108.
\newblock London Mathematical Society Lecture Note Series. Cambridge University
  Press, 1999.

\bibitem[BP96]{bpdisosurfaces}
Alexander~I. Bobenko and Ulrich Pinkall.
\newblock Discrete isothermic surfaces.
\newblock {\em Journal für die reine und angewandte Mathematik},
  1996(475):187--208, 1996.

\bibitem[BP99]{bpdiscsurfaces}
Alexander~I. Bobenko and Ulrich Pinkall.
\newblock {\em Discrete integrable geometry and physics}, chapter
  Discretization of surfaces and integrable systems, page 3–58.
\newblock Oxford University Press, 1999.

\bibitem[BPW10]{bpwcurvature}
Alexander~I. Bobenko, Helmut Pottmann, and Johannes Wallner.
\newblock A curvature theory for discrete surfaces based on mesh parallelity.
\newblock {\em Mathematische Annalen}, 348(1):1--24, Sep 2010.

\bibitem[BS08]{BS2008DDGbook}
Alexander~I. Bobenko and Yuri~B. Suris.
\newblock {\em Discrete Differential Geometry}, volume~98 of {\em Graduate
  Studies in Mathematics}.
\newblock American Mathematical Society, Providence, RI, 2008.

\bibitem[BS09]{BS2009Koenigsnets}
Alexander~I. Bobenko and Yuri~B. Suris.
\newblock Discrete {K}oenigs nets and discrete isothermic surfaces.
\newblock {\em Int. Math. Res. Not. IMRN}, 2009(11):1976--2012, 2009.

\bibitem[BSST16]{bsstconfocal}
Alexander~I. Bobenko, Wolfgang~K. Schief, Yuri~B. Suris, and Jan Techter.
\newblock On a discretization of confocal quadrics. {I}. {A}n integrable
  systems approach.
\newblock {\em Journal of Integrable Systems}, 1(1):xyw005, 08 2016.

\bibitem[CA14]{casas2014}
Eduardo Casas-Alvero.
\newblock {\em Analytic projective geometry}.
\newblock EMS Textbooks in Mathematics. European Mathematical Society (EMS),
  Z\"{u}rich, 2014.

\bibitem[CDS97]{cdscircular}
Jan Cieśliński, Adam Doliwa, and Paolo~Maria Santini.
\newblock The integrable discrete analogues of orthogonal coordinate systems
  are multi-dimensional circular lattices.
\newblock {\em Physics Letters A}, 235(5):480--488, 1997.

\bibitem[CPS23]{cpsconstrained}
Joseph Cho, Mason Pember, and Gudrun Szewieczek.
\newblock Constrained elastic curves and surfaces with spherical curvature
  lines.
\newblock {\em Indiana University Mathematics Journal}, 72:2059--2099, 01 2023.

\bibitem[Dar15]{darboux1915lecons}
G.~Darboux.
\newblock {\em Lecons sur la th{\'e}orie g{\'e}n{\'e}rale des surfaces et les
  applications g{\'e}om{\'e}triques du calcul infinit{\'e}simal: 2{\`e}me
  Partie}.
\newblock Gauthier-Villars, Paris, 2 edition, 1915.

\bibitem[Del24]{dellingercb}
Felix Dellinger.
\newblock Discrete isothermic nets based on checkerboard patterns.
\newblock {\em Discrete Comput. Geom.}, 72(1):209--245, Jul 2024.

\bibitem[Dol97]{doliwa1997geometricToda}
Adam Doliwa.
\newblock Geometric discretisation of the {T}oda system.
\newblock {\em Phys. Lett. A}, 234(3):187--192, 1997.

\bibitem[Dol03]{doliwa2003}
Adam Doliwa.
\newblock Geometric discretization of the {K}oenigs nets.
\newblock {\em J. Math. Phys.}, 44(5):2234--2249, 2003.

\bibitem[Dol07]{Doliwa2007}
Adam Doliwa.
\newblock The {B}-quadrilateral lattice, its transformations and the
  algebro-geometric construction.
\newblock {\em J. Geom. Phys.}, 57(4):1171--1192, 2007.

\bibitem[DS97]{dsmultidimconjugate}
Adam Doliwa and Paolo~Maria Santini.
\newblock Multidimensional quadrilateral lattices are integrable.
\newblock {\em Physics Letters A}, 233(4):365--372, 1997.

\bibitem[HS24]{hoffmannSzewieczek2024isothermic}
Tim Hoffmann and Gudrun Szewieczek.
\newblock Isothermic nets with spherical parameter lines from discrete
  holomorphic maps, 2024.
\newblock arXiv:2403.13476.

\bibitem[KMT23]{KMT2023conenets}
Martin Kilian, Christian M\"{u}ller, and Jonas Tervooren.
\newblock Smooth and {D}iscrete {C}one-{N}ets.
\newblock {\em Results Math.}, 78(3):Paper No. 110, 2023.

\bibitem[MP24]{mpagag}
Christian M{\"u}ller and Helmut Pottmann.
\newblock The geometry of discrete asymptotic-geodesic 4-webs in isotropic
  3-space.
\newblock {\em Monatshefte f{\"u}r Mathematik}, 203(1):223--246, Jan 2024.

\bibitem[Sau33]{sauer1933wackelige}
Robert Sauer.
\newblock Wackelige {K}urvennetze bei einer infinitesimalen
  {F}l\"{a}chenverbiegung.
\newblock {\em Math. Ann.}, 108(1):673--693, 1933.

\bibitem[Sau70]{sauer1970differenzen}
Robert Sauer.
\newblock {\em Differenzengeometrie}.
\newblock Springer, Berlin, 1970.

\bibitem[Tzi24]{Tzitzeica1924geometrie}
Georges Tzitz{\'e}ica.
\newblock {\em G{\'e}om{\'e}trie diff{\'e}rentielle projective des
  r{\'e}seaux}.
\newblock Acad{\'e}mie roumaine. {\'E}tudes et recherches. Cultura Nationala,
  Bucharest, 1924.

\bibitem[Wal87]{Walter1987}
Rolf Walter.
\newblock Explicit examples to the {$H$}-problem of {H}einz {H}opf.
\newblock {\em Geom. Dedicata}, 23(2):187--213, 1987.

\end{thebibliography}

\end{document}